\documentclass[10pt]{article}

\usepackage{amsfonts,a4wide,amsmath,amssymb,amsthm}
\usepackage[utf8]{inputenc}  
\usepackage[T1]{fontenc}
\usepackage[all]{xy}    
\usepackage{mathtools}  
\usepackage{multirow}    
\usepackage{csquotes}
\usepackage{tikz}   
\usepackage{float}
\usepackage{relsize}
\usepackage{hyperref}      
\hypersetup{
     colorlinks=true, 
     breaklinks=true, 
     urlcolor= blue,  
     linkcolor= black, 
     linktoc	=all,	
     citecolor=red,	
     filecolor=black,	
     bookmarksopen=true,            
     pdftitle={Tubular Runge theorem}, 
     pdfauthor={Samuel Le Fourn},     
}

\numberwithin{equation}{section}
\allowdisplaybreaks

\theoremstyle{plain}
\newtheorem{thm}{Theorem}
\newtheorem{prop}{Proposition}[section]
\newtheorem{cor}[prop]{Corollary}
\newtheorem*{thmsansnom}{Theorem}

\newtheorem{lem}[prop]{Lemma}

\theoremstyle{definition}
\newtheorem{defi}[prop]{Definition}
\newtheorem{defiprop}[prop]{Definition-Proposition}

\newtheorem{exe}[prop]{Example}

\theoremstyle{remark}
 \newtheorem{rem}[prop]{Remark}

\newcommand{\gp}{{\mathfrak{p}}}
\newcommand{\gP}{{\mathfrak{P}}}

\newcommand{\Acal}{{\mathcal A}}
\newcommand{\Bcal}{{\mathcal B}}
\newcommand{\Ccal}{{\mathcal C}}
\newcommand{\Dcal}{{\mathcal D}}
\newcommand{\Ecal}{{\mathcal E}}
\newcommand{\Fcal}{{\mathcal F}}
\newcommand{\Gcal}{{\mathcal G}}
\newcommand{\Hcal}{{\mathcal H}}
\newcommand{\Ical}{{\mathcal I}}

\newcommand{\Lcal}{{\mathcal L}}

\newcommand{\Ocal}{{\mathcal O}}

\newcommand{\Vcal}{{\mathcal V}}

\newcommand{\Xcal}{{\mathcal X}}
\newcommand{\Ycal}{{\mathcal Y}}

\newcommand{\N}{{\mathbb{N}}}
\newcommand{\Z}{{\mathbb{Z}}}
\newcommand{\Q}{{\mathbb{Q}}}
\newcommand{\R}{{\mathbb{R}}}
\newcommand{\C}{{\mathbb{C}}}
\renewcommand{\P}{\mathbb{P}}

\newcommand{\A}{{\mathbb{A}}}

\renewcommand{\div}{\operatorname{div}}

\newcommand{\Gal}{\operatorname{Gal}}
\newcommand{\GL}{\operatorname{GL}}

\renewcommand{\Im}{\operatorname{Im}}

\newcommand{\ord}{\operatorname{ord}}
\newcommand{\Proj}{\operatorname{Proj}}
\newcommand{\Spec}{\operatorname{Spec}}

\newcommand{\Hom}{\operatorname{Hom}}

\renewcommand{\Re}{\operatorname{Re}}

\newcommand{\Sp}{\operatorname{Sp}}
\newcommand{\Pic}{\operatorname{Pic}}
\newcommand{\carac}{\textrm{char}}
\newcommand{\Tr}{\operatorname{Tr}}

\newcommand{\lmt}{\longmapsto}

\newcommand{\lra}{\longrightarrow}

\newcommand{\Llra}{\Longleftrightarrow}

\newcommand{\fonction}[5]{\begin{array}{c|ccl}           
#1: & #2 & \longrightarrow & #3 \\
    & #4 & \longmapsto & #5 \end{array}}

\newcommand{\fonctionsansnom}[4]{\begin{array}{ccl}      
#1 & \lra & #2 \\\
#3 & \lmt & #4 
\end{array}}

\newcommand{\quot}[2]                                    
{\raisebox{.6ex}{\newline$#1$}\!/\!\raisebox{-.6ex}{$#2$}}

\newcommand{\Kb}{\overline{K}}
\newcommand{\Qb}{\overline{\Q}}

\title{\Huge A ``tubular'' variant of Runge's method in all dimensions, with applications to integral points on Siegel modular varieties}
\author{Samuel Le Fourn\footnote{Email : \url {samuel.le_fourn@ens-lyon.fr}  }
\\ ENS de Lyon}
\date\today

\begin{document}
\maketitle

\begin{abstract}
	Runge's method is a tool to figure out integral points on algebraic curves  effectively in terms of height. This method has been generalised to varieties of any dimension, unfortunately its conditions of application are often too restrictive. In this paper, we provide a further generalisation intended to be more flexible while still effective, and exemplify its applicability by giving finiteness results for integral points on some Siegel modular varieties. As a special case, we obtain an explicit finiteness result for integral points on the Siegel modular variety $A_2(2)$.
\end{abstract}

\section*{Introduction}
\addcontentsline{toc}{section}{Introduction}  

One of the major motivations of number theory is the description of rational or integral solutions of diophantine equations, which from a geometric perspective amounts to understanding the behaviour of rational or integral points on algebraic varieties. In dimension one, there are many techniques and results providing a good overview of the situation such as the famous Faltings' theorem (for genus $\geq 2$ and algebraic points) or Siegel's theorem (for integral points and a function with at least three poles). Nevertheless, in many cases the quest for effectivity (meaning a bound on the height on these points) is still open, and effective methods are quite different from these two powerful theoretical theorems.

We focus in this paper on a method for integral points on algebraic varieties called \textit{Runge's method}, and its generalisations and applications for Siegel modular varieties.

To keep the introduction fluid, we first explain the principles behind Runge's method and its applicatons to Siegel modular varieties, with simplified statements and a minimum of references and details. Afterwards, we describe precisely the structure of the article, in particular where the details we omitted first are given.

On a smooth algebraic projective curve $C$ over a number field $K$, Runge's method proceeds as follows. Let $\phi \in K(C)$ be a nonconstant rational function on $C$. For any finite extension $L/K$, we denote by $M_L$ the set of places of $L$ (and $M_L^\infty$ the archimedean ones). For $S_L$ a finite set of places of $L$ containing $M_L^\infty$, we denote the ring of $S_L$-integers of $\Ocal_L$ by
\[
	\Ocal_{L,S_L} = \{ x \in L \, \, |x|_v \leq 1 \, \,  {\textrm{for all }} \, v \in M_L \backslash S_L \}.
\]
Now, let $r_L$ be the number of orbits of poles of $\phi$ under the action of $\Gal(\overline{L}/L)$. The \textit{Runge condition} on a pair $(L,S_L)$ is the inequality 
\begin{equation}
\label{eqRungeconditioncourbes}
	|S_L|<r_L.
\end{equation}
Then, Bombieri's generalisation (\cite{BombieriGubler}, paragraph 9.6.5 and Theorem 9.6.6) of Runge's theorem, the latter being formulated only for $L=K=\Q$ and $r_\Q \geq 2$, states that for every pair $(L,S_L)$ satisfying Runge condition and every point $P \in C(L)$ such that $\phi(P) \in \Ocal_{L,S_L}$, there is an \textit{absolute} bound $B$ (only depending on $C$ and $\phi$, \textit{not} on such a pair $(L,S_L)$) such that 
\[
	h(\phi(P)) \leq B,
\]
where $h$ is the Weil height. In short, as long as the point $\phi(P)$ has few non-integrality places (the exact condition being \eqref{eqRungeconditioncourbes}), there is an absolute bound on the height of $\phi(P)$. There is a very natural justification (due to Bilu) for Bombieri's theorem: let us fix a pair $(L,S_L)$ satisfying Runge condition and $P \in C(L)$ such that $\phi(P) \in \Ocal_{L,S_L}$. For every place $v \in M_L \backslash S_L$, as $|\phi(P)|_v$ is small, it means that $P$ is $v$-adically far from all orbits of poles of $\phi$. For $v \in S_L$, $P$ can be $v$-adically close to one of the  orbits but only one of them because they are pairwise disjoint. We eliminate such an orbit if it exists, and applying the process for every $v \in S_L$, Runge's condition guarantees that there remains at the end of the process one orbit $\Ocal$ which is $v$-far from $P$ for \textit{all} places $v \in M_L$. This in turn implies finiteness : indeed, choosing by Riemann-Roch an auxiliary function $g_\Ocal \in L(C)$ whose poles are the points of $\Ocal$, this means that $h(g_\Ocal(P))$ is small as $P$ is far from its poles at every places, hence $P$ belongs to a finite set by Northcott condition. It is a bit more technical to obtain a bound on the height $h (\phi(P))$ (and which does not depend on $(L,S_L)$) in the general case) but it is the same idea. This justification also provides a method to bound in practice the heights of such points (when one knows well enough the auxiliary functions $g_\Ocal$), which is called \textit{Runge's method}. When applicable, this method has two important assets: it gives good bounds, and it is uniform in the pairs $(L,S_L)$, which for example is not true for Baker's method.

The goal of this paper was to find ways to transpose the ideas for Runge's method on curves to higher-dimensional varieties, where it is generally very difficult to obtain finiteness of integral or rational points, as the extent of our knowledge is much more limited. First, let us recall a previous generalisation of Bombieri's theorem in higher dimensions obtained by Levin (\cite{Levin08}, Theorem 4). To sum it up in a simpler case, on a projective smooth variety $X$, the analogues of poles of $\phi$ are effective divisors $D_1, \cdots, D_r$. We have to fix a smooth integral model $\Xcal$ of $X$ on $\Ocal_K$, and denote by $\Dcal_1, \cdots, \Dcal_r$ the Zariski closures of the divisors in this model, of union $\Dcal$, so our integral points here are the points of $(\Xcal \backslash \Dcal) (\Ocal_{L,S_L})$. There are two major changes in higher dimension. Firstly, the divisors have to be ample (or at least big) to obtain finiteness results (this was automatic for dimension 1). Secondly, instead of the condition $|S_L| < r$ as for curves, the \textit{higher-dimensional Runge condition} is
\begin{equation}
\label{eqintroRungemultidim}
m |S_L|<r,
\end{equation}
where $m$ is the smallest number such that any $(m+1)$ divisors amongst $D_1, \cdots, D_r$ have empty common intersection. Levin's theorem states in particular that when the divisors are ample, 
\[
	\left( \bigcup_{\substack{(L,S_L) \\ m |S_L|< r}} \! \! \left( \Xcal \backslash \Dcal \right) (\Ocal_{L,S_L}) \right) \, \, \, {\textrm{is finite}}.
\]
The issue with \eqref{eqintroRungemultidim} is that the maximal number $|S_L|$ satisfying this condition is much lowered because of $m$, even more as the ample (or big) hypothesis tends to give a lower bound on this $m$. When we tried to apply Levin's theorem to some Siegel modular varieties with chosen divisors, we found that the higher-dimensional Runge condition was too restrictive (remember that $S_L$ contains archimedean places, so $|S_L|\geq [K:\Q]/2$), hence the theorem was not applicable. This was the initial motivation for a generalisation of this theorem, called ``tubular Runge theorem'', designed to be more flexible in terms of Runge condition. Let us explain its principle below.

Additionally to $X$ and $D_1, \cdots, D_r$, we fix a closed subvariety $Y$ of $X$ which is meant to be ``a subvariety of $X$ where the divisors $D_1, \cdots, D_r$ intersect a lot more than outside it''. More precisely, let $m_Y$ the smallest number such that any $(m_Y+1)$ divisors amongst $D_1, \cdots, D_r$ have common intersection included in $Y$. In particular, $m_Y \leq m$, and the goal is to have $m_Y$ as small as possible without asking $Y$ to be too large. Now, we fix a ``tubular neighbourhood'' of $Y$, which is the datum of a family $\Vcal=(V_v)_v$ where $v$ goes through the places $v$ of $\Kb$, every $V_v$ is a neighbourhood of $Y$ in $v$-adic topology, and this family is uniformly not too small in some sense. For example, if $\Ycal$ is the Zariski closure of $Y$ in $\Xcal$, we can define at a finite place $v$ the neighbourhood $V_v$ to be the set of points of $\Xcal(\overline{K_v})$ reducing in $\Ycal$ modulo $v$. We say that a point $P \in X(\Kb)$ does \emph{not} belong to $\Vcal$ if $P \notin V_v$ for every place $v$ of $\Kb$, and intuitively, this means that $P$ is $v$-adically far away from $Y$ for \emph{every} place $v$ of $\Kb$. Now, assume our integral points are not in $\Vcal$. It implies that at most $m_Y$ divisors amongst $D_1, \cdots, D_r$ can be $v$-adically close to them,  hence using the same principles of proof as Levin, this gives the \textit{tubular Runge condition}
\begin{equation}
\label{eqintroRungetub}
	m_Y |S_L| < r.
\end{equation}
With this additional data, one can now give an idea of our tubular Runge theorem.

\begin{thmsansnom}[Simplified version of ``tubular Runge'' (Theorem \ref{thmRungetubulaire})]
\hspace*{\fill}

	For $X,\Xcal,Y,D_1, \cdots,D_r,m_Y$ and a tubular neighbourhood $\Vcal$ of $Y$ as in the paragraph above, let $(\Xcal \backslash \Dcal) (\Ocal_{L,S_L}) \backslash \Vcal$ be the set of points of $(\Xcal \backslash \Dcal)(\Ocal_{L,S_L})$ which do not belong to $\Vcal$. Then, if $D_1, \cdots, D_r$ are ample, for every such tubular neighbourhood, the set 
	\[
		\left( \bigcup_{\substack{(L,S_L) \\ m_Y |S_L|< r}} \! \! \left( \Xcal \backslash \Dcal \right) (\Ocal_{L,S_L}) \backslash \Vcal \right) \, \, \, {\textrm{is finite}},
	\]
	and bounded in terms of some auxiliary height.
\end{thmsansnom}

This is a very simplified form of the theorem : one can have $D_1, \cdots, D_r$ defined on a scalar extension of $X$ and big instead of ample, and $\Xcal$ normal for example. The general (and more precise) version is Theorem \ref{thmRungetubulaire}. As the implicit bound on the height is parametered by the tubular neighbourhood $\Vcal$, it can be seen as a \textit{concentration result} rather as a finiteness one : essentially, it states that the points of $(\Xcal \backslash \Dcal) (\Ocal_{L,S_L})$ concentrate near the closed subset $Y$. As such, we have compared it to theorems of \cite{CorvajaLevinZannier}, notably Autissier Theorem and CLZ Theorem, in section \ref{sectiontubularRunge} (in particular, our version is made to be effective, whereas these results are based on Schmidt's subspace theorem, hence theoretically ineffective). 

In the second part of our paper, we applied the method for Siegel modular varieties, both as a proof of principle and because integral points on these varieties are not very well understood, apart from Shafarevich conjecture proved by Faltings. As we will see below, this is also a case where a candidate for $Y$ presents itself, thus giving tubular neighbourhoods a natural interpretation.

For $n \geq 2$, the variety denoted by $A_2(n)$ is the variety over $\Q(\zeta_n)$  parametrising triples $(A,\lambda,\alpha_n)$ with $(A,\lambda)$ is a principally polarised abelian variety of dimension 2 and $\alpha_n$ is a symplectic level $n$ structure on $(A,\lambda)$. It is a quasi-projective algebraic variety of dimension 3, and its Satake compactification (which is a projective algebraic variety) is denoted by $A_2(n)^S$, the boundary being $\partial A_2(n) = A_2(n)^S \backslash A_2(n)$. The extension of scalars $A_2(n)_\C$ is the quotient of the half-superior Siegel space $\Hcal_2$ by the natural action of the symplectic congruence subgroup $\Gamma_2(n)$ of $\Sp_4(\Z)$ made up with the matrices congruent to the identity modulo $n$. Now, we consider some divisors ($n^4/2 +2$ of them) defined by the vanishing of some modular forms, specifically theta functions. One finds that they intersect a lot on the boundary $\partial A_2(n)$ ($m$ comparable to $n^4$), but when we fix $Y=\partial A_2(n)$, we get $m_Y \leq (n^2 - 3)$ hence giving the \textit{tubular Runge condition} 
\[
	(n^2 - 3) |S_L| < \frac{n^4}{2} + 2.
\]

Now, the application of our tubular Runge theorem gives for every even $n \geq 2$ a finiteness result for the integral points for these divisors and some tubular neighbourhoods associated to potentially bad reduction for the finite places : this is Theorem \ref{thmtubularRungegeneral}. In the special case $n=2$, as a demonstration of the effectiveness of the method, we made this result completely explicit in Theorem \ref{thmproduitCEexplicite}. A simplified case of this Theorem is the following result.
\begin{thmsansnom}
	[Theorem \ref{thmproduitCEexplicite}, simplified case]
	
	Let $K$ be either $\Q$ or a quadratic imaginary field.
	
	Let $A$ be a principally polarised abelian surface defined over $K$ as well as all its 2-torsion and having potentially good reduction at all finite places of $K$.
	
	Then, if the semistable reduction of $A$ is a product of elliptic curves at most at 3 finite places of $K$, we have the explicit bound 
	\[
	h_\Fcal(A) \leq 1070,
	\]
	where $h_\Fcal$ is the stable Faltings height. In particular, there are only finitely many such abelian surfaces.
\end{thmsansnom}

\setcounter{thm}{0}

To conclude this introduction, we explain the structure of the paper, emphasizing where the notions sketched above and proofs are given in detail.

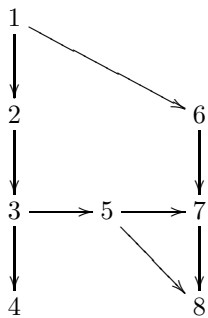
\[
	\xymatrix{
	\ref{sectionnotations} \ar[d] \ar[rrd] & & \\
 	\ref{sectionvoistub} \ar[d] & & \ref{sectionrappelsSiegel} \ar[d] \\
	\ref{sectionresultatscles} \ar[d] \ar[r] & \ref{sectiontubularRunge} \ar[rd] \ar[r]& \ref{sectionapplicationsSiegel} \ar[d] \\
	\ref{sectionRungecourbes} &  & \ref{sectionexplicitRunge}
	}
\]

Section \ref{sectionnotations} is devoted to the notations used throughout the paper, including heights, $M_K$-constants and bounded sets (Definition \ref{defMKconstante}). We advise the reader to pay particular attention to this first section as it introduces notations which are ubiquitous in the rest of the paper. Section \ref{sectionvoistub} is where the exact definition (Definition \ref{defvoistub}) and basic properties of tubular neighbourhoods are given. In section \ref{sectionresultatscles}, we prove the key result for Runge tubular theorem (Proposition \ref{propcle}), essentially relying on a well-applied Nullstellensatz. For our purposes, in Proposition \ref{propreductionamplegros}, we also translate scheme-theoretical integrality in terms of auxiliary functions. In section \ref{sectionRungecourbes}, we reprove Bombieri's theorem for curves (written as Proposition \ref{propBombieri}) with Bilu's idea, as it is not yet published to our knowledge (although this is exactly the principle behind Runge's method in \cite{BiluParent09} for example). To finish with the theoretical part, we prove and discuss our tubular Runge theorem (Theorem \ref{thmRungetubulaire}) in section \ref{sectiontubularRunge}.

For the applications to Siegel modular varieties, section \ref{sectionrappelsSiegel} gathers the necessary notations and reminders on these varieties (subsection \ref{subsecabvarSiegelmodvar}), their integral models with some discussions on the difficulties on dealing with them in dimension at least 2 (subsection \ref{subsecfurtherpropSiegelmodvar}) and the important notion of theta divisors on abelian varieties and their link with classical theta functions (subsection \ref{subsecthetadivabvar}). The theta functions are crucial because the divisors we use in our applications of tubular Runge method are precisely the divisors of zeroes of some of these theta functions.

In section \ref{sectionapplicationsSiegel}, we consider the case of abelian surfaces we are interested in, especially for the behaviour of theta divisors (subsection \ref{subsecthetadivabsur}) and state in subsection \ref{subsectubularRungethmabsur} the applications of Runge tubular theorem for the varieties $A_2(n)^S$ and the divisors mentioned above (Theorems \ref{thmtubularRungeproduitCE} and \ref{thmtubularRungegeneral}).

Finally, in section \ref{sectionexplicitRunge}, we make explicit Theorem \ref{thmtubularRungeproduitCE} by computations on the ten fourth powers of even characteristic theta constants. To do this, the places need to be split in three categories. The finite places not above 2 are treated by the theory of algebraic theta functions in subsection \ref{subsecalgebraicthetafunctions}, the archimedean places by estimates of Fourier expansions in subsection \ref{subsecarchimedeanplaces} and the finite places above 2 (the hardest case) by the theory of Igusa invariants and with polynomials built from our ten theta constants in subsection \ref{subsecplacesabove2}. The final estimates are given as Theorem \ref{thmproduitCEexplicite} in subsection \ref{subsecfinalresultRungeCEexplicite}, both in terms of a given embedding of $A_2(2)$ and in terms of Faltings height.

The main results of this paper have been announced in the recently published note \cite{LeFourn4}, and apart from section \ref{sectionexplicitRunge} and some improvements can be found in the author's thesis manuscript \cite{LeFournthese2} (both in French).
 
\section*{Acknowledgements}
I am very grateful to Fabien Pazuki and Qing Liu for having kindly answered my questions and given me useful bibliographic recommandations on the subject of Igusa invariants.

\begin{tableofcontents}
\end{tableofcontents}

\section{Notations and preliminary notions}
\label{sectionnotations}
The following notations are classical and given below for clarity. They will be used throughout the paper.

\begin{itemize}
\item[$\bullet$] $K$ is a number field.
\item[$\bullet$] $M_K$ (resp. $M_K^\infty$) is the set of places (resp. archimedean places). We also denote by $M_{\Kb}$ the set of places of $\Kb$.
\item[$\bullet$] $|\cdot|_\infty$ is the usual absolute value on $\Q$, and $|\cdot|_p$ is the place associated to $p$ prime, whose absolute value is normalised by 
\[
	|x|_p = p^{-\ord_p (x)},
\]
where $\ord_p (x)$ is the unique integer such that $x = p^{\ord_p(x)} a/b$ with  $p \nmid ab$. By convention, $|0|_p=0$.
\item[$\bullet$] $|\cdot|_v$ is the absolute value on $K$ associated to $v \in M_K$, normalised to extend $|\cdot|_{v_0}$ when $v$ is above $v_0 \in M_\Q$, and the local degree is $n_v = [K_v : \Q_{v_0}]$, so that for every $x \in K^*$, one has sthe product formula 
\[
	\prod_{v \in M_K} |x|_v^{n_v} = 1.
\]
When $v$ comes from a prime ideal $\gp$ of $\Ocal_K$, we indifferently write $|\cdot|_v$ and $|\cdot|_\gp$. 

\item[$\bullet$] For any place $v$ of $K$, one defines the sup norm on $K^{n+1}$ by 
\[
\| (x_0,  \cdots, x_n) \|_v = \max_{0 \leq i \leq n} |x_i|_v.
\]
(this will be used for projective coordinates of points of $\P^n (K)$).

\item[$\bullet$] Every set of places $S \subset M_K$ we consider is finite and contains $M_K^ {\infty}$. We then define the ring of $S$-integers as 
\[
	\Ocal_{K,S} = \{ x \in K \, \,  | \, \, |x|_v \leq 1 \textrm{ for every } v \in M_K \backslash S \},
\]
in particular $\Ocal_{K,M_K^{\infty}} = \Ocal_K$.

\item[$\bullet$] For every $P \in \P^n (K)$, we denote by 
\[
x_P=(x_{P,0}, \cdots, x_{P,n}) \in K^{n+1}
\] 
any possible choice of projective coordinates for $P$, this choice being of course fixed for consistency when used in a formula or a proof.

\item[$\bullet$]  The logarithmic Weil height of $P \in \P^n(K)$ is defined by
\begin{equation}
\label{eqdefinitionhauteurdeWeil}
	h(P) = \frac{1}{[K : \Q]}\sum_{v \in M_K} n_v \log \| x_P \|_v,
\end{equation}
does not depend on the choice of $x_P$ nor on the number field, and satisfies Northcott property.

\item[$\bullet$] For every $n \geq 1$ and every $i \in \{0, \cdots,n\}$, the $i$-th coordinate open subset $U_i$ of $\P^n$ is the affine subset defined as 
\begin{equation}
\label{eqdefUi}
	U_i = \{ (x_0 : \cdots : x_n) \, \, | \, \, x_i \neq 0 \}.
\end{equation}
The normalisation function $\varphi_i : U_i \rightarrow \A^{n+1}$ is then defined by 
\begin{equation}
\label{eqdefvarphii}
	\varphi_i (x_0 : \cdots : x_n) = \left(\frac{x_0}{x_i}, \cdots, 1, \cdots \frac{x_n}{x_i} \right).
\end{equation}
Equivalently, it means that to $P \in U_i$, we associate the choice of $x_P$ whose $i$-th coordinate is 1. 
\end{itemize}

For most of our results, we need to formalize the notion that some families of sets indexed by the places $v \in M_K$ are ``uniformly bounded''. To this end, we recall some classical definitions (see \cite{BombieriGubler}, section 2.6).

\begin{defi}[$M_K$-constants and $M_K$-bounded sets]
\hspace*{\fill}
\label{defMKconstante}

\begin{itemize}
	\item[$\bullet$] An \textit{$M_K$-constant} is a family $\Ccal = (c_v)_{v \in M_K}$ of real numbers such that $c_v=0$ except for a finite number of places $v \in M_K$.
The $M_K$-constants make up a cone of $\R^{M_K}$, stable by finite sum and maximum on each coordinate.

\item[$\bullet$] Let $L/K$ be a finite extension. For an $M_K$-constant $(c_v)_{v \in M_K}$, we define (with abuse of notation) an $M_L$-constant  $(c_w)_{w \in M_L}$ by $c_w : = c_v$ if $w|v$. Conversely, if $(c_w)_{w \in M_L}$ is an $M_L$-constant, we define (again with abuse of notation) $(c_v)_{v \in M_K}$ by $c_v := \max_{w|v} c_w$, and get in both cases the inequality
\begin{equation}
\label{eqineqinductionMKconstante}
	\frac{1}{[L : \Q]} \sum_{w \in M_L} n_w c_w \leq \frac{1}{[K: \Q]} \sum_{v \in M_K} n_v c_v.
\end{equation}

\item[$\bullet$] If $U$ is an affine variety over $K$ and $E \subset U(\Kb) \times M_{\Kb}$, a regular function $f \in \Kb[U]$ is \textit{$M_K$-bounded on $E$} if there is a $M_K$-constant $\Ccal = (c_v)_{v \in M_K}$ such that for every $(P,w) \in E$ with $w$ above $v$ in $M_K$,
\[
	\log |f(P)|_w \leq c_v.
\]
\item[$\bullet$] 
An \textit{$M_K$-bounded subset of $U$} is, by abuse of definition, a subset $E$ of $U(\Kb) \times M_{\Kb}$ such that every regular function $f \in \Kb[U]$ is $M_K$-bounded on $E$.
\end{itemize}
\end{defi}

\begin{rem}
\label{remdefMKconstantes}

There are fundamental examples to keep in mind when using these definitions: 

$(a)$ For every $x \in K^*$, the family $(\log |x|_v)_{v \in M_K}$ is an $M_K$-constant.

$(b)$ In the projective space $\P^n_K$, for every $i \in \{ 0 , \cdots, n\}$, consider the set 
\begin{equation}
\label{eqdefEi}
	E_i = \{ (P,w) \in \P^n(\Kb) \times M_{\Kb} \, \,  | \, \, |x_{P,i}|_w = \|x_P\|_w \}.
\end{equation}
The regular functions $x_j/x_i$ ($j \neq i$) on $\Kb[U_i]$ (notation \eqref{eqdefUi}) are trivially $M_K$-bounded (by the zero $M_K$-constant) on $E_i$, hence $E_i$ is $M_K$-bounded in $U_i$. Notice that the $E_i$ cover $\P^n (\Kb) \times M_{\Kb}$. We will also consider this set place by place, by defining for every $w \in M_{\Kb}$ :
\begin{equation}
	\label{eqdefEiw}
	E_{i,w} = \{ P \in \P^n(\Kb) \, \, |\, \, |x_{P,i}|_w = \|x_P\|_w \}.
\end{equation}

$(c)$ With notations \eqref{eqdefinitionhauteurdeWeil}, \eqref{eqdefUi} and \eqref{eqdefvarphii}, for a subset $E$ of $U_i(\Kb)$, if the coordinate functions of $U_i$ are $M_K$-bounded on $E \times M_{\Kb}$, the height $h \circ \varphi_i$ is straightforwardly bounded on $E$ in terms of the involved $M_K$-constants. This simple observation will be the basis of our finiteness arguments. 
\end{rem}

The following lemma is useful to split $M_K$-bounded sets in an affine cover.

\begin{lem}
\label{lemMKbornerecouvrement}
	Let $U$ be an affine variety and $E$ an $M_K$-bounded set. If $(U_j)_{j \in J}$ is a finite affine open cover of $U$, there exists a cover $(E_j)_{j \in J}$ of $E$ such that every $E_j$ is $M_K$-bounded in $U_j$.
\end{lem}

\begin{proof}
This is Lemma 2.2.10 together with Remark 2.6.12 of \cite{BombieriGubler}.
\end{proof}

Let us now recall some notions about integral points on schemes and varieties.

For a finite extension $L$ of $K$, a point $P \in \P^n(L)$ and a nonzero prime ideal $\gP$ of $\Ocal_L$ of residue field $k(\gP) = \Ocal_L/\gP$, the point $P$ extends to a unique morphism $\Spec \Ocal_{L,\gP} \rightarrow \P^n_{\Ocal_K}$, and the image of its special point is \textit{the reduction of $P$ modulo $\gP$}, denoted by $P_\gP \in \P^n (k(\gP))$. It is explicitly defined as follows : after normalisation of the coordinates $x_P$ of $P$ so that they all belong to $\Ocal_{L,\gP}$ and one of them to $\Ocal_{L,\gP}^*$, one has 
\begin{equation}
\label{eqreddansPn}
	P_\gP = (x_{P,0} \! \! \mod \gP  : \cdots : x_{P,n} \! \! \mod \gP) \in \P^n_{k(\gP)}.
\end{equation}

The following (easy) proposition expresses scheme-theoretic reduction in terms of functions (there will be another in Proposition \ref{propreductionamplegros}). We write it below as it is the inspiratoin behind the notion of tubular neighbourhood in section \ref{sectionvoistub}.

\begin{prop}
\label{proplienreductionpointssvaluation}
Let $S$ be a finite set of places of $K$ containing $M_K^\infty$, and $\Xcal$ be a projective scheme on $\Ocal_{K,S}$, seen as a closed subscheme of $\P^n_{\Ocal_{K,S}}$.

Let $\Ycal$ be a closed sub-$\Ocal_{K,S}$-scheme of $\Xcal$. 
	
	Consider $g_1, \cdots, g_s \in \Ocal_{K,S} [X_0, \cdots, X_n]$ homogeneous generators of the ideal of definition of $\Ycal$ in $\P^n_{\Ocal_{K,S_0}}$.  For every nonzero prime $\gP$ of $\Ocal_L$ not above $S$, every point $P  \in \Xcal(L)$, the reduction $P_\gP$ belongs to  $\Ycal_\gp (k(\gP))$ (with $\gp = \gP \cap \Ocal_K$) if and only if
	\begin{equation}
	\label{eqecplicitereduction2}
	\forall j \in \{1, \cdots, s \}, \quad  
	|g_j (x_P)|_\gP < \|x_P\|_\gP^{\deg g_j}.
	\end{equation}
\end{prop}

%

\begin{proof}
	For every $j \in \{1, \cdots, s\}$, by homogeneity of $g_j$, for a choice $x_P$ of coordinates for $P$ belonging to $\Ocal_{L,\gP}$ with one of them in $\Ocal_{L,\gP}^*$, the inequality \eqref{eqecplicitereduction2} amounts to 
	\[
	g_j(x_{P,0}, \cdots, x_{P,n}) = 0 \mod \gP
	\] .
	On another hand, the reduction of $P$ modulo $\gP$ belongs to $\Ycal_\gp (\overline{k(\gP)})$ if and only if its coordinates satisfy the equations defining $\Ycal_\gp$ in $X_\gp$, but these are exactly the equations $g_1, \cdots, g_s$ modulo $\gp$. This remark immediately gives the Proposition by \eqref{eqreddansPn}.
\end{proof}

\section{Definition and properties of tubular neighbourhoods}
\label{sectionvoistub}

The explicit expression  \eqref{eqecplicitereduction2} is the motivation for our definition of \textit{tubular neighbourhood}, at the core of our results. 
This definition is meant to be used by exclusion : with the same notations as Proposition \ref{proplienreductionpointssvaluation}, we want to say that a point $P \in X(L)$ is \textit{not} in some
tubular neighbourhood of $\Ycal$ if it \textit{never} reduces in $\Ycal$, whatever the prime ideal $\gP$ of $\Ocal_L$ is.

The main interest of this notion is that it provides us with a convenient alternative to this assumption for the places in $S$ (which are the places where the reduction is not well-defined, including the archimedean places), and also allows us to loosen up this reduction hypothesis in a nice fashion. Moreover, as the definition is function-theoretic, we only need to consider the varieties over a base field, keeping in mind that Proposition \ref{proplienreductionpointssvaluation} above makes the link with reduction at finite places.

\begin{defi}[Tubular neighbourhood]
\label{defvoistub}
\hspace*{\fill}

Let $X$ be a projective variety over $K$ and $Y$ be a closed $K$-subscheme of $X$.

 We choose an embedding $X \subset \P^n_K$, a set of homogeneous generators $g_1, \cdots, g_s$ in $K[X_0, \cdots, X_n]$ of the homogeneous ideal defining $Y$ in $\P^n$ and an $M_K$-constant $\Ccal = (c_v)_{v \in M_K}$.
 
 The \textit{tubular neighbourhood of $Y$ in $X$ associated to $\Ccal$ and $g_1, \cdots, g_s$} (the embedding made implicit) is the family 
 ${\Vcal = (V_w)_{w \in M_{\Kb}}}$ of subsets of $X(\Kb)$ defined as follows.

For every $w \in M_{\Kb}$ above some $v \in M_K$, $V_w$ is the set of points $P \in X(\Kb)$ such that
\begin{equation}
\label{eqdefvoistub}
	\forall j \in \{1, \cdots,s\}, \quad \log |g_j(x_P)|_w < \deg (g_j) \cdot \log \|x_P \|_w + c_v.
\end{equation}

\end{defi}

As we said before, this definition will be ultimately used by exclusion:

\begin{defi}
\label{defhorsdunvoistub}
\hspace*{\fill}

Let $X$ be a projective variety over $K$ and $Y$ be a closed $K$-subscheme of $X$.

For any tubular neighbourhood $\Vcal = (V_w)_{w \in M_{\Kb}}$ of $Y$, we say that a point $P \in X(\Kb)$ \textit{ does not belong to }$\Vcal$ (and we denote it by $P \notin \Vcal$) if 
\[
\forall w \in M_{\Kb}, \quad P \notin V_w.
\]
\end{defi}

\begin{rem}
\hspace*{\fill}
\label{remhorsvoistub}

$(a)$ Comparing \eqref{eqecplicitereduction2} and \eqref{eqdefvoistub}, it is obvious that for the $M_K$-constant $\Ccal=0$ and with the notations of Proposition \ref{proplienreductionpointssvaluation}, at the finite places $w$ not above $S$, the tubular neighbourhood $V_w$ is exactly the set of points $P \in X(\Kb)$ reducing in $\Ycal$ modulo $w$. Furthermore, instead of dealing with any homogeneous coordinates, one can if desired manipulate normalised coordinates, which makes the term $\deg(g_j) \log \|x_P\|_v$ disappear. Actually, we will do it multiple times in the proofs later, as it amounts to covering $\P^n_{\Kb}$ by the bounded sets $E_i$ (notation \eqref{eqdefEi}) and thus allows to consider affine subvarieties when needed.

$(b)$ In a topology, a set containing a neighbourhood is one as well : here, we will define everything by being out of a tubular neighbourhood, therefore allowing sets too large would be too restrictive. One can think about this definition as a family of neighbourhoods being one by one not too large but not too small, and uniformly so in the places.

$(c)$ If $Y$ is an ample divisor of $X$ and $\Vcal$ is a tubular neighbourhood of $Y$, one easily sees that if $P \notin \Vcal$ then $h(\psi(P))$ is bounded for some embedding $\psi$ associated to $Y$, from which we get the finiteness of the set of points $P$ of bounded degree outside of $\Vcal$. This illustrates why such an assumption is only really relevant when $Y$ is of small dimension.

$(d)$ A tubular neighbourhood of $Y$ can also be seen as a family of open subsets defined by bounding strictly a global arithmetic distance function to $Y$ (see \cite{Vojtadiophapp}, paragraph 2.5).
\end{rem}

\begin{exe}
We have drawn below three different pictures of tubular neighbourhoods at the usual archimedean norm. One consider $\P^2(\R)$ with coordinates $x,y,z$, the affine open subset $U_z$ defined by $z \neq 0$, and $E_x,E_y,E_z$ the respective sets such that $|x|,|y|,|z| = \max (|x|,|y|,|z|)$. These different tubular neighbourhoods are drawn in $U_z$, and the contribution of the different parts $E_x$, $E_y$ and $E_z$ is made clear.
\begin{figure}[H]
\centering
\resizebox{8cm}{8cm}
{
      \begin{tikzpicture}
	\draw[->] (-1,0) -- (9,0);
	\draw[->] (0,-1) -- (0,9);
	\fill[color=gray!20]
	 (0,6) -- (2,2)
	 -- (2,2) -- (6,0)
	 -- (6,0) -- (6,6)
	 -- (6,6) -- (0,6)
	 -- cycle;
	\draw (2,2) node[below left]{${(2,2)}$};
	\draw (2,2) node {${ \bullet}$};
	\draw (0,6) node[left]{$ (0,6)$};
	\draw (0,6) node {${\bullet}$};
	\draw (6,0) node[above right]{${ (6,0)}$};
	\draw (6,0) node {${ \bullet}$};
	\draw (6,6) node[right]{${ (6,6)}$};
	\draw (6,6) node {${ \bullet}$};
	\draw (3,3) node[right]{${\displaystyle P}$};
	\draw (3,3) node {$\bullet$};
	\draw[ultra thin] (2,2) -- (6,6);
	\draw (11/4,5) node[below] {${ E_y}$};
	\draw (5,5/2) node[below] {${ E_x}$};
    \end{tikzpicture}
}
\caption{Tubular neighbourhood of the point $P = (3:3:1)$ associated to the inequality 
$\max (|x-3y,y-3z|) < \frac{1}{2} \max(|x|,|y|,|z|).$
}
	\end{figure}
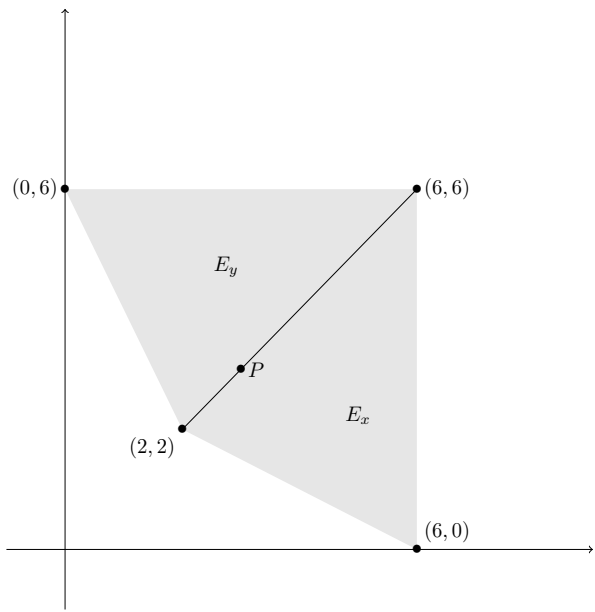
	
\begin{figure}[H]
	\centering 
\resizebox{8cm}{8cm}
{
	\begin{tikzpicture}
	\draw[->, very thin] (-6,0) -- (6,0);
	\draw[->, very thin] (0,-6) -- (0,6);
	\fill[color=gray!20, opacity=0.8]
	 (-1/2,1) -- (4,4)
	 -- (4,4) -- (6,6)
	 -- (6,6) -- (1,6)
	 -- (1,6) -- (-4/3,4/3)
	 (-4/3,4/3) -- (-6,-1)
	 -- (-6,-1) -- (-6,-6)
	 -- (-6,-6) -- (-4,-4)
	 -- (-4,-4) -- (-1,1/2)
	 -- (-1,1/2) -- (-1/2,1)
	 -- cycle;
	 \draw[thick] (-6,-4) -- (4,6);
	 \draw (-4,-4) node {${ \bullet}$};
	 \draw (-4,-4) node[below right] {${ \scriptstyle (-4,-4)}$};
	 \draw (-4/3,4/3) node[above left] {${\scriptstyle (-4/3,4/3)}$};
	  \draw (4,4) node[right] {${ \scriptstyle (4,4)}$};
	 \draw (-4/3,4/3) node {${ \bullet}$};
	 \draw (4,4) node {${ \bullet}$};
	 \draw (-1,1/2) node {${ \bullet}$};
	 \draw (-1,1/2) node[below] {${\scriptstyle (-1,1/2)}$};
	  \draw (-1/2,1) node {${ \bullet}$};
	  \draw (-1/2,1) node[right] {${\scriptstyle (-1/2,1)}$};
	  \draw (2,5) node {\textbf{\textit{L}}};
	  \draw (-3/4,3/4) node {${ E_z}$};
	  \draw[ultra thin] (-1,1) -- (-1/2,1);
	  \draw[ultra thin] (-1,1/2) -- (-1,1);
	  \draw[ultra thin] (-1,1) -- (-4/3,4/3);
	  \draw (1,4) node {${ E_y}$};
	  \draw (3,4) node {${ E_y}$};
	  \draw (-4,-1) node {${ E_x}$};
	  \draw (-4,-3) node {${E_x}$};  
	 \end{tikzpicture}
}	
\caption{Tubular neighbourhood of the line $D : y - x + 2z = 0$ associated to the inequality 
$\max (|x-y+2z|) < \frac{1}{2} \max(|x|,|y|,|z|)$.} 
The boundary of the neighbourhood is made up with segments between the indicated points
\end{figure}
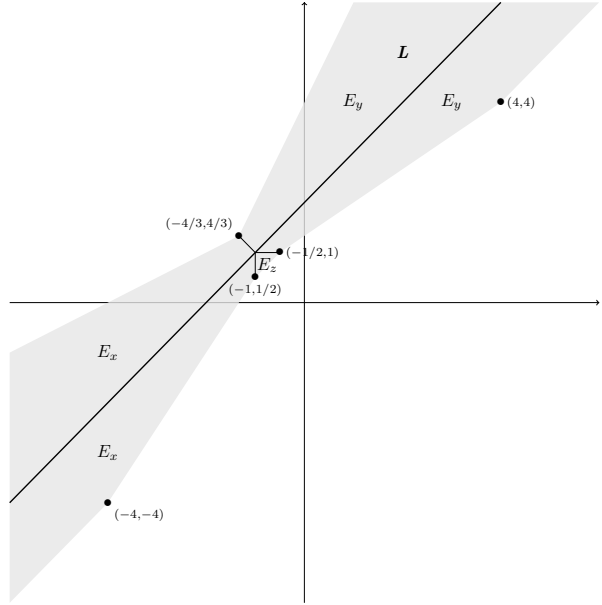

\begin{figure}[H]
\centering
\resizebox{8cm}{8cm}
{
  \begin{tikzpicture}
	\draw[dashed, ultra thin] (-3,-6) -- (3,6);
	\draw[dashed, ultra thin] (-3,6) -- (3,-6);
	\draw[dashed, ultra thin] (-6,-3) -- (6,3);
	\draw[dashed, ultra thin] (-6,3) -- (6,-3);
	 \draw[->, very thin] (-6,0) -- (6,0);
	\draw[->, very thin] (0,-6) -- (0,6);
	\draw[domain=7/6:3/2] plot (\x, {-\x + sqrt(\x +3)});
	\fill[color=gray!30, opacity=0.8]
	plot[domain=-17/6:1/2] (\x, {-\x + sqrt(\x*\x +2)})
	-- plot[domain=1/2:1] (\x,{1/(2*\x)})
	-- plot[domain=1:6] (\x, {-\x/2 + 1/\x})
	-- (6,-5/2) -- (6,7/2)
	-- plot [domain=6:sqrt(2)] (\x,{\x/2 + 1/\x})
	-- plot [domain=sqrt(2):19/6] (\x,{\x + sqrt(\x*\x - 2)})
	-- (19/6,6) -- (-17/6,6)
	-- cycle;
	\fill[color=gray!30, opacity=0.8]
	plot[domain=-6:-1] (\x, {-\x/2 + 1/\x})
	-- plot[domain=-1:-1/2] (\x,{1/(2*\x)})
	-- plot[domain=-1/2:17/6] (\x, {-\x - sqrt(\x*\x +2)})
	-- (17/6,-6) -- (-19/6,-6)
	-- plot [domain=-19/6:{-sqrt(2)}] (\x,{\x - sqrt(\x*\x - 2)})
	-- plot [domain={-sqrt(2)}:-6] (\x,{\x/2 + 1/\x})
	-- (-6,-19/6) -- (-6,17/6)
	-- cycle;
	\draw[domain=1/6:6,samples=300, thick]  plot (\x,{1/\x}); 
	\draw[domain=-6:-1/6,samples=300, thick]  plot (\x,{1/\x}); 
	\draw (1,5) node {\textit{\textbf{H}}};
	\draw[thin] (1/2,1) -- (1,1);
	\draw[thin] (1,1) -- (1,1/2);
	\draw[thin] (-1,-1) -- (-1,-1/2);
	\draw[thin] (-1,-1) -- (-1/2,-1);
	\draw (-3/4,-3/4) node {${ E_z}$};
	\draw (3/4,3/4) node {${ E_z}$};
	\draw[thin] (1,1) -- ({sqrt(2)},{sqrt(2)});
	\draw[thin] (-1,-1) -- ({-sqrt(2)},{-sqrt(2)});
	\draw (-7/2,1/2) node {${ E_x}$};
	\draw (-7/2,-1) node {${ E_x}$};
	\draw (-1,-4) node {${ E_y}$};
	\draw (1,-4) node {${ E_y}$};
	\draw (1,4) node {${ E_y}$};
	\draw (-1,4) node {${ E_y}$};
	\draw (7/2,-1/2) node {${ E_x}$};
	\draw (7/2,1) node {${ E_x}$};
	\draw ({-sqrt(2)},{-sqrt(2)}) node {$\bullet$};
	\draw ({-sqrt(2)},{-sqrt(2)}) node[below left] {${\scriptscriptstyle - (\sqrt{2},\sqrt{2})}$};
	\draw ({sqrt(2)},{sqrt(2)}) node {$\bullet$};
	\draw ({sqrt(2)},{sqrt(2)}) node[above right] {${\scriptscriptstyle (\sqrt{2},\sqrt{2})}$};
	\draw (1/2,1) node {$\bullet$};
	\draw (1/2,1) node[below left] {${\scriptscriptstyle (1/2,1)}$};
	\draw (1,1/2) node {$\bullet$};
	\draw (1,1/2) node[below] {${\scriptscriptstyle (1,1/2)}$};
	\draw (-1/2,-1) node {$\bullet$};
	\draw (-1/2,-1) node[above right] {${\scriptscriptstyle (-1/2,-1)}$};
	\draw (-1,-1/2) node {$\bullet$};
	\draw (-1,-1/2) node[above] {${\scriptscriptstyle (-1,-1/2)}$};
    \end{tikzpicture}
}
\caption{Tubular neighbourhood of the hyperbola $H : xy - z^2 = 0$ given by the inequality 
$|xy-z^2| < \frac{1}{2} \max (|x|,|y|,|z|)$.} The boundary is made up with arcs of hyperbola between the indicated points.
\end{figure}
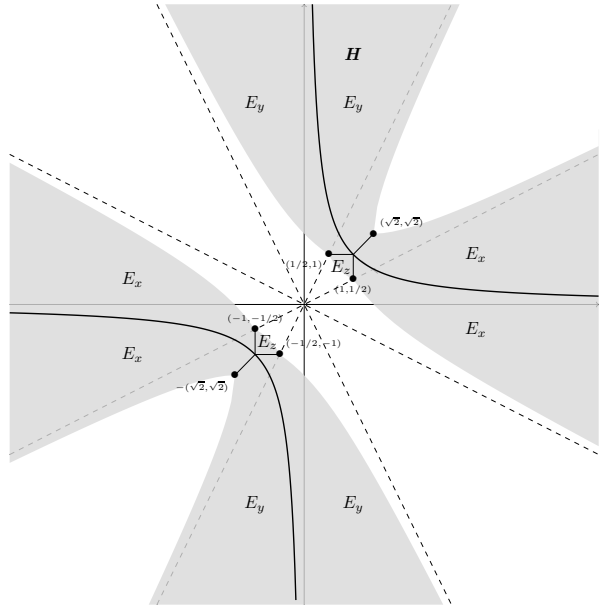
	 
\end{exe}

The notion of tubular neighbourhood does not seem very intrinsic, but as the proposition below shows, it actually is.

\begin{prop}[Characterisation of tubular neighbourhoods]
\label{propcarvoistub}

Let $X$ be a projective variety over $K$ and $Y$ a closed $K$-subscheme of $X$. 

A family $\Vcal=(V_w)_{w \in M_{\Kb}}$ is included in a tubular neighbourhood of $Y$ in $X$  if and only if for every affine open subset $U$ of $X$, every $E \subset U(\Kb) \times M_{\Kb}$ which is $M_K$-bounded in $U$, and every regular function $f \in \Kb[U]$ such that $f_{|Y \cap U} = 0$, there is an $M_K$-constant $\Ccal$ such that 
\[
	\forall (P,w) \in E, \quad P \in V_w \Rightarrow \log |f(P)|_w < c_v
\]
(intuitively, this means that every function vanishing on $Y$ is ``$M_K$-small'' on $\Vcal$).
	
\end{prop}

\begin{rem}
	One can also give a criterion for containing a tubular neighbourhood (using generators in $\Kb[U]$ of the ideal defining $Y \cap U$). Together, these imply that the tubular neighbourhoods made up by an embedding of $X$ are essentially the same. Indeed, one can prove that for two different projective embeddings of $X$, a tubular neighbourhood as defined by the first one can be an intermediary between two tubular neighbourhoods as defined by the second embedding.
\end{rem}

\begin{proof}
	First, a family $\Vcal$ satisfying this property is included in a tubular neighbourhood. Indeed, if we choose $g_1, \cdots, g_s$ homogeneous generators of the ideal defining $Y$ for some embedding of $X$ in $\P^n_K$, for every $i \in \{0, \cdots, n\}$, consider (using notations \eqref{eqdefUi}, \eqref{eqdefvarphii} and \eqref{eqdefEi}) the $M_K$-bounded set $E_i$ and the regular functions $g_j \circ \varphi_i$ on $U_i$, $1 \leq j \leq s$. By hypothesis, (taking the maximum of all the $M_K$-constants for $0 \leq i \leq n, 1 \leq j \leq s$), there is an $M_K$-constant $(c_v)_{v \in M_K}$ such that for every $w \in M_{\Kb}$,
	\[
		\forall j \in \{1, \cdots, s \}, \forall i \in \{0, \cdots, n\}, \forall P \in E_{i,w}, \textrm{  if } \, P \in V_w,  \quad \log |g_j \circ \varphi_i(P)|_w < c_v
	\]
because $g_j \circ \varphi_i = 0$ on $Y \cap U_i$ by construction and the $\varphi_i(P)$ are normalised coordinates for $P \in E_{i,w}$. Hence, $\Vcal$ is included in the tubular neighbourhood of $Y$ in $X$ associated to $\Ccal$ and the generators $g_1, \cdots, g_s$.

It now remains to prove that any tubular neighbourhood of $Y$ satisfies this characterisation, and we will do so (with the same notations as Definition \ref{defvoistub}) for the tubular neighbourhood defined by a given embedding $X \subset \P^n_K$, homogeneous equations $g_1, \cdots, g_s$ defining $Y$ in $\P^n_K$ and some $M_K$-constant $\Ccal_0 = (c_{0,v})_{v \in M_K}$ (we will use multiple $M_K$-constants, hence the numbering).

Let us fix an affine open subset $U$ of $X$ and $E$ an $M_K$-bounded set on $U$. We can cover $U$ by principal affine open subsets of $X$, more precisely we can write 
\[
	U = \bigcup_{h  \in \Fcal} U_h
\]
where $h$ runs through a finite family $\Fcal$ of nonzero homogeneous polynomials of $\Kb[X_0, \cdots, X_n]$ and 
\[
	U_h = \{ P \in X \, | \, h(P) \neq 0 \}.
\]
For every such $h$, the regular functions on $U_h$ are the $s/h^k$ where $s$ is homogeneous on $\Kb[X_0, \cdots, X_n]$ of degree $k \cdot \deg(h)$ (as $X$ is a closed subvariety of $\P^n$, the only subtlety is that identical regular functions on $U_h$ can come from different fractions $s/h^k$ but this will not matter in the following).

By Lemma \ref{lemMKbornerecouvrement}, there is a cover $E = \cup_{h \in \Fcal} E_h$ such that every $E_h$ is $M_K$-bounded on $U_h$. This implies that for any $i \in \{0, \cdots, n\}$, the functions $x_i^{\deg (h)} / h \in \Kb[U_f]$ are $M_K$-bounded on $E_h$, therefore we have an $M_K$-constant $\Ccal_{1}$ such that for all $(P,w) \in E_h$ with coordinate $x_0, \cdots, x_n$, 
\begin{equation}
\label{eqfoncintercoordonnees}
	\log \| x_P \|_w \leq c_{1,v} + \frac{1}{\deg(h)} \log |h(x_P)|_w.
\end{equation}
Now, let $f$ be a regular function on $\Kb[U]$ such that $f_{|Y \cap U} = 0$. For every $h \in \Fcal$, we can write $f_{|U_h}=s/h^k$ for some homogeneous $s \in \Kb[X_0, \cdots, X_n]$, therefore as a homogeneous function on $X$, one has $h \cdot s = 0$ on $Y$ (it already cancels on $Y \cap U$, and outside $U$ by multiplication by $f$). Hence, we can write 
\[
	f_{|U_h} = \sum_{j=1}^s  \frac{a_{j,h} g_{j}}{h^{k_j}}
\]
with the $a_{j,h}$ homogeneous on $\Kb[X_0, \cdots, X_n]$ of degree $k_j \deg(h) - \deg(g_j)$. Now, bounding the coefficients of all the $a_{j,h}$ (and the number of monomials in the archimedean case), we get an $M_K$-constant $\Ccal_2$ such that for every $P \in \P^n(\Kb)$, 
\[
	\log |a_{j,h} (x_P)|_w \leq c_{2,v} + \deg(a_{j,h}) \cdot \log \|x_P \|_w.
\]
Combining this inequality with \eqref{eqecplicitereduction2} and \eqref{eqfoncintercoordonnees}, we get that for every $h \in \Fcal$, every $(P,w) \in E_h$ and every $j \in \{1, \cdots, s\}$ :
\[
\textrm{ if } P \in V_w, \quad \log \left|  \frac{a_{j,h} g_j}{h^{k_j}}(P) \right|_w < c_{0,v} + c_{2,v} + k_j c_{1,v}
\]
which after summation on $j \in \{1, \cdots, s \}$ and choice of $h$ such that $(P,w) \in E_h$ proves the result.

\end{proof}

\section{Key results}
\label{sectionresultatscles}

We will now prove the key result for Runge's method, as a consequence of the Nullstellensatz. We mainly use the projective case in the rest of the paper  but the affine case is both necessary for its proof and enlightening for the method we use.

\begin{prop}[Key proposition]
\hspace*{\fill}
\label{propcle}

$(a)$ (Affine version)

Let $U$ be an affine variety over $K$ and $Y_1, \cdots, Y_r$ closed subsets of $U$ defined over $K$, of intersection $Y$. For every $\ell \in \{1, \cdots, r\}$, define $g_{\ell,1}, \cdots g_{\ell,s_\ell}$ generators of the ideal of definition of $Y_\ell$ in $K[U]$, and $h_1, \cdots, h_s$ generators of the ideal of definition of $Y$ in $K[U]$.
For every $M_K$-bounded set $E$ of $U$ and every $M_K$-constant $\Ccal_0$, there is an $M_K$-constant $\Ccal$ such that for every $(P,w) \in E$ with $w$ above $v \in M_K$, one has the following dichotomy : 
\begin{equation}
\label{eqdichoaff}
	\max_{\substack{1 \leq \ell \leq r \\ 1 \leq j \leq s_i }} \log |g_{\ell,j} (P)|_w \geq c_v \quad \textrm{or} \quad \max_{1 \leq j \leq s} \log |h_j (P)|_w < c_{0,v}. 
\end{equation}

$(b)$ (Projective version)

Let $X$ be a normal projective variety over $K$ and $\phi_1, \cdots, \phi_r \in K(X)$. Let $Y$ be the closed subset of $X$ defined as the intersection of the supports of the (Weil) divisors of poles of the $\phi_i$. For every tubular neighbourhood $\Vcal$ of $Y$ (Definition \ref{defvoistub}), there is an $M_K$-constant $\Ccal$ depending on $\Vcal$ such that for every $w \in M_{\Kb}$ (above $v \in M_K)$ and every $P \in X(\Kb)$, 
\begin{equation}
\label{eqdichoproj}
\min_{1 \leq \ell \leq r} \log |\phi_\ell (P)|_w \leq c_v \quad \textrm{or} \quad P \in V_w.
\end{equation}
	
\end{prop}
This result has an immediate corollary when $Y=\emptyset$: Lemma 5 of \cite{Levin08}, restated below.

\begin{cor}[\cite{Levin08}, Lemma 5]
\hspace*{\fill}
\label{corpasdepolecommun}

	Let $X$ be a normal projective variety over $K$ and $\phi_1, \cdots, \phi_r \in K(X)$ having globally no common pole. Then, there is an $M_K$-constant $\Ccal$ such that for every $w \in M_{\Kb}$ (above $v \in M_K)$ and every $P \in X(\Kb)$, 
	\begin{equation}
		\label{eqpasdepolecommun}
		\min_{1 \leq \ell \leq r} \log |\phi_\ell (P|_w \leq c_v.
	\end{equation}
\end{cor}

\begin{rem}
\hspace*{\fill}

$(a)$ As will become clear in the proof, part $(b)$ is actually part $(a)$ applied to a good cover of $X$ by $M_K$-bounded subsets of affine open subsets of $X$ (inspired by the natural example of Remark \ref{remdefMKconstantes} $(b)$).

$(b)$ Besides the fact that the results must be uniform in the places (hence the $M_K$-constants), the principle of $(a)$ and $(b)$ is simple. For $(a)$, we would like to say that if a point $P$ is sufficiently close to $Y_1, \cdots, Y_r$ (i.e. the first part of the dichotomy is not satisfied) it must be close to a point of intersection of the $Y_i$, hence the generators of the intersection should be small at $P$ (second part of the dichotomy). This is not true in the affine case, taking for example the hyperbola and the real axis in $\A^2$, infinitely close but disjoint (hence the necessity of taking a bounded set $E$ to compactify the situation), but it works in the projective case because the closed sets are then compact.

$(c)$ Corollary \ref{corpasdepolecommun} is the key for Runge's method in the case of curves in section \ref{sectionRungecourbes}. Notice that Lemma 5 of \cite{Levin08} assumed $X$ smooth, but the proof is actually exactly the same for $X$ normal. Moreover, the argument below follows the structure of Levin's proof.

$(d)$ If we replace $Y$ by $Y' \supset Y$ and $\Vcal$ by a tubular neighbourhood $\Vcal'$ of $Y'$, the result remains true with the same proof, which is not surprising because tubular neighbourhood of $Y'$ are larger than tubular neighbourhoods of $Y$.
\end{rem}

\begin{proof}[Proof of Proposition \ref{propcle}]
\hspace*{\fill}

	$(a)$ By the Nullstellensatz applied on $K[U]$ to the $Y_\ell \quad (1 \leq \ell \leq r)$ and $Y$, by hypothesis, for some power $p \in \N_{>0}$, there are regular functions $f_{\ell,j,m} \in K[U]$ such that for every $m \in \{1, \cdots, s\}$, 
	\[
		\sum_{\substack{1 \leq \ell \leq r \\ 1 \leq j \leq s_\ell}} g_{\ell,j} f_{\ell,j,m} = h_m^p.
	\]
As $E$ is $M_K$-bounded on $U$, all the $f_{\ell,j,m}$ are $M_K$-bounded on $E$ hence there is an auxiliary $M_K$-constant $\Ccal_1$ such that for all $P \in E$, 
\[
	\max_{\substack{1 \leq  \ell \leq r \\ 1 \leq j \leq s_\ell \\ 1 \leq m \leq s}} \log |f_{\ell,j,m} (P)|_w \leq c_{1,v},
\]
therefore
\[
	|h_m(P)^p|_w = \left| \sum_{\substack{1 \leq \ell \leq r \\ 1 \leq j \leq s_\ell}} g_{\ell,j} (P) f_{\ell,j,m} (P) \right|_w \leq N^{\delta_v} e^{c_{1,v}} \max_{\substack{1 \leq \ell \leq r \\ 1 \leq j \leq s_\ell }} |g_{\ell,j} (P)|_w
\]
where $\delta_v$ is 1 if $v$ is archimedean and 0 otherwise, and $N$ the total number of generators $g_{\ell,j}$. For fixed $w$ and $P$, either $\log |h_m(P)|_w < c_{0,v}$ for all $m \in \{ 1, \cdots, s\}$ (second part of dichotomy \eqref{eqdichoaff}), or the above inequality applied to some $m \in \{1, \cdots,s\}$ gives 
\[
	p \cdot c_{0,v} \leq  \delta_v \log(N) + c_{1,v} +  \max_{\substack{1 \leq \ell \leq r \\ 1 \leq j \leq s_\ell }} \log |g_{\ell,j} (P)|_w,
\]
which is equivalent to  
\[
	\max_{\substack{1 \leq \ell \leq r \\ 1 \leq j \leq s_\ell }} \log |g_{\ell,j} (P)|_w \geq \delta_v \log(N) + c_{1,v} - p \cdot c_{0,v},
\]
and taking the $M_K$-constant defined by $c_v := c_{1,v} + \delta_v \log(N) - p \cdot c_{0,v}$ for every $v \in M_K$ gives exactly the first part of dichotomy \eqref{eqdichoaff}.

$(b)$ We consider $X$ as embedded in some $\P^n_K$ so that $\Vcal$ is exactly the tubular neighbourhood of $Y$ in $X$ associated to an $M_K$-constant $\Ccal_0$ and generators $g_1, \cdots, g_s$ for this embedding. We will use again the notations \eqref{eqdefUi}, \eqref{eqdefvarphii} and \eqref{eqdefEi}. In particular we define $X_i := X \cap U_i$ for every $i \in \{0, \cdots, n\}$. The following argument is designed to make $Y$ appear as a common zero locus of regular functions built with the $\phi_\ell$.

For every $\ell \in \{1, \cdots, r\}$, let $D_\ell$ be the positive Weil divisor of zeroes of $\phi_\ell$ on $X$. For every $i \in \{0, \cdots, n\}$, let $I_{\ell,i}$ be the ideal of $K[X_i]$ made up with the regular functions $h$ on the affine variety $X_i$ such that $\div(h) \geq (D_\ell)_{|X_i}$, and we choose generators $h_{\ell,i,1}, \cdots, h_{\ell,i,j_{\ell,i}}$ of this ideal. The functions $h_{\ell,i,j}/(\phi_\ell)_{|X_i}$ are then regular on $X_i$ and 
\[
	\forall j \in \{1, \cdots, j_{\ell,i} \}, \quad \div \left( \frac{h_{\ell,i,j}}{(\phi_\ell)_{|X_i}} \right) \geq (\phi_{\ell,i})_\infty 
\]
(the divisor of poles of $\phi_\ell$ on $X_i$). By construction of $I_{\ell,i}$, the minimum (prime Weil divisor by prime Weil divisor) of the $\div(h_{\ell,i,j})$ is exactly $(D_\ell)_{|X_i}$ : indeed, for every finite family of distinct prime Weil divisors $D'_1, \cdots, D'_s, D''$ on $X_i$, there is a uniformizer $h$ for $D''$ of order 0 for each of the $D'_k$, otherwise the prime ideal associated to $D''$ in $X_i$ would be included in the finite union of the others. This allows to build for every prime divisor $D'$ of $X_i$ not in the support of $(D_\ell)_{|X_i}$ a function $h \in I_{\ell,i}$ of order $0$ along $D'$ (and of the good order for every $D'$ in the support of $(D_\ell)_{|X_i}$. Consequently, the minimum of the divisors of the $h_{\ell,i,j} / (\phi_{\ell})_{|X_i}$, being naturally the minimum of the divisors of the $h / (\phi_{\ell})_{|X_i} \, \, ( h \in K[X_i])$, is exactly $(\phi_{\ell,i})_\infty$. 

Thus, by definition of $Y$, for fixed $i$, the set of commmon zeroes of the regular functions $h_{\ell,i,j} / (\phi_\ell)_{|X_i} \, (1 \leq \ell \leq r, 1 \leq j \leq j_{\ell,i})$ on $X_i$ is $Y \cap X_i$, so they generate a power of the ideal of definition of $Y \cap X_i$. We apply part $(a)$ of this Proposition to the $h_{\ell,i,j} / (\phi_\ell)_{|X_i} \, (1 \leq \ell \leq r, 1 \leq j \leq j_{\ell,i})$, the $g_j \circ \varphi_i \, (1 \leq j \leq s)$ and the $M_K$-constant $\Ccal_0$, which gives us an $M_K$-constant $\Ccal'_i$ and the following dichotomy on $X_i$ for every $(P,w) \in E_i$ : 
\[
\max_{\substack{1 \leq \ell \leq r \\ 1 \leq j \leq s_i }} \log \left| \frac{h_{\ell,i,j}}{\phi_\ell} (P) \right|_w \geq c'_{i,v} \quad \textrm{or} \quad \max_{1 \leq j \leq s} \log |g_j \circ \varphi_i (P)|_w < c_{0,v}. 	
\]
Now, the $h_{\ell,i,j}$ are regular on $X_i$ hence $M_K$-bounded on $E_i$, therefore there is a second $M_K$-constant $\Ccal''_i$ such that for every $(P,w) \in E_i$ :
\[
	\max_{\substack{1 \leq \ell \leq r \\ 1 \leq j \leq s_i }} \log \left| \frac{h_{\ell,i,j}}{\phi_\ell} (P) \right|_w \geq c'_{i,v} \Longrightarrow \min_{1 \leq \ell \leq r} \log |\phi_\ell (P)|_w \leq c''_{i,v}.
\]
Taking $\Ccal$ as the maximum of the $M_K$-constants $\Ccal''_i, 0 \leq i \leq n$, for every $(P,w) \in X(\Kb) \times M_{\Kb}$, we choose $i$ such that $(P,w) \in E_i$ and then we have the dichotomy \eqref{eqdichoproj} by definition of the tubular neighbourhood $V_w$.
\end{proof}

To finish this section, we will give the explicit link between integral points on a projective scheme (relatively to a divisor) and integral points relatively to rational functions on the scheme. In particular, this catches up with the definition of integral points of section 2 of \cite{Levin08}.

\begin{prop}
\hspace*{\fill}
\label{propreductionamplegros}

	Let $\Xcal$ be a normal projective scheme over $\Ocal_{K,S}$.
	
	$(a)$ If $\Ycal$ is an effective Cartier divisor on $\Xcal$ such that $\Ycal_K$ is an ample (Cartier) divisor of $\Xcal_K$, there is a projective embedding $\psi : \Xcal_K \rightarrow \P^n_K$ and an $M_K$-constant $\Ccal$ such that
	\begin{itemize}
	\item[$\bullet$] The pullback by $\psi$ of the hyperplane of equation $x_0=0$ in $\P^n_K$ is $\Ycal_K$.
	
	\item[$\bullet$] For any finite extension $L$ of $K$ and any $w \in M_L$ not above $S$, 

\begin{equation}
	\label{eqlienreducschemasfonctionsample}
	\forall P \in (\Xcal \backslash \Ycal) (\Ocal_{L,w}),  \quad \log \|x_{\psi(P)} \|_w \leq c_v + \log |x_{\psi(P),0}|_w  .
\end{equation}
This amounts to say that if the coordinates by $\psi$ of such a $P$ are normalised so that the first one is 1, all the other ones have $w$-norm bounded by $e^{c_v}$.
\end{itemize}

$(b)$ If $\Ycal$ is an effective Cartier divisor on $\Xcal$ such that $\Ycal_K$ is a big (Cartier) divisor of $\Xcal_K$, there is a strict Zariski closed subset $Z_K$ of $\Xcal_K$, a morphism $\psi : \Xcal_K \backslash \Ycal_K \rightarrow \P^n_K$ which induces a closed immersion of $\Xcal_K \backslash Z_K$ and an $M_K$-constant $\Ccal$ such that:
\begin{itemize}
	\item[$\bullet$] The pullback by $\psi$ of the hyperplane of equation $x_0=0$ in $\P^n_K$ is contained in $\Ycal_K \cup Z_K$.
	
	\item[$\bullet$] For any finite extension $L$ of $K$ and any $w \in M_L$ not above $S$, formula \eqref{eqlienreducschemasfonctionsample} holds.
\end{itemize}

\end{prop}

\begin{rem}
\hspace*{\fill}
\label{rempropamplegros}

$(a)$ This Proposition is formulated to avoid the use of local heights, but the idea is exactly that under the hypotheses above, the fact that $P \in (\Xcal \backslash \Ycal) (\Ocal_{L,w})$ implies that the local height at $w$ of $P$ for the divisor $\Ycal$ is bounded. 

$(b)$ The hypotheses on ampleness (or ``bigness'') are only necessary at the generic fiber. If we considered $\Ycal$ ample on $\Xcal$, it would give us a result with the zero $M_K$-constant (using an embedding over $\Ocal_{K,S}$ given by $\Ycal$), and an equivalence, but this is not crucial here. Once again, the auxiliary functions replace the need for a complete understanding of what happens at the finite places.

$(c)$ The only difference between ample and big cases is hidden in the function $\psi$ : in the big case, the formula still holds but does not say much for points belonging in $Z_K$ because the morphism $\psi$ is not an embedding there.
\end{rem}

\begin{proof}[Proof of Proposition \ref{propreductionamplegros}]
\hspace*{\fill}

$(a)$ As $\Ycal_K$ is ample and effective, there is a projective embedding $\psi : \Xcal_K \rightarrow \P^n$ such that the support of the divisor $\Ycal_K$ is exactly the inverse image of the hyperplane $x_0=0$ by $\psi$. Let us fix such an embedding and consider for every $i \in \{1, \cdots, n\}$ the coordinate functions $\phi_i := (x_i/x_0) \circ \psi$ in $ K(\Xcal_K)$, whose poles are contained in $\Ycal_K$ by construction. Now, we choose a tubular neighbourhood $\Vcal$ of $\Ycal_K$ defined by an embedding of $\Xcal$ in some projective $\P^m_{\Ocal_{K,S}}$ (which can be completely unrelated to $\psi$), homogeneous generators $g_1, \cdots, g_s$ of the ideal of definition of $\Ycal$ in $\P^m_{\Ocal_{K,S}}$ and the zero $M_K$-constant. By Proposition \ref{propcle} $(b)$ applied to $\Vcal$ and $\phi_j$, we obtain an $M_K$-constant $\Ccal_j$ such that for every finite extension $L$ of $K$ and every $w \in M_L$ (with the notations \eqref{eqdefvarphii} and \eqref{eqdefEiw}),
\[
	\forall P \in \Xcal(L) , \qquad \log |\phi_j(P)|_w  \leq c_{j,v} \quad {\textrm{or}} \quad P \in V_w.
\]
By construction of $\Vcal$ and Proposition \ref{proplienreductionpointssvaluation}, if $w$ is not above a place of $S$ and $P \in (\Xcal \backslash \Ycal)(\Ocal_{L,w})$, we necessarily have $\log |\phi_j(P)|_w \leq c_{j,v}$. Taking the maximum of the $M_K$-constants $\Ccal_1, \cdots, \Ccal_n$, we obtain the Proposition in the ample case.

$(b)$ The proof for big divisors is the same as part $(a)$, except that we can only extend our function $\psi$ to $\Xcal_K \backslash Z_K$ for some proper Zariski closed subset $Z_K$ such that outside of this set, $\psi$ is a closed immersion. The coordinate functions $\phi_i \in K(\Xcal_K)$, similarly defined, also have poles contained in $\Ycal_K$. Applying the same arguments as in part $(a)$ for points $P \in \Xcal(L)$, we obtain the same result.
\end{proof}

\section{The case of curves revisited}
\label{sectionRungecourbes}

In this section, we reprove the generalisation of an old Runge theorem \cite{Runge1887} obtained by Bombieri (\cite{BombieridecompWeil} p. 305, also rewritten as Theorem 9.6.6 in \cite{BombieriGubler}), following an idea exposed by Bilu in an unpublished note and mentioned for the case $K=\Q$ by \cite{SchoofCatalan} (Chapter 5). The aim of this section is therefore to give a general understanding of this idea (quite different from the original proof of Bombieri), as well as explain how it actually gives a \textit{method} to bound heights of integral points on curves. 

It is also a good start to understand how the intuition behind this result can be generalised to higher dimension, which will be done in the next section.

\begin{prop}[Bombieri, 1983]
\label{propBombieri}
Let $C$ be a smooth projective algebraic curve defined over a number field $K$ and $\phi \in K(C)$ not constant.

For any finite extension $L/K$, let $r_L$ be the number of orbits of the natural action of $\Gal(\overline{L}/L)$ over the poles of $\phi$. For any set of places $S_L$ of $L$ containing $M_L^{\infty}$, we say that $(L,S_L)$ satisfies the \textbf{Runge condition} if 
\begin{equation}
\label{eqconditionRunge}
	|S_L|<r_L.
\end{equation}

Then, the reunion
\begin{equation}
\label{eqreunionpointsRungecondition}
	\bigcup_{\substack{(L,S_L)}} \left\{ P \in C(L) \, | \, \phi(P) \in \Ocal_{L,S_L} \right\},
\end{equation}
where $(L,S_L)$ runs through all the pairs satisfying Runge condition, is \textbf{finite} and can be explicitly bounded in terms of the height $h \circ \phi$.
\end{prop}

\begin{exe}
As a concrete example, consider the modular curve $X_0(p)$ for $p$ prime and the $j$-invariant function. This curve is defined over $\Q$ and $j$ has two rational poles (which are the cusps of $X_0(p)$), hence $r_L=2$ for any choice of $L$, and we need to ensure $|M_L^{\infty}| \leq |S_L|< 2$. The only possibilities satisfying Runge condition are thus imaginary quadratic fields $L$ with $S_L = \{ | \cdot |_{\infty} \}$.

We thus proved in \cite{LeFourn1} that for any imaginary quadratic field $L$ and any $P \in X_0(p)(L)$ such that $j(P) \in \Ocal_L$, one has 
	\[
		\log |j(P)| \leq 2 \pi \sqrt{p} + 6 \log (p) + 8.
	\]
The method for general modular curves is carried out in \cite{BiluParent09} and gives explicit estimates on the height for integral points satisfying Runge condition. This article uses the theory of modular units and implicitly the same proof of Bombieri's result as the one we expose below.
\end{exe}

\begin{rem}
\hspace*{\fill}
\label{remRungecourbes}

$(a)$ The claim of an explicit bound deserves a clarification : it can actually be made explicit when one knows well enough the auxiliary functions involved in the proof below (which is possible in many cases, e.g. for modular curves thanks to the modular units). Furthermore, even as the theoretical proof makes use of $M_K$-constants and results of section \ref{sectionresultatscles}, they are frequently implicit in pratical cases.
	
$(b)$ Despite the convoluted formulation of the proof below and the many auxiliary functions to obtain the full result, its principle is as descrbibed in the Introduction. It also gives the framework to apply Runge's method to a given couple $(C,\phi)$
\end{rem}

\begin{proof}[Proof of Proposition \ref{propBombieri}]
	We fix $K'$ a finite Galois extension of $K$ on which every pole of $\phi$ is defined. For any two distinct poles $Q,Q'$ of $\phi$, we choose by Riemann-Roch theorem a function $g_{Q,Q'} \in K'(C)$ whose only pole is $Q$ and vanishing at $Q'$. For every point $P$ of $C(\Kb)$ which is not a pole of $\phi$, one has $\ord_P (g_{Q,Q'}) \geq 0$ thus $g_{Q,Q'}$ belongs to the intersection of the discrete valuation rings of $\Kb(C)$ containing $\phi$ and $\Kb$ (\cite{Hartshorne}, proof of Lemma I.6.5), which is exactly the integral closure of $K[\phi]$ in $\Kb(C)$ (\cite{AtiyahMacDonald}, Corollary 5.22). Hence, the function $g_{Q,Q'}$ is integral on $K[\phi]$ and up to multiplication by some nonzero integer, we can and will assume it is integral on $\Ocal_K[\phi]$.

	For any fixed finite extension $L$ of $K$ included in $\Kb$, we define $f_{Q,Q',L} \in L(C)$ the product of the conjugates of $g_{Q,Q'}$ by $\Gal(\overline{L}/L)$. If $Q$ and $Q'$ belong to distinct orbits of poles for $\Gal(\overline{L}/L)$, the function $f_{Q,Q',L}$ has for only poles the orbit of poles of $Q$ by $\Gal(\Kb/L)$ and cancels at the poles of $\phi$ in the orbit of $Q'$ by $\Gal(\Kb/L)$ . Notice that we thus built only finitely many different functions (even with $L$ running through all finite extensions of $K$) because each $g_{Q,Q'}$ only has finitely many conjugates in $\Gal(K'/K)$.
	
	Now, let $\Ocal_1, \cdots, \Ocal_{r_L}$ be the orbits of poles of $\phi$ and denote for any $i \in \{1, \cdots, r_L\}$ by $f_{i,L}$ a product of $f_{Q_i, Q'_j,L}$ where $Q_i \in \Ocal_i$ and $Q'_j$ runs through representatives of the orbits (except $\Ocal_i$). Again, there is a finite number of possible choices, and we  obtain a function $f_{i,L} \in L(C)$ having for only poles the orbit $\Ocal_i$ and vanishing at all the other poles of $\phi$. By our construction of the $g_{Q,Q'}$ and $f_{i,L}$, we can and do choose $n \in \N_{\geq 1}$ such that for every $i \in \{ 1, \cdots, r_L \}$, $\phi f_{i,L}^n$ has exactly as poles the points of $\Ocal_i$ and is integral over $\Ocal_K[\phi]$. This implies that for any finite place $w \in M_L$, if $|\phi(P)|_w \leq 1$ then $|f_{i,L} (P)|_w \leq 1$, but we also need such a result for archimedean places. To do this, we apply Corollary \ref{corpasdepolecommun} to $f_{i,L}/\phi^k$ and $f_{i,L}$ (for any $i$) for some $k$ such that $f_{i,L}/\phi^k$ does not have poles at $\Ocal_i$, and take the maximum of the induced $M_K$-constants (Definition \ref{defMKconstante}) for any $L$ and $1 \leq i \leq r_L$. This gives an $M_K$-constant $\Ccal_0$ independant of $L$ such that
	\[
	\forall i \in \{1, \cdots, r_L\}, \forall w \in M_{\Kb}, \forall P \in C(\Kb), \log \min \left( \left| \frac{f_{i,L}}{\phi^k} (P)\right|_w, |f_{i,L}(P)|_w \right) \leq c_{0,v} \quad (w|v \in M_K). 
	\]
	In particular, the result interesting us in this case is that 
	\begin{equation}
	\label{eqmajofilplacesarchi}
		\forall i \in \{1, \cdots, r_L\}, \forall w \in M_{\Kb}, \forall P \in C(\Kb), |\phi(P)|_w \leq 1 \Rightarrow \log |f_{i,L} (P)|_w \leq c_{0,v},
	\end{equation}
and we can assume $c_{0,v}$ is 0 for any finite place $v$ by integrality of the $f_{i,L}$ over $\Ocal_K[\phi]$.
	As the sets of poles of the $f_{i,L}$ are mutually disjoint, we reapply Corollary \ref{corpasdepolecommun} for every pair $(\phi f_{i,L}^n, \phi f_{j,L}^n)$ with $1 \leq i < j \leq r_L$, which again by taking the maximum of the induced $M_K$-constants for all the possible combinations (Definition \ref{defMKconstante}) gives an $M_K$-constant $\Ccal_1$ such that for every $v \in M_K$ and every $(P,w) \in C(\Kb) \times M_{\Kb}$ with $w|v$, the inequality 
	\begin{equation}
	\label{eqineqsaufpourunefonc}
		\log |(\phi \cdot f_{i,L}^n) (P)|_w \leq c_{1,v}
	\end{equation}
is true for all indices $i$ except at most one (depending of the choice of $P$ and $w$).

Let us now suppose that $(L,S_L)$ is a pair satisfying Runge condition and $P \in C(L)$ with $\phi(P) \in \Ocal_{L,S_L}$. By integrality on $\Ocal_K[\phi]$, for every $i \in \{1, \cdots, r_L \}$, $|f_{i,L}(P)|_w \leq 1$ for every place $w \in M_L \backslash S_L$. For every place $w \in S_L$, there is at most one index $i$ not satisying \eqref{eqineqsaufpourunefonc} hence by Runge condition and pigeon-hole principle, there remains one index $i$ (depending on $P$) such that 
\begin{equation}
\label{eqmajophifiL}
	\forall w \in M_L, \quad \log |\phi(P) f_{i,L}^n (P)|_w \leq c_{1,v}.
\end{equation}
With \eqref{eqmajofilplacesarchi} and \eqref{eqmajophifiL}, we have obtained all the auxiliary results we need to finish the proof. By the product formula,
\begin{eqnarray*}
	0 & = & \sum_{w \in M_L} n_w \log |f_{i,L} (P)|_w \\
	& = & \sum_{\substack{w \in M_L \\ |\phi(P)|_w >1 }} n_w \log |f_{i,L} (P)|_w + \sum_{\substack{w \in M_L^{\infty} \\ |\phi(P)|_w \leq 1}} n_w \log |f_{i,L} (P)|_w  + \sum_{\substack{w \in M_L  \! \! \backslash M_L^{\infty} \\ |\phi(P)|_w \leq 1}} n_w \log |f_{i,L} (P)|_w.
\end{eqnarray*}
Here, the first sum on the right side will be linked to the height $h \circ \phi$ and the third sum is negative by integrality of the $f_{i,L}$, so we only have to bound the second sum. From \eqref{eqmajofilplacesarchi} and \eqref{eqineqinductionMKconstante}, we obtain 
\[
\sum_{\substack{w \in M_L^{\infty} \\ |\phi(P)|_w \leq 1}} n_w \log |f_{i,L} (P)|_w \leq \sum_{\substack{w \in M_L^{\infty} \\ |\phi(P)|_w \leq 1}} n_w c_{0,v} \leq [L:K] \sum_{v \in M_K^{\infty}} n_v c_{0,v}.
\]
On another side, by \eqref{eqmajophifiL} (and \eqref{eqineqinductionMKconstante} again), we have 
\begin{eqnarray*}
	n \cdot \sum_{\substack{w \in M_L \\ |\phi(P)|_w >1 }} n_w \log |f_{i,L} (P)|_w  & = & \sum_{\substack{w \in M_L \\ |\phi(P)|_w >1 }} n_w \log |\phi f_{i,L}^n(P)|_w - \sum_{\substack{w \in M_L \\ |\phi(P)|_w >1 }} n_w \log |\phi(P)|_w \\
	& \leq & \left([L:K] \sum_{v \in M_K} n_v c_{1,v} \right) - [L:\Q] h(\phi(P)). 
\end{eqnarray*}
Hence, we obtain 
\begin{eqnarray*}
	0 & \leq & [L:K] \sum_{v \in M_K} n_v c_{1,v} - [L:\Q] h(\phi(P)) + [L:K] n \sum_{v \in M_K^{\infty}} n_v c_{0,v},
\end{eqnarray*}
which is equivalent to 
\[
	h (\phi(P)) \leq \frac{1}{[K : \Q]}\sum_{v \in M_K} n_v (c_{1,v} + n c_{0,v}).
\]
We thus obtained a bound on $h (\phi(P))$ independent on the choice of $(L,S_L)$ satisfying the Runge condition, and together with the bound on the degree 
\[
	[L : \Q] \leq 2 |S_L| < 2 r_L \leq 2 r,
\]
we get the finiteness.
\end{proof}

\section{The main result : tubular Runge theorem}
\label{sectiontubularRunge}

We will now present our version of Runge theorem with tubular neighbourhoods, which 
generalises Theorem 4 $(b)$ and $(c)$ of \cite{Levin08}. As its complete formulation is quite lengthy, we indicated the different hypotheses by the letter $H$ and the results by the letter $R$ to simplify the explanation of all parts afterwards. The key condition for integral points generalising Runge condition of Proposition \ref{propBombieri} is indicated by the letters TRC.

We recall that the crucial notion of tubular neighbourhood is explained in Definitions \ref{defvoistub} and \ref{defhorsdunvoistub}, and we advise the reader to look at the simplified version of this theorem stated in the Introduction to get more insight if necessary.

\begin{thm}[Tubular Runge theorem]	
\hspace*{\fill}
\label{thmRungetubulaire}

\textbf{(H0)} Let $K$ be a number field, $S_0$ a set of places of $K$ containing $M_K^{\infty}$ and $\Ocal$ the integral closure of $\Ocal_{K,S_0}$ in some finite Galois extension $K'$ of $K$.

\textbf{(H1)} Let $\Xcal$ be a normal projective scheme over $\Ocal_{K,S_0}$ and $D_1, \cdots, D_r$ be effective Cartier divisors on $\Xcal_\Ocal = \Xcal \times_{\Ocal_{K,S_0}} \Ocal$ such that $D_\Ocal= \bigcup_{i=1}^r D_i$ is the scalar extension to $\Ocal$ of some Cartier divisor $D$ on $\Xcal$, and that $\Gal(K'/K)$ permutes the generic fibers $(D_i)_{K'}$. For every extension $L/K$, we denote by $r_L$ the number of orbits of $(D_1)_{K'}, \cdots, (D_r)_{K'}$ for the action of $\Gal(K'L/L)$.

\textbf{(H2)} Let $Y$ be a closed sub-$K$-scheme of $\Xcal_K$ and $\Vcal$ be a tubular neighbourhood of $Y$ in $\Xcal_K$. Let $m_Y \in \N$ be the minimal number such that the intersection of any $(m_Y+1)$ of the divisors $(D_i)_{K'}$ amongst the $r$ possible ones is included in $Y_{K'}$. 

\textbf{(TRC)} The \textbf{tubular Runge condition} for a pair $(L,S_L)$, where $L/K$ is finite and $S_L$ contains all the places above $S_0$, is 
\[
	m_Y |S_L| < r_L.
\]

Under these hypotheses and notations, the results are the following :

\textbf{(R1)} If $(D_1)_{K'}, \cdots , (D_r)_{K'}$ are ample divisors, the set 
\begin{equation}
\label{eqensfiniRungetubample}
	\bigcup_{(L,S_L)} \{P \in (\Xcal \backslash D) (\Ocal_{L,S_L}) \, | \, P \notin \Vcal \},
\end{equation}
where $(L,S_L)$ goes through all the pairs satisfying the tubular Runge condition, is \textbf{finite}.

\textbf{(R2)} If $(D_1)_{K'}, \cdots, (D_r)_{K'}$ are big divisors, there exists a proper closed subset $Z_{K'}$ of $\Xcal_{K'}$ such that the set 
\[
	\left( \bigcup_{(L,S_L)} \{P \in (\Xcal \backslash D) (\Ocal_{L,S_L}) \, | \, P \notin \Vcal \} \right) \backslash Z_{K'} (\Kb),
\]
where $(L,S_L)$ goes through all the pairs satisfying the tubular Runge condition, is \textbf{finite}.
\end{thm}

We separated the comments about Theorem \ref{thmRungetubulaire} in two remarks below : the first one explains its hypotheses and results, the second compares it with other theorems.

\begin{rem}
\hspace*{\fill}
\label{remRungetubulaire}

$(a)$ 
	The need for the extensions of scalars to $K'$ and $\Ocal$ in \textit{\textbf{(H0)}} and \textit{\textbf{(H1)}} is the analogue of the fact that the poles of $\phi$ are not necessarily $K$-rational in the case of curves, hence the assumption that the $(D_i)_{K'}$ are all conjugates by $\Gal(K'/K)$ and the definition of $r_L$ given in \textit{\textbf{(H1)}}. It will induce technical additions of the same flavour as the auxiliary functions $f_{Q,Q',L}$ in the proof of Bombieri's theorem (Proposition \ref{propBombieri}).
	
$(b)$ The motivation for the tubular Runge condition is the following : imitating the principle of proof for curves (Remark \ref{remRungecourbes} $(b)$), if $P \in (\Xcal \backslash D) (\Ocal_{L,S_L})$, we can say that at the places $w$ of $M_L \backslash S_L$, this point is ``$w$-adically far'' from $D$. Now, the divisors $(D_1)_{K'}, \cdots, (D_r)_{K'}$ can intersect (which does not happen for distinct points on curves), so for $w \in S_L$, this point $P$ can be ``$w$-adically close'' to many divisors at the same time. More precisely, it can be ``$w$-adically close'' to at most $m$ such divisors, where $m=m_{\emptyset}$, i.e. the largest number such that there are $m$ divisors among $D_1, \cdots, D_r$ whose set-theoretic intersection is nonempty. This number is also defined in \cite{Levin08} but we found that for our applications, it often makes Runge condition too strict. Therefore, we allow the use of the closed subset $Y$ in \textbf{\textit{(H2)}}, and if we assume that our point $P$ is never too close to $Y$ (i.e. $P \notin \Vcal$), this $m$ goes down to $m_Y$ by definition.  Thus, we only need to take out $m_Y$ divisors by place $w$ in $S_L$, hence the tubular Runge condition $m_Y |S_L|< r_L$. Actually, one can even mix the Runge conditions, i.e. assume that $P$ is close to $Y$ exactly at $s_1$ places, and close from one of the divisors (but not $Y$) at $s_2$ places : following along the lines of the proof below, we obtain finiteness given the Runge condition $s_1 m_{\emptyset} + s_2 m_Y < r_L$.

$(c)$ The last main difference with the case of curves is the assumption of ample or big divisors, respectively in \textit{\textbf{(R1)}} and \textit{\textbf{(R2)}}. In both cases, such an assumption is necessary twice. First, we need it to translate by Proposition \ref{propreductionamplegros} the integrality condition on schemes to an integrality expression on auxiliary functions (such as in section 2 of \cite{Levin08}) to use the machinery of $M_K$-constants and the key result (Proposition \ref{propcle}). Then, we need it to ensure that after obtaining a bound on the heights associated to the divisors, it implies finiteness (implicit in Proposition \ref{propreductionamplegros}, see also Remark \ref{rempropamplegros} $(a)$).
\end{rem}

\begin{rem}
\hspace*{\fill}
\label{remcomparaisonCLZetstratification}

$(a)$ This theorem has some resemblance to Theorem CLZ of \cite{CorvajaLevinZannier} (where our closed subset $Y$ would be the analogue of the $\Ycal$ in that article), let us point out the differences. In Theorem CLZ, there is no hypothesis of the set of places $S_L$, no additional hypothesis of integrality (appearing for us under the form of a tubular neighbourhood), and the divisors are assumed to be normal crossing divisors, which is replaced in our case by the tubular Runge condition. As for the results themselves, the finiteness formulated by CLZ depends on the set $S_L$ (that is, it is not clear how it would prove such an union of sets such as in our Theorem is finite). Finally, the techniques employed are greatly different : Theorem CLZ uses Schmidt's subspace theorem which is noneffective, whereas our method can be made effective if one knows the involved auxiliary functions. It might be possible (and worthy of interest) to build some bridges between the two results, and the techniques involved.

$(b)$ Theorem \ref{thmRungetubulaire} can be seen as a stratification of Runge-like results depending on the dimension of the intersection of the involved divisors : at one extreme, the intersection is empty, and we get back Theorem 4 $(b)$ and $(c)$ of \cite{Levin08}. At the other extreme, the intersection is a divisor (ample or big), and the finiteness is automatic by the hypothesis for points not belonging in the tubular neighbourhood (see Remark \ref{remhorsvoistub}). Of course, this stratification is not relevant in the case of curves. In another perspective, for a fixed closed subset $Y$, Theorem \ref{thmRungetubulaire} is more a concentration result of integral points than a finiteness result, as it means that even if we choose a tubular neighbourhood $\Vcal$ of $Y$ as small as possible around $Y$, there is only a finite number of integral points in the set \eqref{eqensfiniRungetubample}, i.e. these integral points (ignoring the hypothese $P \notin \Vcal$) must concentrate around $Y$ (at least at one of the places $w \in M_L$). Specific examples will be given in section \ref{sectionapplicationsSiegel} and \ref{sectionexplicitRunge}.

\end{rem}

Let us now prove Theorem \ref{thmRungetubulaire}, following the ideas outlined in Remark \ref{remRungetubulaire}.

\begin{proof}[Proof of Theorem \ref{thmRungetubulaire}]
\hspace*{\fill}
	
\textit{\textbf{(R1)}} Let us first build the embeddings we need. For every subextension $K''$ of $K'/K$, the action of $\Gal(K'/K'')$ on the divisors $(D_1)_{K'}, \cdots, (D_r)_{K'}$ has orbits denoted by $O_{K'',1}, \cdots, O_{K'',r_{K''}}$. Notice that any $m_Y+1$ such orbits still have their global intersection included in $Y$ : regrouping the divisors by orbits does not change this fact.

For each such orbit, the sum of its divisors is ample by hypothesis and coming from an effective Cartier divisor on $\Xcal_{K''}$, hence one can choose by Proposition \ref{propreductionamplegros} an appropriate embedding $\psi_{K'',i} : \Xcal_{K''} \rightarrow \P^{n_i}_{K''}$, whose coordinates functions (denoted by $\phi_{K'',i,j} = (x_j/x_0) \circ \psi_{K'',i} (1 \leq j \leq n_i)$) are small on integral points of $(\Xcal_\Ocal \backslash O_{K'',i})$. We will denote by $\Ccal_0$ the maximum of the (induced) $M_K$-constants obtained for by the Proposition \ref{propreductionamplegros} for all possible $K''/K$ and orbits $O_{K'',i} (1 \leq i \leq r_{K''})$. The important point of this is that for any extension $L/K$, any $v \in M_K \backslash S_0$, any place $w \in M_L$ above $v$ and any $P \in (\Xcal \backslash D) (\Ocal_{L,w})$, choosing $L'=K' \cap L$, one has 
\begin{equation}
\label{eqinterRungetousplongements}
	\max_{\substack{1 \leq i \leq r_L \\ 1 \leq j \leq n_i}} \log |\phi_{L',i,j} (P)|_w \leq c_{0,v}.
\end{equation}

This is the first step to obtain a bound on the height of one of the $\psi_{K'',i} (P)$. For fixed $P$, we only have to do so for one of the $i \in \{1, \cdots, r_L \}$ as long as the bound is uniform in the choice of $(L,S_L)$ (and $P$), to obtain finiteness as each $\psi_{K'',i}$ is an embedding. To this end, one only needs to bound the coordinate functions on the places $w$ of $M_L \backslash S_L$, which is what we will do now.

For a subextension $K''$ of $K'/K$ again, by hypothesis \textbf{\textit{(H2)}} (and especially the definition of $m_Y$), taking any set $\Ical$ of $m_Y+1$ couples $(i,j), 1 \leq i \leq r_{K''}, j \in \{1, \cdots, n_i\}$ with $m_Y+1$ different indices $i$ and considering the rational functions $\phi_{K'',i,j}, (i,j) \in \Ical$, whose common poles are included in $Y$ by hypothesis, we can apply Proposition \ref{propcle} to these functions and the tubular neighbourhood $\Vcal = (V_w)_{w \in M_{\Kb}}$. Naming as $\Ccal_1$ the maximum of the (induced) obtained $M_K$-constants (also for all the possible $K''$), we just proved that for every subextension $K''$ of $K'/K$, every place $w \in M_{\Kb}$ (above $v \in M_K$) and any $P \in \Xcal(\Kb) \backslash V_w$, the inequality 
\begin{equation}
\label{eqineqfaussepourauplusmY}
	\max_{1 \leq j \leq n_i} \log |\phi_{K'',i,j} (P)|_w \leq c_{1,v}
\end{equation}
is true except for at most $m_Y$ different indices $i \in \{1, \cdots, r_{K''} \}$.

Now, let us consider $(L,S_L)$ a pair satisfying tubular Runge condition $m_Y |S_L| < r_L$ and denote $L' = K' \cap L$ again. For $P \in (\Xcal \backslash D) (\Ocal_{L,S_L})$ not belonging to $\Vcal$, by \eqref{eqinterRungetousplongements},  \eqref{eqineqfaussepourauplusmY} and tubular Runge condition, there remains an index $i \in \{1, \cdots, r_L\}$ (dependent on $P$) such that 
\[
	\forall w \in M_L, \quad \max_{1 \leq j \leq n_i} \log |\phi_{L',i,j} (P)|_w \leq \max(c_{0,v},c_{1,v}) \quad (w | v \in M_K).
\]
This gives immediately a bound on the height of $\psi_{L',i}(P)$ independent of the choice of pair $(L,S_L)$ (except the fact that $L' = K' \cap L$) and this morphism is an embedding, hence the finiteness of the set of points 
\[
	\bigcup_{(L,S_L)} \{P \in (\Xcal \backslash D) (\Ocal_{L,S_L}) \, | \, P \notin \Vcal \},
\]
where $(L,S_L)$ goes through all the pairs satisfying tubular Runge condition, because $[L:\Q]$ is also bounded by this condition.

\textit{\textbf{(R2)}}
 
The proof is the same as for \textit{\textbf{(R1)}} except that we have to exclude a closed subset of $\Xcal_{K'}$ for every big divisor involved, and their reunion will be denoted by $Z_{K'}$. The arguments above hold for every point $P \notin Z_{K'} (\Kb)$ (both for the expression of integrality by auxiliary functions, and for the conclusion and finiteness outside of this closed subset), using again Propositions \ref{propreductionamplegros} and \ref{propcle}.

\end{proof}

\section{Reminders on Siegel modular varieties}
\label{sectionrappelsSiegel}

In this section, we recall the classical constructions and results for the Siegel modular varieties, parametrising principally polarised abelian varieties with a level structure. Most of those results are extracted (or easily deduced) from these general references : Chapter V of \cite{CornellSilvermanArithmeticGeometry} for the basic notions on abelian varieties, \cite{Debarre99} for the complex tori, their line bundles, theta functions and moduli spaces, Chapter II of \cite{MumfordTata} for the classical complex theta functions and \cite{MumfordTataII} for their links with theta divisors, and Chapter V of \cite{ChaiFaltings} for abelian schemes and their moduli spaces. 

Unless specified, all the vectors of $\Z^g, \R^g$ and $\C^g$ are assumed to be row  vectors.

\subsection{Abelian varieties and Siegel modular varieties}
\label{subsecabvarSiegelmodvar}

\begin{defi}[Abelian varieties and polarisation]
\hspace*{\fill}
\label{defibaseabvar}
\begin{itemize}
\item[$\bullet$] An \textit{abelian variety} $A$ over a field $k$ is a projective algebraic group over $k$. Each abelian variety $A_{/k}$ has a dual abelian variety denoted by $\widehat{A} = \Pic^0 (A/k)$ (\cite{CornellSilvermanArithmeticGeometry}, section V.9).

\item[$\bullet$] A \textit{principal polarisation} is an isomorphism $\lambda : A \rightarrow \widehat{A}$ such that there exists a line bundle $L$ on $A_{\overline{k}}$ with $\dim H^0(A_{\overline{k}},L)=1$ and $\lambda$ is the morphism
\[
	\fonction{\lambda}{A_{\overline{k}}}{\widehat{A_{\overline{k}}}}{x}{T_x^* L \otimes L^{-1}}
\]
 (\cite{CornellSilvermanArithmeticGeometry}, section V.13).

\item[$\bullet$] Given a pair $(A,\lambda)$, for every $n \geq 1$ prime to $\carac(k)$, we can define the \textit{Weil pairing}
\[
	A[n] \times A[n] \rightarrow \mu_n (\overline{k}),
\]
where $A[n]$ is the $n$-torsion of $A(\overline{k})$ and $\mu_n$ the group of $n$-th roots of unity in $\overline{k}$. It is alternate and nondegenerate (\cite{CornellSilvermanArithmeticGeometry}, section V.16). 

\item[$\bullet$] Given a pair $(A,\lambda)$, for $n \geq 1$ prime to $\carac(k)$, a \textit{symplectic level $n$ structure on }$A[n]$ is a basis $\alpha_n$ of $A[n]$ in which the matrix of the Weil pairing is 
\[
	J = \begin{pmatrix} 0 & I_g \\ 
		- I_g & 0
	\end{pmatrix}.
\]

\item[$\bullet$] Two triples $(A,\lambda,\alpha_n)$ and $(A',\lambda',\alpha'_n)$ of principally polarised abelian varieties over $K$ with level $n$-structures are \textit{isomorphic} if there is an isomorphism of abelian varieties $\phi : A \rightarrow A'$ such that $\phi^* \lambda' = \lambda$ and $\phi^* \alpha'_n = \alpha_n$.

\end{itemize}	
\end{defi}

In the case of complex abelian varieties, the previous definitions can be made more explicit.

\begin{defi}[Complex abelian varieties and symplectic group]
\hspace*{\fill}
\label{deficomplexabvar}

Let $g \geq 1$.

\begin{itemize}
\item[$\bullet$]  The \textit{half-superior Siegel space of order }$g$, denoted by $\Hcal_g$, is the set of matrices 
\begin{equation}
\label{eqdefdemiespaceSiegel}
\Hcal_g := \{ \tau \in M_g (\C) \, | \, {}^t \tau = \tau \, \, \textrm{and} \, \, \Im \tau >0 \},
\end{equation}
where $\Im \tau >0$ means that this symmetric matrix of $M_g (\R)$ is positive definite. This space is an open subset of $M_g(\C)$.

\item[$\bullet$] For any $\tau \in \Hcal_g$, we define 
\begin{equation}
\Lambda_\tau := \Z^g + \Z^g \tau \quad \textrm{and} \quad A_\tau := \C^g / \Lambda_\tau.
\end{equation}
Let $L_\tau$ be the line bundle on $A_\tau$ made up as the quotient of $\C^g \times \C$ by the action of $\Lambda_\tau$ defined by 
\begin{equation}
\label{eqdeffibresurAtau}
\forall p,q \in \Z^g, \quad (p \tau + q) \cdot (z,t) = \left(z+p \tau + q, e^{ - i \pi p \tau {}^tp - 2 i \pi p {}^t z} t \right).
\end{equation}
Then, $L_\tau$ is an an ample line bundle on $A_\tau$ such that $\dim H^0 (A_\tau,L_\tau)=1$, hence $A_\tau$ is a complex abelian variety and $L_\tau$ induces a principal polarisation denoted by $\lambda_\tau$ on $A_\tau$ (see for example \cite{Debarre99}, Theorem VI.1.3). We also denote by $\pi_\tau : \C^g \rightarrow A_\tau$ the quotient morphism.

\item[$\bullet$] For every $n \geq 1$, the Weil pairing $w_{\tau,n}$ associated to $(A_\tau,\lambda_\tau)$ on $A_\tau[n]$ is defined by
\[
	\fonction{w_{\tau,n}}{A_\tau[n] \times A_\tau[n]}{\mu_n (\C)}{(\overline{x},\overline{y})}{e^{ 2 i \pi n w_\tau(x,y)}},
\]
where $x,y \in \C^g$ have images $\overline{x}, \overline{y}$ by $\pi_\tau$, and $w_\tau$ is the $\R$-bilinear form on $\C^g \times \C^g$ (so that $w_\tau(\Lambda_\tau \times \Lambda_\tau) = \Z$) defined by 
\[
	w_\tau (x,y) := \Re(x) \cdot \Im (\tau)^{-1} \cdot {}^t \Im(y) - \Re(y) \cdot \Im (\tau)^{-1} \cdot  {}^t \Im(x)
\]
(also readily checked by making explicit the construction of the Weil pairing).

\item[$\bullet$] Let $(e_1, \cdots, e_g)$ be the canonical basis of $\C^g$. The family 
\begin{equation}
	\label{eqdefalphataun}
	(\pi_\tau(e_1/n), \cdots, \pi_\tau(e_g/n), \pi_\tau(e_1 \cdot \tau/n), \cdots, \pi_\tau(e_g \cdot \tau/n))
\end{equation}
is a symplectic level $n$ structure on $(A_\tau, \lambda_\tau)$, denoted by $\alpha_{\tau,n}$.

\item[$\bullet$] Let $J = \begin{pmatrix} 0 & 1 \\ -1 & 0 \end{pmatrix} \in M_{2g}(\Z)$. For any commutative ring $A$, the \textit{symplectic group of order $g$ over $A$}, denoted by $\Sp_{2g} (A)$, is the subgroup of $\GL_{2g}(A)$ defined by 
\begin{equation}
\label{eqdefgroupesymplectique}
\Sp_{2g} (A) := \{ M \in \GL_{2g} (A) \, \, | \, \, {}^t M J M = J \}, \qquad J:= \begin{pmatrix} 0 & I_g \\- I_g & 0 \end{pmatrix}.
\end{equation}

For every $n \geq 1$, the \textit{symplectic principal subgroup of degree $g$ and level $n$}, denoted by $\Gamma_g(n)$, is the subgroup of $\Sp_{2g} (\Z)$ made up by the matrices congruent to $I_{2g}$ modulo $n$. For every $\gamma = \begin{pmatrix}A & B \\ C & D \end{pmatrix} \in \Sp_{2g} (\R)$ and every $\tau \in \Hcal_g$, we define 
\begin{equation}
\label{eqdefjetactionsymplectique}
j_\gamma (\tau) = C \tau + D \in \GL_g (\C), \quad \textrm{and} \quad \gamma \cdot \tau = (A \tau + B)(C \tau + D)^{-1},
\end{equation}
which defines a left action by biholomorphisms of $\Sp_{2g} (\R)$ on $\Hcal_g$, and $(\gamma,\tau) \mapsto j_\gamma(\tau)$ is a left cocycle for this action (\cite{Klingen}, Proposition I.1).

\item[$\bullet$] For every $g \geq 2$, $n \geq 1$ and $k \geq 1$, a \textit{Siegel modular form of degree $g$, level $n$ and weight $k$} is an holomorphic function $f$ on $\Hcal_g$ such that 
\begin{equation}
\label{eqdefSiegelmodularform}
	\forall \gamma \in \Gamma_g(n), \quad f (\gamma \cdot z) = \det(j_\gamma(z))^k f(z).
\end{equation}

\end{itemize}

\end{defi}

The reason for this seemingly partial description of the complex abelian varieties is that the $(A_\tau,\lambda_\tau)$ described above actually make up all the principally polarised complex abelian varieties up to isomorphism. The following results can be found in Chapter VI of \cite{Debarre99} except the last point which is straightforward.

\begin{defiprop}[Uniformisation of complex abelian varieties]
\hspace*{\fill}
\label{defipropuniformcomplexabvar}
\begin{itemize}
	\item[$\bullet$] Every principally polarised complex abelian variety of dimension $g$ with symplectic structure of level $n$ is isomorphic to some triple $(A_\tau, \lambda_\tau,\alpha_{\tau,n})$ where $\tau \in \Hcal_g$.
	
	\item[$\bullet$] For every $n \geq 1$, two triples $(A_\tau,\lambda_\tau,\alpha_{\tau,n})$ and $(A_{\tau'},\lambda_{\tau'},\alpha_{\tau',n})$ are isomorphic if and only if there exists $\gamma \in \Gamma_g(n)$ such that $\gamma \cdot \tau = \tau'$, and then such an isomorphism is given by 
	\[
		\fonctionsansnom{A_\tau}{A_{\tau'}}{z \! \mod \Lambda_\tau}{z \! \cdot j_\gamma(\tau)^{-1} \mod \Lambda_{\tau'}}.
	\]
	
	\item[$\bullet$] The \textit{Siegel modular variety of degree $g$ and level $n$} is the quotient $A_g(n)_\C := \Gamma_g(n) \backslash \Hcal_g$. From the previous result, it is the moduli space of principally polarised complex abelian varieties of dimension $g$ with a symplectic level $n$ structure. As a quotient, it also inherits a structure of normal analytic space (with finite quotient singularities) of dimension $g(g+1)/2$, because $\Gamma_g(n)$ acts properly discontinuously on $\Hcal_g$.
	
	\item[$\bullet$] For every positive divisor $m$ of $n$, the natural morphism $A_g(n)_\C \rightarrow A_g(m)_\C$ induced by the identity of $\Hcal_g$ corresponds in terms of moduli to multiplying the symplectic basis $\alpha_{\tau,n}$ by $n/m$, thus obtaining $\alpha_{\tau,m}$.
	
	\item[$\bullet$] For every $g \geq 1$ and $n \geq 1$, the quotient of $\Hcal_g \times \C$ by the action of $\Gamma_g(n)$ defined as
	\begin{equation}
		\label{eqdeffibreL}
		\gamma \cdot (\tau,t) = (\gamma \cdot \tau, t / \det( j_\gamma(z))) 
	\end{equation}
	is a variety over $\Hcal_g$ denoted by $L$. For a large enough power of $k$ (or if $n \geq 3$), $L^{\otimes k}$ is a line bundle over $A_g(n)_\C$, hence $L$ is a $\Q$-line bundle over $A_g(n)_\C$ called \textit{line bundle of modular forms of weight one} over $A_g(n)_\C$. By definition \eqref{eqdefSiegelmodularform}, for every $k \geq 1$, the global sections of $L^{\otimes k}$ are the Siegel modular forms of degree $g$, level $n$ and weight $k$.

\end{itemize}
\end{defiprop}

Let us now present the compactification of $A_g(n)_\C$ we will use, that is the Satake compactification (for a complete description of it, see section 3 of \cite{Namikawa}).

\begin{defiprop}[Satake compactification]
\hspace*{\fill}
\label{defipropSatakecompactification}

Let $g \geq 1$ and $n \geq 1$.  The normal analytic space $A_g(n)_\C$ admits a compactification called \textit{Satake compactification} and denoted by $A_g(n)^S_\C$, satisfying the following properties.

$(a)$ $A_g(n)^S_\C$ is a compact normal analytic space (of dimension $g(g+1)/2$, with finite quotient singularities) containing $A_g(n)_\C$ as an open subset and the boundary $\partial A_g(n)_\C := A_g(n)^S_\C \backslash A_g(n)_\C$ is of codimension $g$ (see \cite{CartanSatake57} for details).

$(b)$ As a normal analytic space, $A_g(n)^S_\C$ is a projective algebraic variety. More precisely, for ${\textrm{M}}_g(n)$ the graded ring of Siegel modular forms of degree $g$ and level $n$, $A_g(n)^S_\C$ is canonically isomorphic to $\Proj_\C {\textrm{M}}_g (n)$ (\cite{CartanPlongements57}, ``théorème fondamental'').

In particular, one can obtain naturally $A_g(n)^S_\C$ by fixing for some large enough weight $k$ a basis of modular forms of ${\textrm{M}}_g(n)$ of weight $k$ and evaluating them all on $A_g(n)_\C$ to embed it in a projective space, so that $A_g(n)^S_\C$ is the closure of the image of the embedding in this projective space.

$(c)$ The $\Q$-line bundle $L$ of modular forms of weight 1 on $A_g(n)_\C$ extends naturally to $A_g(n)^S_\C$ (and is renoted $L$), to an ample $\Q$-line bundle (this is a direct consequence of $(b)$).
\end{defiprop}

\subsection{Further properties of Siegel modular varieties}
\label{subsecfurtherpropSiegelmodvar}
As we are interested in the reduction of abelian varieties on number fields, one needs to have a good model of $A_g(n)_\C$ over integer rings, as well as some knowledge of the geometry of $A_g(n)_\C$. The integral models below and their properties are given in Chapter V of \cite{ChaiFaltings}.

\begin{defi}[Abelian schemes]
\hspace*{\fill}
\label{defabelianscheme}

	$(a)$ An \textit{abelian scheme} $A \rightarrow S$ is a smooth proper group scheme whose fibers are geometrically connected. It also has a natural \textit{dual} abelian scheme $\widehat{A} = \Pic^0 (A/S)$, and it is \textit{principally  polarised} if it is endowed with an isomorphism $\lambda : A \rightarrow \widehat{A}$ such that at every geometric point $\overline{s}$ of $S$, the induced isomorphism $\lambda_{\overline{s}} : A_{\overline{s}} \rightarrow \widehat{A}_{\overline{s}}$ is a principal polarisation of $A_{\overline{s}}$.
	
	$(b)$ A \textit{symplectic structure of level $n \geq 1$} on a principally polarised abelian scheme $(A,\lambda)$ over  a $\Z[\zeta_n,1/n]$-scheme $S$ is the datum of an isomorphism of group schemes $A[n] \rightarrow (\Z/n\Z)^{2g}$, which is symplectic with respect to $\lambda$ and the canonical pairing on $(\Z/n\Z)^{2g}$ given by the matrix $J$ (as in \eqref{eqdefgroupesymplectique}).
\end{defi}

\begin{defiprop}[Algebraic moduli spaces]
\hspace*{\fill}
\label{defipropalgmodulispaces}

	For every integers $g \geq 1$ and $n \geq 1$ : 
	
	$(a)$ The Satake compactification $A_g(n)^S_\C$ has an integral model $\Acal_g (n)^S$ on $\Z[\zeta_n, 1/n]$ which contains as a dense open subscheme the (coarse, if $n \leq 2$) moduli space $\Acal_g(n)$ on $\Z[\zeta_n,1/n]$ of principally polarised abelian schemes of dimension $g$ with a symplectic structure of level $n$. This scheme $\Acal_g(n)^S$ is normal, proper and of finite type on $\Z[\zeta_n,1/n]$ (\cite{ChaiFaltings}, Theorem V.2.5).

	$(b)$ For every divisor $m$ of $n$, we have canonical degeneracy morphisms $\Acal_g(n)^S \rightarrow \Acal_g(m)^S$ extending the morphisms of Definition \ref{defipropuniformcomplexabvar}.
	
\end{defiprop}

Before tackling our own problem, let us give some context on the divisors on $A_g(n)^S_\C$ to give a taste of the difficulties to overcome.

\begin{defi}[Rational Picard group]
\hspace*{\fill}

For every normal algebraic variety $X$ on a field $K$, the \textit{rational Picard group} of $X$ is the $\Q$-vector space 
\[
\Pic(X)_\Q := \Pic(X) \otimes_\Z \Q.
\]
\end{defi}

\begin{prop}[Rational Picard groups of Siegel modular varieties]
\hspace*{\fill}
\label{proprationalPicardSiegel}

Let $g \geq 2$ and $n \geq 1$.

$(a)$ Every Weil divisor on $A_g(n)_\C$ or $A_g(n)^S_\C$ is up to some multiple a Cartier divisor, hence their rational Picard group is also their Weil class divisor group tensored by $\Q$.

$(b)$ For $g=3$, the Picard rational groups of $A_3(n)^S_\C$ and $A_3(n)_\C$ are equal to $\Q \cdot L$ for every $n \geq 1$.

$(c)$ For $g=2$, one has $\Pic_\Q (A_2(1)^S_\C) = \Q \cdot L$.
	
\end{prop}

This result has the following immediate corollary, because $L$ is ample on $A_g(n)^S_\C$ for every $g \geq 2$ and every $n \geq 1$ (Definition-Proposition \ref{defipropSatakecompactification} $(c)$).

\begin{cor}[Ample and big divisors on Siegel modular varieties]
\hspace*{\fill}

A $\Q$-divisor on $A_g(n)_\C$ or $A_g(n)^S_\C$ with $g=3$ (or $g=2$ and $n=1$) is ample if and only if it is big if and only if it is equivalent to $a \cdot L$ with $a>0$.

\end{cor}

\begin{rem}
\label{remampledifficilepourA2}
	We did not mention the case of modular curves (also difficult, but treated by different methods): the point here is that the cases $g \geq 3$ are surprisingly much more uniform because then $\Pic(A_g(n)^S_\C) = \Pic (A_g(1)^S_\C)$. The reason is that some rigidity appears from $g \geq 3$ (essentially by the general arguments of \cite{Borel81}), whereas for $g=2$, the situation seems very complex already for the small levels (see for example $n=3$  in \cite{HoffmanWeintraub00}).
	
	This is why the ampleness (or bigness) is in general hard to figure out for given divisors of $A_2(n), n >1$. We consider specific divisors in the following (namely, divisors of zeroes of theta functions), whose ampleness will not be hard to prove.
\end{rem}

\begin{proof}[Proof of Proposition \ref{proprationalPicardSiegel}]
\hspace*{\fill}

$(a)$ This is true for the $A_g(n)^S_\C$ by \cite{ArtalBartolo14} as they only have finite quotient singularities,  (this result actually seems to have been generally assumed a long time ago). Now, as $\partial A_g(n)^S_\C$ is of codimension at least 2, the two varieties $A_g(n)^S_\C$ and $A_g(n)_\C$ have the same Weil and Cartier divisors, hence the same rational Picard groups.

$(b)$ This is a consequence of general results of \cite{Borel81} further refined in \cite{Weissauer92} (it can even be generalised to every $g \geq 3$). 

$(c)$ This comes from the computations of section III.9 of \cite{Mumford83} (for another  compactification, called toroidal), from which we extract the result for $A_2(1)_\C$ by a classical restriction theorem (\cite{Hartshorne}, Proposition II.6.5) because the boundary for this compactification is irreducible of codimension 1. The result for $A_2(1)^S_\C$ is then the same because the boundary is of codimension 2. 
\end{proof}

\subsection{Theta divisors on abelian varieties and moduli spaces}
\label{subsecthetadivabvar}
We will now define the useful notions for our integral points problem.

\begin{defi}[Theta divisor on an abelian variety]
\hspace*{\fill}
\label{defithetadivisorabvar}

Let $k$ be an algebraically closed field and $A$ an abelian variety over $k$.

Let $L$ be an ample symmetric line bundle on $A$ inducing a principal polarisation $\lambda$ on $A$. A \textit{theta function associated to} $(A,L)$ is a nonzero global section $\vartheta_{A,L}$ of $L$. The \textit{theta divisor associated to }$(A,L)$, denoted by $\Theta_{A,L}$, is the divisor of zeroes of $\vartheta_{A,L}$, well-defined and independent of our choice because $\dim H^0(A,L)=\deg(\lambda)^2 = 1$.
\end{defi}

The theta divisor is in fact determined by the polarisation $\lambda$ itself, up to a finite ambiguity we make clear below.

\begin{prop}
\hspace*{\fill}
\label{propambiguitedivthetaAL}

Let $k$ be an algebraically closed field and $A$ an abelian variety over $k$.

Two ample symmetric line bundles $L$ and $L'$ on $A$ inducing a principal polarisation induce the same one if and only if $L' \cong T_x^* L$ for some $x \in A|2]$, and then 
\[
	\Theta_{A,L'} = \Theta_{A,L} + x.
\]

\end{prop}

\begin{proof}
	For any line bundle $L$ on $A$, let us define 
	\[
	\fonction{\lambda_L}{A}{\widehat{A} = \Pic^0(A)}{x}{T_x^* L \otimes L^{-1}}.
	\]
	This is a group morphism and the application $L \mapsto \lambda_L$ is additive from $\Pic(A)$ to $\Hom(A,\widehat{A})$, with kernel $\Pic^0(A)$ (\cite{MumfordAbVar}, Chapter II, Corollary 4 and what follows, along with section II.8). Moreover, when $L$ is ample, the morphism $\lambda_L$ is the polarisation associated to $L$, in particular surjective. Now, for every $x \in A(k)$, if $L' \cong T_x^* L$, then $L' \otimes L^{-1}$ belongs to $\Pic^0(A)$, therefore $\lambda_{L'} = \lambda_{L}$. Conversely, if $\lambda_{L'} = \lambda_L$, one has $L' \otimes L^{-1} \in \Pic^0(A)$, hence if $L$ is ample, by surjectivity, one has $x \in A(k)$ such that $L'  \cong T_x^* L$. Finally, if $L$ and $L'$ are symmetric, having $[-1]^* L \cong L$ and $[-1]^* L' \cong L'$, we obtain $T_{-x}^* L \cong T_x^* L$ but as $\lambda_L$ is an isomorphism, this implies $[2] \cdot x = 0$, hence $x \in A[2]$.
	
	Therefore, for $\vartheta_{A,L}$ a nonzero section of $L$, $T_x^* \vartheta_{A,L}$ can be identified to a nonzero section of $L$, hence
	\[
	\Theta_{A,L'} = \Theta_{A,L} - x = \Theta_{A,L} + x.
	\]
\end{proof}

When $\carac(k) \neq 2$, adding to a principally polarised abelian variety $(A,\lambda)$ of dimension $g$ the datum $\alpha_2$ of a symplectic structure of level 2, we can determine an unique ample symmetric line bundle $L$ with the following process called \textit{Igusa correspondence}, devised in \cite{Igusa67bis}. To any ample symmetric Weil divisor $D$ defining a principal polarisation, one can associate bijectively a quadratic form $q_D$ from $A[2]$ to $\{ \pm 1 \}$ called \textit{even}, which means that the sum of its values on $A[2]$ is $2^{g}$ (\cite{Igusa67bis}, Theorem 2 and the previous arguments). On another side, the datum $\alpha_2$ also determines an even quadratic form $q_{\alpha_2}$, by associating to a $x \in A[2]$ with coordinates $(a,b) \in (\Z/2\Z)^{2g}$ in the basis $\alpha_2$ of $A[2]$ the value
\begin{equation}
\label{eqcorrespondanceIgusa}
q_{\alpha_2}(x) = (-1)^{a {}^t b}.
\end{equation}
We now only have to choose the unique ample symmetric divisor $D$ such that $q_D = q_{\alpha_2}$ and the line bundle $L$ associated to $D$.

By construction of this correspondence (\cite{Igusa67bis}, p. 823), a point $x \in A[2]$ of coordinates $(a,b) \in (\Z/2\Z)^{2g}$ in $\alpha_2$ automatically belongs to $\Theta_{A,L}$ (with $L$ associated to $(A,\lambda,\alpha_2)$) if $a{}^t b= 1 \mod 2$. A point of $A[2]$ with coordinates $(a,b)$ such that $a {}^t b = 0 \mod 2$ can also belong to $\Theta_{A,L}$ but with even multiplicity.

This allows us to get rid of the ambiguity of choice of an ample symmetric $L$ in the following, as soon as we have a symplectic level 2 structure (or finer) (
this result is a reformulation of Theorem 2 of \cite{Igusa67bis}).

\begin{defiprop}[Theta divisor canonically associated to a symplectic even level structure]
\label{defipropthetadiviseurcanonique}
\hspace*{\fill}

Let $n \geq 2$ even and $k$ algebraically closed such that $\carac(k)$ does not divide $n$.

For $(A,\lambda,\alpha_n)$ a principally polarised abelian variety of dimension $g$ with symplectic structure of level $n$ (Definition \ref{deficomplexabvar}), there is up to isomorphism an unique ample symmetric line bundle $L$ inducing $\lambda$ and associated by Igusa correspondence to the symplectic basis of $A[2]$ induced by $\alpha_n$. The \textit{theta divisor associated to} $(A,\lambda,\alpha_n)$, denoted by $\Theta_{A,\lambda,\alpha_n}$, is then the theta divisor associated to $(A,L)$, .

\end{defiprop}

The Runge-type theorem we give in section \ref{sectionapplicationsSiegel} (Theorem \ref{thmtubularRungegeneral}) focuses on principally polarised abelian surfaces $(A,\lambda)$ on a number field $K$ whose theta divisor does not contain any $n$-torsion point of $A$ (except 2-torsion points, as we will see it is automatic). This will imply (Proposition \ref{propnombrepointsdivthetajacobienne}) that $A$ is not a product of elliptic curves, but this is not a sufficient condition, as pointed out for example in \cite{BoxallGrant}.

We will once again start with the complex case to figure out how such a condition can be formulated on the moduli spaces, using complex theta functions (\cite{MumfordTata}, Chapter II).

\begin{defiprop}[Complex theta functions]
\hspace*{\fill}
\label{defipropcomplexthetafunctions}

Let $g \geq 1$.

The holomorphic function $\Theta$ on $\C^g \times \Hcal_g$ is defined by the series (convergent on any compact subset)
\begin{equation}
\label{eqdefserietheta}
\Theta(z,\tau) = \sum_{n \in \Z^g} e^{ i \pi n \tau {}^t n + 2 i \pi n {}^t z}.
\end{equation}

For any $a,b \in \R^g$, we also define the holomorphic function $\Theta_{a,b}$ by 
\begin{equation}
\label{eqdefseriethetaab}
\Theta_{a,b}(z,\tau) = \sum_{n \in \Z^g} e^{ i \pi (n+a) \tau {}^t (n+a) + 2 i \pi (n+a) {}^t (z+b)}.
\end{equation}

For a fixed $\tau \in \Hcal_g$, one defines $\Theta_\tau : z \mapsto \Theta(z,\tau)$ and similarly for $\Theta_{a,b,\tau}$. These functions have the following properties.

$(a)$ For every $a,b \in \Z^g$, 
\begin{equation}
\label{eqthetaabenfonctiontheta}
	\Theta_{a,b,\tau} (z) = e^{i \pi a \tau {}^t a + 2 i \pi a {}^t (z+b)} \Theta_\tau(z + a \tau + b).
\end{equation}

$(b)$ For every $p,q \in \Z^g$, 
\begin{equation}
\label{eqfoncthetaptranslation}
\Theta_{a,b,\tau}(z+p\tau + q) = e^{- i \pi p \tau {}^t p - 2 i \pi p{}^t z + 2 i \pi (a{}^t q - b {}^t p) } \Theta_{a,b,\tau} (z).
\end{equation}

$(c)$ Let us denote by $\vartheta$ and $\vartheta_{a,b}$ the \textit{normalised theta-constants},  which are the holomorphic functions on $\Hcal_g$ defined by 
\begin{equation}
\label{eqdefthetaconstantes}
\vartheta(\tau) : = \Theta(0,\tau) \quad {\textrm{and}} \quad \vartheta_{a,b} (\tau) := e^{ - i \pi a {}^t b} \Theta_{a,b} (0,\tau).
\end{equation}

These theta functions satisfy the following modularity property : with the notations of Definition \ref{deficomplexabvar},
\begin{equation}
\label{eqmodularitethetaconstantes}
\forall \gamma \in \Gamma_g(2), \quad \vartheta_{a,b} (\gamma \cdot \tau) = \zeta_8(\gamma) e^{ i \pi (a,b)^t V_\gamma} \sqrt{j_\gamma(\tau)} \vartheta_{(a,b)\gamma} (\tau),
\end{equation}
where $\zeta_8(\gamma)$ (a $8$-th root of unity) and $V_\gamma \in \Z^g$ only depend on $\gamma$ and the determination of the square root of $j_\gamma(\tau)$.

In particular, for every even $n \geq 2$, if $(na,nb) \in \Z^{2g}$, the function $\vartheta_{a,b}^{8n}$ is a Siegel modular form of degree $g$, level $n$ and weight $4n$, which only depends on $(a,b) \! \mod \Z^{2g}$.

\end{defiprop}

\begin{proof}
	The convergence of these series as well as their functional equations  \eqref{eqthetaabenfonctiontheta} and \eqref{eqfoncthetaptranslation} are classical and can be found in section II.1 of \cite{MumfordTata}.

	The modularity property \eqref{eqmodularitethetaconstantes} (also classical) is a particular case of the computations of section II.5 of \cite{MumfordTata} (we do not need here the general formula for $\gamma \in \Sp_{2g} (\Z)$).
	
	Finally, by natural computations of the series defining $\Theta_{a,b}$, one readily obtains that 
	\[
		\vartheta_{a+p,b+q} = e^{2 i \pi (a{}^t q - b{}^t p)} \vartheta_{a,b}. 
	\]
Therefore, if $(na,nb) \in \Z^{2g}$, the function $\vartheta_{a,b}^n$ only depends on $(a,b) \! \mod \Z^{2g}$. Now, putting the modularity formula \eqref{eqmodularitethetaconstantes} to the power $8n$, one eliminates the eight root of unity and if $\gamma \in \Gamma_g(n)$, one has $(a,b) \gamma = (a,b) \mod \Z^g$ hence $\vartheta_{a,b}^{8n}$ is a Siegel modular form of weight $4n$ for $\Gamma_g(n)$.	
\end{proof}

There is of course an explicit link between the theta functions and the notion of theta divisor, which we explain now with the notations of Definition \ref{deficomplexabvar}.

\begin{prop}[Theta divisor and theta functions]
\hspace*{\fill}
\label{propliendivthetafonctiontheta}

Let $\tau \in \Hcal_g$. 

The line bundle $L_\tau$ is ample and symmetric on $A_\tau$, and defines a principal polarisation on $A_\tau$. It is also the line bundle canonically associated to the 2-structure $\alpha_{\tau,2}$ and its polarisation by Igusa correspondence (Definition-Proposition \ref{defipropthetadiviseurcanonique}).

Furthermore, the global sections of $L_\tau$ canonically identify to the multiples of $\Theta_{\tau}$, hence the theta divisor associated to $(A_\tau, \lambda_\tau,\alpha_{\tau,2})$ is exactly the divisor of zeroes of $\Theta_{\tau}$ modulo $\Lambda_\tau$.

Thus, for every $a,b \in \R^g$, the projection of $\pi_\tau(a \tau + b)$ belongs to $\Theta_{A_\tau, \lambda_\tau, \alpha_{\tau,2}}$ if and only if $\vartheta_{a,b} (\tau)=0$.
	
\end{prop}

\begin{rem}
	The proof below that the $L_\tau$ is the line bundle associated to $(A_\tau, \lambda_\tau, \alpha_{\tau,2})$ is a bit technical, but one has to suspect that Igusa normalised its correspondence by \eqref{eqcorrespondanceIgusa} exactly to make it work.
\end{rem}

\begin{proof}

One can easily see that $L_\tau$ is symmetric by writing $[-1]^* L_\tau$ as a quotient of $\C^g \times \C$ by an action of $\Lambda_\tau$, then figuring out it is the same as \eqref{eqdeffibresurAtau}. Then, by simple connexity, the global sections of $L_\tau$ lift by the quotient morphism $\C^g \times \C \rightarrow L_\tau$ into functions $z \mapsto (z,f(z))$, and the holomorphic functions $f$ thus obtained are exactly the functions satisfying functional equation \eqref{eqfoncthetaptranslation} for $a=b=0$ because of \eqref{eqdeffibresurAtau}, hence the same functional equation as $\Theta_{\tau}$. This identification is also compatible with the associated divisors, hence $\Theta_{A_\tau,L_\tau}$ is the divisor of zeroes of $\Theta_{\tau}$ modulo $\Lambda_\tau$. For more details on the theta functions and line bundles, see (\cite{Debarre99}, Chapters IV,V and section VI.2).

We now have to check that Igusa correspondence indeed associates $L_\tau$ to $(A_\tau,\lambda_\tau,\alpha_{\tau,2})$. With the notations of the construction of this correspondence (\cite{Igusa67bis}, pp.822, 823 and 833), one sees that the meromorphic function $\psi_x$ on $A_\tau$ (depending on $L_\tau$) associated to $x \in A_\tau[2]$ has divisor $[2]^* T_x^* \Theta_{A_\tau,L_\tau} - [2]^* \Theta_{A_\tau,L_\tau}$, hence it is (up to a constant) the meromorphic function induced on $A_\tau$ by 
\[
	f_x(z) = \frac{\Theta_{a,b,\tau} (2z)}{\Theta_\tau(2z)} \quad {\textrm{where}} \, \, x= a \tau + b \mod \Lambda_\tau.
\]
Now, the quadratic form $q$ associated to $L_\tau$ is defined by the identity 
\[
	f_x(-z) = q(x) f_x(z)
\]
for every $z \in \C^g$, but $\Theta_\tau$ is even hence 
\[
	f_x(-z) = e^{4 i \pi a^t b} f_x(z)
\]
by formula \eqref{eqthetaabenfonctiontheta}. Now, the coordinates of $x$ in $\alpha_{\tau,2}$ are exactly $(2b,2a) \mod \Z^{2g}$ by definition, hence $q=q_{\alpha_{\tau,2}}$.

Let us finally make the explicit link between zeroes of theta-constants and theta divisors : using the argument above, the divisor of zeroes of $\Theta_\tau$ modulo $\Lambda_\tau$ is exactly $\Theta_{A_\tau,L_\tau}$, hence $\Theta_{A_\tau,\lambda_\tau,\alpha_{\tau,2}}$ by what we just proved for the Igusa correspondence.  This implies that for every $z \in \C^g$, $\Theta_\tau(z)=0$ if and only if $\pi_\tau(z)$ belongs to $\Theta_{A_\tau,\lambda_\tau, \alpha_{\tau,2}}$, and as $\vartheta_{a,b} (\tau)$ is a nonzero multiple of $\Theta(a \tau + b,\tau)$, we finally have that $\vartheta_{a,b}(\tau)=0$ if and only if $\pi_\tau(a\tau+b)$ belongs to $\Theta_{A_\tau,\lambda_\tau, \alpha_{\tau,2}}$.

\end{proof}

\section{Applications of the main result on a family of Siegel modular varieties}
\label{sectionapplicationsSiegel}

We now have almost enough definitions to state the problem which we will consider for our Runge-type result (Theorem \ref{thmtubularRungegeneral}). We consider theta divisors on abelian surfaces, and their torsion points.

\subsection{The specific situation for theta divisors on abelian surfaces}
\label{subsecthetadivabsur}
As an introduction and a preliminary result, let us treat first the case of theta divisors on elliptic curves.

\begin{lem}[Theta divisor on an elliptic curve]
\label{lemdivthetaCE}
\hspace*{\fill}

Let $E$ be an elliptic curve on an algebraically closed field $k$ with $\carac(k) \neq 2$ and $L$ an ample symmetric line bundle defining the principal polarisation on $E$. 

The effective divisor $\Theta_{E,L}$ is a 2-torsion point of $E$ with multiplicity one. More precisely, if $(e_1,e_2)$ is the basis of $E[2]$ associated by Igusa correspondence to $L$ (Definition-Proposition \ref{defipropthetadiviseurcanonique}),
\begin{equation}
	\label{eqdivthetaCEexplicite}
	\Theta_{E,L} = [e_1 + e_2].
\end{equation}
\end{lem}

\begin{rem}
	In the complex case, this can simply be obtained by proving that $\Theta_{1/2,1/2,\tau}$ is odd for every $\tau \in \Hcal_1$ hence cancels at 0, and has no other zeroes (by a residue theorem for example), then using Proposition \ref{propliendivthetafonctiontheta}.
\end{rem}

\begin{proof}
	By Riemann-Roch theorem on $E$, the divisor $\Theta_{E,L}$ is of degree 1 because $h^0(E,L)=1$ (and effective). Now, as explained before when discussing Igusa correspondence, for $a,b \in \Z$, $a e_1 + b e_2$ automatically belongs to $\Theta_{E,L}$ if $a b = 1 \mod 2 \Z$, hence $\Theta_{E,L} = [e_1+ e_2]$.
\end{proof}

This allows to use to describe the theta divisor of a product of two elliptic curves.
\begin{prop}[Theta divisor on a product of two elliptic curves]
\label{propdivthetaproduitCE}
\hspace*{\fill}

Let $k$ be an algebraically closed field with $\carac(k) \neq 2$.

Let $(A,L)$ with $A=E_1 \times E_2$ a product of elliptic curves on $k$ and $L$ an ample symmetric line bundle on $A$ inducing the product principal polarisation on $A$. 
The divisor $\Theta_{A,L}$ is then of the shape 
\begin{equation}
\label{eqdivthetaproduitCE}
\Theta_{A,L} = \{x_1\} \times E_2 + E_1 \times \{x_2\},
\end{equation}
with $x_i \in E_i[2]$ for $i=1,2$. In particular, this divisor has a (unique) singular point of multiplicity two at $(x_1,x_2)$, and :

$(a)$ There are exactly seven 2-torsion points of $A$ belonging to $\Theta_{A,L}$: the six points given by the coordinates $(a,b) \in (\Z/2\Z)^4$ such that $a{}^t b= 1$ in a basis giving $\Theta_{A,L}$ by Igusa correspondence, and the seventh point $(x_1,x_2)$.

$(b)$ For every even $n \geq 2$ which is nonzero in $k$, the number of $n$-torsion (but not $2$-torsion) points of $A$ belonging to $\Theta_{A,L}$ is exactly $2(n^2-4)$.
\end{prop}

\begin{proof}
	By construction of $(A,L)$, a global section of $(A,L)$ corresponds to a tensor product of global sections of $E_1$ and $E_2$ (with their principal polarisations), hence the shape of $\Theta_{A,L}$ is a consequence of Lemma \ref{lemdivthetaCE}.
	
	We readily deduce $(a)$ and $(b)$ from this shape, using that the intersection of the two components of $\Theta_{A,L}$ is a 2-torsion point of even multiplicity for the quadratic form hence different from the six other ones.
\end{proof}

To explain the result for abelian surfaces which are not products of elliptic curves, we recall below a fundamental result.

\begin{prop}[Shapes of principally polarised abelian surfaces]
\label{propsurfabnonproduitCEetdivtheta}
\hspace*{\fill}

	Let $k$ be any field.
	
	A principally polarised abelian surface $(A,\lambda)$ on $k$ is, after a finite extension of scalars, either the product of two elliptic curves (with its natural product polarisation), or the jacobian $J$ of an hyperelliptic curve $C$ of genus 2 (with its canonical principal polarisation). In the second case, for the Albanese embedding $\phi_x : C \rightarrow J$ with base-point $x$ and an ample symmetric line bundle $L$ on $K$ inducing $\lambda$, the divisor $\Theta_{J,L}$ is irreducible, and it is actually a translation of $\phi_x(C)$ by some point of $J(\overline{k})$.
	
\end{prop}

\begin{proof}
	This proposition (together with the dimension 3 case, for the curious reader) is the main topic of \cite{OortUeno} (remarkably, its proof starts with the complex case and geometric arguments before using scheme and descent techniques to extend it to all fields).
\end{proof}

Let us now fix an algebraically closed field $k$ with $\carac(k) \neq 2$.

Let $C$ be an hyperelliptic curve of genus 2, and $\iota$ its hyperelliptic involution. This curve has exactly six Weierstrass points (the fixed points of $\iota$, by definition), and we fix one of them, denoted by $\infty$. For the Albanese morphism $\phi_\infty$, the divisor $\phi_\infty(C)$ is stable by $[-1]$ because the divisor $[x] + [\iota(x)] - 2 [\infty]$ is principal for every $x \in C$. As $\Theta_{J,L}$ is also symmetric and a translation of $\phi_\infty(C)$, we know that $\Theta_{J,L} = T_x^* (\phi_\infty(C))$ for some $x \in J[2]$.

This tells us that understanding the points of $\Theta_{J,L}$ amounts to understanding how the curve $C$ behaves when embedded in its jacobian (in particular, how its points add). 
It is a difficult problem to know which torsion points of $J$ belong to the theta divisor (see \cite{BoxallGrant} for example), but we will only need to bound their quantity here, with the following result.

\begin{prop}
\hspace*{\fill}
\label{propnombrepointsdivthetajacobienne}

Let $k$ an algebraically closed field with $\carac(k) \neq 2$.

Let $C$ be an hyperelliptic curve of genus 2 on $k$ with jacobian $J$, and $\infty$ a fixed Weierstrass point of $C$.
We denote by $\widetilde{C}$ the image of $C$ in $J$ by the associated embedding $\phi_\infty : x \mapsto \overline{[x] - [\infty]}$.

$(a)$ The set $\widetilde{C}$ is stable by $[-1]$, and the application 
\[
	\fonctionsansnom{\operatorname{Sym}^2(\widetilde{C})}{J}{\{P,Q\}}{P+Q}
\]
is injective outside the fiber above 0.

$(b)$ There are exactly six 2-torsion points of $J$ belonging to $\widetilde{C}$, and they are equivalently the images of the Weierstrass points and the points of coordinates $(a,b) \in ((\Z/2\Z)^2)^2$ such that $a{}^t b= 1$ in a basis giving $\widetilde{C}$ by Igusa correspondence.

$(c)$ For any even $n \geq 2$ which is nonzero in $k$, the number of $n$-torsion points of $J$ belonging to $\widetilde{C}$ is bounded by $\sqrt{2} n^2 + \frac{1}{2}$.
\end{prop}

\begin{rem}
	This proposition is not exactly a new result, and its principle can be found (with slightly different formulations) in Theorem 1.3 of \cite{BoxallGrant} or  in Lemma 5.1 of \cite{Pazuki12}. For the latter, it is presented as a consequence on Abel-Jacobi theorem on $\C$, and we will here give a more detailed proof, which is also readily valid on any field. The problem of counting (or bounding) torsion points on the theta divisor has interested many people, e.g. \cite{BoxallGrant} and very recently \cite{Torsionthetadivisors} in general dimension. Notice that the results above give the expected bound in the case $g=2$, but we do not know how much we can lower the bound $\sqrt{2} n^2$ in the case of jacobians.
\end{rem}

\begin{proof}
As $[\infty]$ is a Weierstrass point, the divisor $2 [\infty]$ is canonical. Conversely, if a degree two divisor $D$ satisfies $\ell(D) := \dim H^0 (C,\Ocal_C(D)) \geq 2$, then it is canonical. Indeed, by Riemann-Roch theorem, this implies that $\ell(2 [\infty] - D) \geq 1$ but this divisor is of degree 0, hence it is principal and $D$ is canonical. Now, let $x,y,z,t$ be four points of $C$ such that $\phi_\infty(x) + \phi_\infty(y) = \phi_\infty(z) + \phi_\infty(t) $ in $J$. This implies that $[x] + [y] - [z] - [t]$ is the divisor of some function $f$, and then either $f$ is constant (i.e. $\{x,y\} = \{z,t\}$), either $\ell([z] + [t]) \geq 2$ hence $[z] + [t]$ is canonical by the argument above, and in this case the points $P =\phi_\infty(z)$ and $Q=\phi_\infty(t)=0$ of $\widetilde{C}$ satisfy $P+Q=0$ in $J$, which proves $(a)$.

Now, for $n \geq 2$ even, let us denote $\widetilde{C} [n] := \widetilde{C} \cap J[n]$. The summing map from $\widetilde{C} [n]^2$ to $J[n]$ has a fiber of cardinal $|\widetilde{C}[n]|$ above 0 and at most 2 above any other point of $J[n]$ by $(a)$, hence the inequality of degree two 
\[
	|\widetilde{C}[n]|^2 \leq |\widetilde{C}[n]| + 2 (n^4 - 1),
\]
from which we directly obtain $(c)$. In the case $n=2$, it is enough to see that $ 2 \phi_\infty(x) = 0$ if and only if $2 [x]$ is canonical if and only if $x$ is a Weierstrass divisor, which gives $(b)$.
\end{proof}

We can now define the divisors we will consider for our Runge-type theorem, with the following notation.

\textbf{Convention}
\hspace*{\fill}

Until the end of this article, the expression ``a couple $(a,b) \in (\Z /n \Z)^4$ (resp. $\Z^4, \Q^4$ )'' is a shorthand to designate the row vector with four coefficients where $a \in (\Z/n\Z)^2$ (resp. $\Z^2$, $\Q^2$ )  make up the first two coefficients and $b$ the last two coefficients.

\begin{defiprop}[Theta divisors on $A_2(n)^S_\C$]
\hspace*{\fill}
\label{defipropdivthetaA2ncomplexes}

Let $n \in \N_{\geq 2}$ even.

$(a)$ A couple $(a,b) \in (\Z/n\Z)^4$ is called \textit{regular} if it is \textit{not} of the shape $((n/2)a',(n/2)b')$ with $(a',b') \in ((\Z/2\Z)^2)^2$ such that $a' {}^t b' = 1 \mod 2$. There are exactly 6 couples $(a,b)$ not satisfying this condition, which we call \textit{singular}.

$(b)$ If $(a,b) \in (\Z/n\Z)^4$ is regular, for every lift $(\widetilde{a}, \widetilde{b}) \in \Z^4$ of $(a,b)$, the function $\vartheta_{\widetilde{a}/n, \widetilde{b}/n}^{8n}$ is a \textit{nonzero} Siegel modular form of degree 2, weight $4n$ and level $n$, independent of the choice of lifts. The \textit{theta divisor associated to $(a,b)$}, denoted by $(D_{n,a,b})_\C$, is the  Weil divisor of zeroes of this Siegel modular form on $A_2(n)^S_\C$.

$(c)$ For $(a,b)$ and $(a',b')$ regular couples in $(\Z/n\Z)^4$, the Weil divisors $(D_{n,a,b})_\C$ and $(D_{n,a',b'})_\C$ are equal if and only if $(a,b) = \pm (a',b')$. Hence, the set of regular couples defines exactly $n^4/2 + 2$ pairwise distinct Weil divisors.

\end{defiprop}

\begin{rem}
	The singular couples correspond to what are called \textit{odd characteristics} by Igusa. The proof below uses Fourier expansions to figure out which theta functions are nontrivial or proportional, but we conjecture the stronger result that $(D_{n,a,b})_\C$ and $(D_{n,a',b'})_\C$ are set-theoretically distinct (i.e. even without counting the multiplicities) unless $(a,b) = \pm (a',b')$. Such a result seems natural as the image of a curve into its jacobian should generically not have any other symmetry than $[-1]$, but we could not obtain it by looking at the simpler case (in $A_2(n)^S_\C$) of the products of elliptic curves: if $(a,b)$ and $(a',b')$ are both multiples of a primitive vector $v \in (1/n)\Z^4$, it is tedious but straighforward to see that the theta constants $\vartheta_{a,b}$ and $\vartheta_{a',b'}$ vanish on the same products of elliptic curves. Hence, to prove that the reduced divisors of $(D_{n,a,b})_\C$ and $(D_{n,a',b'})_\C$ are distinct unless $(a,b) = \pm (a',b')$, one needs to exhibit a curve $C$ whose jacobian isomorphic to $A_\tau$ contains $\pi_\tau(a \tau + b)$ but not $\pi_\tau(a' \tau + b')$ in its theta divisor.
	
	Notice that this will not be a problem for us later because all our arguments for Runge are set-theoretic, and Proposition \ref{propdivthetaproduitCE} and \ref{propnombrepointsdivthetajacobienne} are not modified if some of the divisors taken into account are equal.
\end{rem}

\begin{proof}[Proof of Definition-Proposition \ref{defipropdivthetaA2ncomplexes}]
	$(a)$ By construction, for any even $n \geq 2$, the number of singular couples $(a,b) \in (\Z/n\Z)^4$ is the number of couples $(a',b') \in (\Z/2\Z)^4$ such that $a' {}^t b' = 1 \mod 2$, and we readily see there are exactly six of them, namely 
	\[
	(0101), (1010), (1101), (1110), (1011) \textrm{  and  }(0111). 
	\]
	For $(b)$ and $(c)$, the modularity of the function comes from Definition-Proposition \ref{defipropcomplexthetafunctions} $(c)$ hence we only have to prove that it is nonzero when $(a,b)$ is regular. To do this, we will use the Fourier expansion of this modular form (for more details on Fourier expansions of Siegel modular forms, see chapter 4 of \cite{Klingen}), and simply prove that it has nonzero coefficients. This is also how we will prove the $\vartheta_{a,b}$ are distinct.
	
	To shorten the notations, given an initial couple $(a,b) \in (\Z/n\Z)^4$, we consider instead $(\tilde{a}/n, \tilde{b}/n) \in \Q^4$ for some lift $(\tilde{a}, \tilde{b})$ of $(a,b)$ in $\Z^4$) and by abuse of notation we renote it $(a,b)$ for simplicity. Regularity of the couple translates into the fact that $(a,b)$ is different from six possibles values modulo $\Z^4$, namely 
	\[
	\left(0 \frac{1}{2}0\frac{1}{2}\right),\left(\frac{1}{2}0\frac{1}{2}0 \right),\left(\frac{1}{2}\frac{1}{2}0\frac{1}{2}\right), \left(\frac{1}{2}\frac{1}{2}\frac{1}{2}0\right), \left(\frac{1}{2}0\frac{1}{2}\frac{1}{2}\right)\left(0\frac{1}{2}\frac{1}{2}\frac{1}{2}\right)
	\]
	by $(a)$, which we will assume now. We also fix $n \in \N$ even such that $(na,nb) \in \Z^4$. 
	
	Recall that 
	\begin{equation}
	\label{eqformulationsimplevarthetapourdevFourier}
		\vartheta_{a,b} (\tau) = e^{i \pi a^t b} \sum_{k \in \Z^2} e^{ i \pi (k+a) \tau{}^t(k+a) + 2 i \pi k^t b} 
	\end{equation}
by \eqref{eqthetaabenfonctiontheta} and \eqref{eqdefthetaconstantes}. Therefore, for any symmetric matrix $S \in M_2(\Z)$ such that $S/(2n^2)$ is half-integral (i.e. with integer coefficients on the diagonal, and half-integers otherwise), we have 
\[
	\forall \tau \in \Hcal_2, \quad \vartheta_{a,b} (\tau + S) = \vartheta_{a,b} (\tau),
\]
because for every $k \in \Z^2$, 
\[
	(k+a)S^t (k+a) \in 2 \Z.
\]
Hence, the function $\vartheta_{a,b}$ admits a Fourier expansion of the form 
\[
	\vartheta_{a,b} (\tau) = \sum_{T} a_T e^{ 2 i \pi \Tr(T\tau)},
\]
where $T$ runs through all the matrices of $S_2(\Q)$ such that $(2n^2) T$ is half-integral. This Fourier expansion is unique, because for any $\tau \in \Hcal_2$ and any $T$, we have
\[
	(2 n^2) a_T = \int_{[0,1]^4} \vartheta_{a,b} (\tau +x) e^{- 2 i \pi \Tr(T(\tau+x))} dx.
\]
In particular, the function $\vartheta_{a,b}$ is zero if and only if all its Fourier coefficients $a_T$ are zero, hence we will directly compute those, which are almost directly given by \eqref{eqformulationsimplevarthetapourdevFourier}. For $a=(a_1,a_2) \in \Q^2$ and $k=(k_1,k_2) \in \Z^2$, let us define 
\[
	T_{a,k} = \begin{pmatrix}
	          	(k_1+a_1)^2 & (k_1+a_1)(k_2 + a_2) \\ (k_1+a_1)(k_2+a_2) & (k_2+a_2)^2
	          \end{pmatrix},
\]
so that 
\begin{equation}
\label{eqpresquedevFouriervartheta}
	\vartheta_{a,b} (\tau) = e^{ i \pi a{}^t b} \sum_{k \in \Z^2} e^{ 2 i \pi k {}^t b} e^{ i \pi \Tr(T_{a,k} \tau)}
\end{equation}
by construction. It is not yet exactly the Fourier expansion, because we have to gather the $T_{a,k}$ giving the same matrix $T$ (and this is where we will use regularity). Clearly, 
\[
	T_{a,k} = T_{a',k'} \Llra (k+a) = \pm (k' +a').
\]
If $2a \notin \Z^2$, the function $k \mapsto T_{a,k}$ is injective, so \eqref{eqpresquedevFouriervartheta} is the Fourier expansion of $\vartheta_{a,b}$, with clearly nonzero coefficients, hence $\vartheta_{a,b}$ is nonzero.

If $2a = A \in \Z^2$, for every $k,k'\in \Z^2$, we have $(k+a) = \pm (k'+a)$ if and only if $k=k'$ or $k+k' = A$, so the Fourier expansion of $\vartheta_{a,b}$ is 
\begin{equation}
	\label{eqpdevFouriervarthetamauvaiscas}
	\vartheta_{a,b} (\tau) = \frac{e^{i \pi a^t b}}{2} \sum_{T} \sum_{\substack{k,k' \in \Z^2 \\ T_{k,a} = T_{k',a} = T}} (e^{ 2 i \pi k^t b} + e^{ 2 i \pi (-A-k)^t b}) e^{ i \pi \Tr(T \tau)}.
\end{equation}
Therefore, the coefficients of this Fourier expansion are all zero if and only if, for every $k \in \Z^2$, 
\[
	e^{ 2 i \pi (2 k + A)^t b} = -1,
\]
i.e. if and only if $b \in (1/2) \Z$ and $(-1)^{4 a^t b} = -1$, and this is exactly singularity of the couple $(a,b)$ which proves $(b)$.

Now, let $(a,b)$ and $(a',b')$ in $(1/n) \Z^4$ regular couples (translated in $\Q^4$ as above), such that $(na,nb)$ and $(na',nb')$ modulo $\Z^4$ have the same associated theta divisor on $A_2(n)^S_\C$. Then, the function 
\[
	\frac{\vartheta_{a,b}^{8n}}{\vartheta_{a',b'}^{8n}}
\]
induces a meromorphic function on $A_2(n)^S_\C$ whose divisor is $0$ hence a constant function, which implies that $\vartheta_{a,b} = \lambda \vartheta_{a',b'}$  for some $
\lambda \in \C^*$.
As these functions depend (up to a constant) only on $(a,b)$ and $(a',b') \! \mod \Z^4$, one can assume that all the coefficients of $(a,b)$ and $(a',b')$ belong to $[-1/2,1/2[$, and we assume first that $a,a' \notin (1/2)\Z^2$. Looking at the Fourier expansions \eqref{eqpresquedevFouriervartheta} gives that for every $k \in \Z^2$, 
\[
e^{i \pi a^t b + 2 i \pi k^t b} = \lambda e^{i \pi a'^t b' + 2 i \pi k^t b'}.
\]
Hence, we have $b= b' \mod \Z^2$ which in turns give $a= a' \mod \Z^2$ The same argument when $a$ or $a'$ belongs to $(1/2) \Z^2$ gives by \eqref{eqpdevFouriervarthetamauvaiscas} the possibilities $b=-b'$ and $a=-a' \mod \Z^4$.

Hence, we proved that if $\vartheta_{a,b}$ and $\vartheta_{a',b'}$ are proportional, then $(a,b)= \pm (a',b') \mod \Z^4$,and the converse is straightforward.
\end{proof}

These divisors have the following properties.

\begin{prop}[Properties of the $(D_{n,a,b})_\C$]
\hspace*{\fill}
\label{propproprietesDnabcomplexes}

Let $n \in \N_{\geq 2}$ even.

$(a)$ For every regular $(a,b) \in (\Z/n\Z)^4$, the divisor $(D_{n,a,b})_\C$ is ample.

$(b)$ For $n=2$, the ten divisors $(D_{2,a,b})_\C$ are set-theoretically pairwise disjoint outside the boundary $\partial A_2(2)_\C := A_2(2)^S_\C \backslash A_2(2)_\C$, and their union is exactly the set of moduli of  products of elliptic curves (with any symplectic basis of the 2-torsion).

$(c)$  For $(A,\lambda,\alpha_n)$ a principally polarised complex abelian surface with symplectic structure of level $n$ : 
\begin{itemize}
	\item If $(A,\lambda)$ is a product of elliptic curves, the moduli of $(A,\lambda,\alpha_n)$ belongs to exactly $n^2 - 3$ divisors $(D_{n,a,b})_\C$.
	\item Otherwise, the point $(A,\lambda,\alpha_n)$ belongs to at most $(\sqrt{2}/2)n^2 + 1/4$ divisors $(D_{n,a,b})_\C$.
\end{itemize}
\end{prop}
 
\begin{proof}
\hspace*{\fill}

$(a)$ The divisor $(D_{n,a,b})_\C$ is by definition the Weil divisor of zeroes of a Siegel modular form of order 2, weight $4n$ and level $n$, hence of a section of $L^{\otimes 4n}$ on $A_2(n)_\C^S$. As $L$ is ample on $A_2(n)^S_\C$ (Definition-Proposition \ref{defipropSatakecompactification} $(c)$), the divisor $(D_{n,a,b})_\C$ is ample.

Now, we know that every complex pair $(A,\lambda)$ is isomorphic to some $(A_\tau,\lambda_\tau)$ with $\tau \in \Hcal_2$ (Definition-Proposition \ref{defipropuniformcomplexabvar}). If $(A,\lambda)$ is a product of elliptic curves, the theta divisor of $(A,\lambda,\alpha_2)$ contains exactly seven 2-torsion points (Proposition \ref{propdivthetaproduitCE}), only one of comes from a regular pair, i.e. $(A,\lambda,\alpha_2)$ is contained in exactly one of the ten divisors. If $(A,\lambda)$ is not a product of elliptic curves, it is a jacobian (Proposition \ref{propsurfabnonproduitCEetdivtheta}) and the theta divisor of $(A,\lambda,\alpha_2)$ only contains the six points coming from singular pairs (Proposition \ref{propnombrepointsdivthetajacobienne}) i.e. $(A,\lambda,\alpha_2)$ does not belong to any of the ten divisors, which proves $(b)$.

To prove $(c)$, we use the same propositions for general $n$, keeping in mind that we only count as one the divisors coming from opposite values of $(a,b)$ : for products of elliptic curves, this gives $(2n^2 - 16)/2 + 7$ divisors (the 7 coming from the 2-torsion), and for jacobians, this gives $(\sqrt{2}/2)n^2 + 1/4$ (there are no nontrivial 2-torsion points to consider here).
\end{proof}

We will now give the natural divisors extending $(D_{n,a,b})_\C$ on the integral models $\Acal_2(n)$ (Definition-Proposition \ref{defipropalgmodulispaces}).

\begin{defi}
\hspace*{\fill}
\label{defidivthetasurespacemoduleentier}

Let $n \in \N_{\geq 2}$ even.

For every regular $(a,b) \in (\Z/n\Z)^4$, the divisor $(D_{n,a,b})_\C$ is the geometric fiber at $\C$ of an effective Weil divisor $D_{n,a,b}$ on $\Acal_2(n)$, such that the moduli of a triple $(A,\lambda,\alpha_n)$ (on a field $k$ of characteristic prime to $n$) belongs to $D_{n,a,b} (k)$ if and only if the point of $A[n](\overline{k})$ of coordinates $(a,b)$ for $\alpha_n$ belongs to the theta divisor $\Theta_{A,\lambda,\alpha_n}$ (Definition-Proposition \ref{defipropthetadiviseurcanonique}).

\end{defi}

\begin{proof}
\hspace*{\fill}

This amounts to giving an algebraic construction of the $D_{n,a,b}$ satisfying the wanted properties. The following arguments are extracted from Remark I.5.2 of \cite{ChaiFaltings}. Let $\pi : A \rightarrow S$ an abelian scheme and $\Lcal$ a symmetric invertible sheaf on $A$, relatively ample on $S$ and inducing a principal polarisation on $A$. If $s : S \rightarrow A$ is a section of $A$ on $S$, the evaluation at $s$ induces an $\Ocal_S$-module isomorphism between $\pi_*\Lcal$ and $s^*\Lcal$. Now, if $s$ is of $n$-torsion in $A$, for $e : S \rightarrow A$ the zero section, the sheaf $(s^* \Lcal)^{\otimes 2n}$ is isomorphic to $(e^* \Lcal)^{\otimes 2n}$, i.e. trivial. We denote by $\omega_{A/S}$ the invertible sheaf on $S$ obtained as the determinant of the sheaf of invariant differential forms on $A$, and the computations of Theorem I.5.1 and Remark I.5.2 of \cite{ChaiFaltings} give $8 \pi_* \Lcal = - 4 \omega_{A/S}$ in $\Pic(A/S)$. Consequenltly, the evaluation at $s$ defines (after a choice of trivialisation of $(e^* \Lcal)^{\otimes 2n}$ and putting to the power $8n$) a section of $\omega_{A/S}^ {\otimes 4n}$. Applying this result on the universal abelian scheme (stack if $n \leq 2$) $\Xcal_2(n)$ on $\Acal_2(n)$ , for every $(a,b) \in (\Z/n\Z)^4$, the section defined by the point of coordinate $(a,b)$ for the $n$-structure on $\Xcal_2(n)$ induces a global section $s_{a,b}$ of $\omega_{\Xcal_2(n)/\Acal_2(n)}^{\otimes 4n}$, and we define $D_{n,a,b}$ as the Weil divisor of zeroes of this section. It remains to check that it satisfies the good properties. 

Let $(A,\lambda,\alpha_n)$ be a triple over a field $k$ of characteristic prime to $n$, and $L$ the ample line bundle associated to it by Definition-Proposition \ref{defipropthetadiviseurcanonique}. By construction, its moduli belongs to $D_{n,a,b}$ if and only if the unique (up to constant) nonzero section vanishes at the point of $A[n]$ of coordinates $(a,b)$ in $\alpha_n$, hence if and only if this point belongs to $\Theta_{A,\lambda,\alpha_n}$.

Finally, we see that the process described above applied to the universal abelian variety $\Xcal_2(n)_\C$ of $\Acal_2(n)_\C$ (by means of explicit description of the line bundles as quotients) gives (up to invertible holomorphic functions) the functions $\vartheta_{\widetilde{a}/n,\widetilde{b}/n}^{8n}$, which proves that $(D_{n,a,b})_\C$ is indeed the geometric fiber of $D_{n,a,b}$ (it is easier to see that their complex points are the same, by Proposition \ref{propproprietesDnabcomplexes} $(c)$ and the above characterisation applied to the field $\C$).

If one does not want to use stacks for $n=2$, one can consider for $(a,b) \in (\Z/2\Z)^4$ the divisor $D_{4,2a,2b}$ which is the pullback of $D_{2,a,b}$ by the degeneracy morphism $A_2(4) \rightarrow A_2(2)$.
\end{proof}

\subsection{Tubular Runge theorems for abelian surfaces and their theta divisors}
\label{subsectubularRungethmabsur}
We can now prove a family of tubular Runge theorems for to the theta divisors $D_{n,a,b}$ (for even $n \geq 2$).

We will state the case $n=2$ first because its moduli interpretation is easier but the proofs are the same, as we explain below.

In the following results, the \textit{boundary} of $A_2(n)^S_\C$ is defined as $\partial A_2(n)^S_\C := A_2(n)^S_\C \backslash A_2(n)_\C$. 

\begin{thm}[Tubular Runge for products of elliptic curves on $\Acal_2(2)^S$]
\hspace*{\fill}
\label{thmtubularRungeproduitCE}

	Let $U$ be an open neighbourhood of $\partial A_2(2)^S_\C$ in $A_2(2)^S_\C$ for the natural complex topology.
	
	For any such $U$, we define $\Ecal(U)$ the set of moduli $P$ of triples $(A,\lambda,\alpha_2)$ in $\Acal_2(2) (\Qb)$ such that (choosing $L$ a number field of definition of the moduli) : 
	\vspace{-0.2cm}
	\begin{itemize}
		\item The abelian surface $A$ has potentially good reduction at every finite place $w \in M_L$ (\textbf{tubular condition for finite places}).
		\vspace{-0.2cm}
		\item For any embedding $\sigma : L \rightarrow \C$, the image $P_\sigma$ of $P$ in $\Acal_2(2)_\C$ is outside of $U$ (\textbf{tubular condition for archimedean places}).
		\vspace{-0.2cm}
		\item The number $s_L$ of non-integrality places of $P$, i.e. places $w \in M_L$ such that 
		\begin{itemize}
		\vspace{-0.2cm}
		\item either $w$ is above $M_L^{\infty}$ or $2$,
		\vspace{-0.2cm}
		\item or the semistable reduction modulo $w$ of $(A,\lambda)$ is a product of elliptic curves
		\end{itemize}
		\vspace{-0.2cm}
		satisfies the \textbf{tubular Runge condition} 
		\vspace{-0.2cm}
		\[
			s_L < 10.
		\]
	\end{itemize}
	\vspace{-0.2cm}
	Then, for every choice of $U$, the set $\Ecal(U)$ is \textbf{finite}.
\end{thm}

\begin{thm}[Tubular Runge for theta divisors on $\Acal_2(n)^S$]
\hspace*{\fill}
\label{thmtubularRungegeneral}

Let $n \geq 4$ even.

	Let $U$ be an open neighbourhood of $\partial A_2(n)^S_\C$ in $A_2(n)^S_\C$ for the natural complex topology.
	
	For any such $U$, we define $\Ecal(U)$ the set of moduli $P$ of triples $(A,\lambda,\alpha_2)$ in $\Acal_2(n) (\Qb)$ such that (choosing $L \supset \Q(\zeta_n)$ a number field of definition of the triple) : 
	\vspace{-0.2cm}
	\begin{itemize}
		\item The abelian surface $A$ has potentially good reduction at every place $w \in M_L^{\infty}$ (\textbf{tubular condition for finite places}).
		\vspace{-0.2cm}
		\item For any embedding $\sigma : L \rightarrow \C$, the image $P_\sigma$ of $P$ in $\Acal_2(n)_\C$ is outside of $U$ (\textbf{tubular condition for archimedean places}).
		\vspace{-0.2cm}
		\item The number $s_P$ of non-integrality places of $P$, i.e. places $w \in M_L$ such that 
		\vspace{-0.2cm}
		\begin{itemize}
		\vspace{-0.2cm}
		\item either $w$ is above $M_L^{\infty}$ or a prime factor of $n$,
		\vspace{-0.2cm}
		\item or the theta divisor of the semistable reduction modulo $w$ of $(A,\lambda,\alpha_n)$ contains an $n$-torsion point which is not one of the six points coming from odd characteristics,
		\end{itemize}
		\vspace{-0.2cm}
		satisfies the \textbf{tubular Runge condition} 
		\[
			(n^2 - 3) s_P < \frac{n^4}{2} + 2.
		\]
	\end{itemize}
	\vspace{-0.2cm}
	Then, for every choice of $U$, the set of points $\Ecal(U)$ is \textbf{finite}.
\end{thm}

\begin{rem}

We put an emphasis on the conditions given in the theorem to make it easier to identify how it is an application of our main result, Theorem \ref{thmRungetubulaire}. The tubular conditions (archimedean and finite) mean that our points $P$ do not belong to some tubular neighbourhood $\Vcal$ of the boundary. We of course chose the boundary as our closed subset to exclude because of its modular interpretation for finite places. The places above $M_L^{\infty}$ or a prime factor of $n$ are automatically of non-integrality for our divisors because the model $\Acal_2(n)$ is not defined at these places. Finally, the second possibility to be a place of non-integrality straightforwardly comes from the moduli interpretation of the divisors $D_{n,a,b}$ (Definition \ref{defidivthetasurespacemoduleentier}). All this is detailed in the proof below.

To give an example of how we can obtain an explicit result in practice, we prove in section \ref{sectionexplicitRunge} an explicit (and even theoretically better) version of Theorem \ref{thmtubularRungeproduitCE}.

It would be more satisfying (and easier to express) to give a tubular Runge theorem for which the divisors considered are exactly the irreducible components parametrising the products of elliptic curves. Unfortunately, except for $n=2$, there is a serious obstruction because those divisors are not ample, and there are even reasons to suspect they are not big. We have explained in Remark \ref{remampledifficilepourA2} why proving the ampleness for general divisors on $A_2(n)^S_\C$ is difficult.

It would also be morally satisfying to give a better interpretation of the moduli of $D_{n,a,b}$ for $n >2$, i.e. not in terms of the theta divisor, but maybe of the structure of the abelian surface if possible (nontrivial endomorphisms ? isogenous to products of elliptic curves ?). As far as the author knows, the understanding of abelian surfaces admitting some nontrivial torsion points on their theta divisor is still very limited.

Finally, to give an idea of the margin the tubular Runge condition gives for $n>2$ (in terms of the number of places which are not ``taken'' by the automatic bad places), we can easily see that the number of places of $\Q(\zeta_n)$ which are archimedean or above a prime factor of $n$ is less than $n/2$. Hence, we can find examples of extensions $L$ of $\Q(\zeta_n)$ of degree $n$ such that some points defined on it still can satisfy tubular Runge condition. This is also where using the full strength of tubular Runge theorem is crucial: for $n=2$, one can compute that some points of the boundary are contained in 6 different divisors $D_{2,a,b}$, and for general even $n$, a similar analysis gives that the intersection number $m_{\emptyset}$ is quartic in $n$, which leaves a lot less margin for the places of non-integrality (or even none at all). 
\end{rem}

\begin{proof}[Proof of Theorems \ref{thmtubularRungeproduitCE} and \ref{thmtubularRungegeneral}]
\hspace*{\fill}

As announced, this result is an application of the tubular Runge theorem (Theorem \ref{thmRungetubulaire}) to $\Acal_2(n)^S_{\Q(\zeta_n)}$ (Definition-Proposition \ref{defipropalgmodulispaces}) and the divisors $D_{n,a,b}$ (Definition  \ref{defidivthetasurespacemoduleentier}), whose properties will be used without specific mention. We reuse the notations of the hypotheses of Theorem \ref{thmRungetubulaire} to explain carefully how it is applied.

\textbf{\textit{(H0)}} The field of definition of $A_2(n)^S_\C$ is $\Q(\zeta_n)$, and the ring over which our model $\Acal_2(n)^S$ is built is $\Z [ \zeta_n, 1/n]$, hence $S_0$ is made up with all the archimedean places and the places above prime factors of $n$. There is no need for a finite extension here as all the $D_{n,a,b}$ are divisors on $\Acal_2(n)^S$.

\textbf{\textit{(H1)}} The model $\Acal_2(n)^S_\C$ is indeed normal projective, and we know that the $D_{n,a,b}$ are effective Weil divisors hence Cartier divisors up to multiplication by some constant by Proposition \ref{proprationalPicardSiegel}. For any finite extension $L$ of $\Q(\zeta_n)$, the number of orbits $r_L$ is the number of divisors $D_{n,a,b}$ (as they are divisors on the base model), i.e. $n^4/2 + 2$ (Proposition \ref{propproprietesDnabcomplexes} $(c)$).

\textbf{\textit{(H2)}} The chosen closed subset $Y$ of $\Acal_2(n)^S_\Q(\zeta_n)$ is the boundary, namely 
\[
\partial \Acal_2(n)^S_{\Q(\zeta_n)} = \Acal_2(n)^S_{\Q(\zeta_n)} \backslash \Acal_2(n)_{\Q(\zeta_n)}.
\]
We have to prove that the tubular conditions given above correspond to a tubular neighbourhood. To do this, let $\Ycal$ be the boundary $\Acal_2(n)^S \backslash \Acal_2(n)$ and $g_1, \cdots, g_s$ homogeneous generators of the ideal of definition of $\Ycal$ after having fixed a projective embedding of $\Acal_2(n)$. Let us find an $M_{\Q(\zeta_n)}$-constant such that $\Ecal(U)$ is included in the tubular neighbourhood of $\partial \Acal_2(n)^S_\Q(\zeta_n)$ in $A_2(n)^S_{\Q(\zeta_n)}$ associated to $\Ccal$ and $g_1, \cdots, g_k$. For the places $w$ not above $M_L^\infty$ or a prime factor of $n$, the fact that $P = (A,\lambda,\alpha_n)$ does not reduce in $Y$ modulo $w$ is exactly equivalent to $A$ having potentially good reduction at $w$ hence we can choose $c_v= 0$ for the places $v$ of $\Q(\zeta_n)$ not archimedean and not dividing $n$. For archimedean places, belonging to $U$ for an embedding $\sigma : L \rightarrow \C$ implies that $g_1, \cdots, g_n$ are small, and we just have to choose $c_v$ stricly larger than the maximum of the norms of the $g_i(U \cap V_j)$ (in the natural affine covering $(V_j)_j$ of the projective space), independant of the choice of $v \in M_{\Q(\zeta_n)}^\infty$. Finally, we have to consider the case of places above a prime factor of $n$. To do this, we only have to recall that having potentially good reduction can be given by integrality of some quotients of the Igusa invariants at finite places, and these invariants are modular forms on $\Gamma_2(1)$. We can add those who vanish on the boundary to the homogeneous generators $g_1, \cdots, g_n$ and consider $c_v=0$ for these places as well. This is explicitly done in part \ref{subsecplacesabove2} for $A_2(2)$.

\textbf{\textit{(TRC)}} As said before, there are $n^4/2 + 2$ divisors considered, and their generic fibers are ample by Proposition \ref{propproprietesDnabcomplexes}. Furthermore,  by Propositions \ref{propdivthetaproduitCE} and \ref{propnombrepointsdivthetajacobienne}, outside the boundary, at most $(n^2 - 3)$ can have nonempty common intersection, and this exact number is attained only for products of elliptic curves, (as $n^2 - 3 = 2(n^2 - 4)/2 + 1$, separating the regular 2-torsion pairs and regular non-2-torsion pairs up to $\pm 1$).

This gives the tubular Runge condition 
\[
	(n^2 - 3) s_L < n^4/2 + 2,
\]
which concludes the proof.

For $n=2$, the union of the ten $D_{2,a,b}$ is made up with the moduli of products of elliptic curves, and they are pairwise disjoint outside $\partial A_2(2)$ (Proposition \ref{propproprietesDnabcomplexes} $(b)$), hence the simply-expressed condition $s_L<10$ in this case.
\end{proof}

\section{The explicit Runge result for level two}
\label{sectionexplicitRunge}

To finish this paper, we improve and make explicit the finiteness result of Theorem \ref{thmtubularRungeproduitCE}, as a proof of principle of the method.

Before stating Theorem \ref{thmproduitCEexplicite}, we need some notations. In level two, the auxiliary functions are deduced from the ten even theta constants of characteristic two, namely the functions $\Theta_{m/2} (\tau)$ (notation \eqref{eqdefseriethetaab}), with the quadruples $m$ going through 
\begin{equation}
\label{eqevenchartheta}
E = \{(0000),(0001),(0010),(0011),(0100),(0110),(1000),(1001),(1100),(1111) \}
\end{equation}
(see subsections \ref{subsecthetadivabvar} and \ref{subsecthetadivabsur} for details).
We recall (\cite{vdG82}, Theorem 5.2) that these functions define an embedding
\begin{equation}
\label{eqdefplongementpsi}
\fonction{\psi}{A_2(2)}{\P^9}{\overline{\tau}}{(\Theta_{m/2}^4 (\tau))_{m \in E}}
\end{equation}
which induces an isomorphism between $A_2(2)^S_\C$ and the subvariety of $\P^9$ (with coordinates indexed by $m \in E$) defined by the linear equations 
\begin{eqnarray}
x_{1000} - x_{1100} + x_{1111} - x_{1001} & = & 0 \\
x_{0000} - x_{0001} - x_{0110} - x_{1100} & = & 0 \\
x_{0110} - x_{0010} + x_{1111} + x_{0011} & = & 0 \\
x_{0100} - x_{0000} + x_{1001} + x_{0011} & = & 0 \\
x_{0100} - x_{1000} + x_{0001} - x_{0010} & = & 0 
\end{eqnarray}
(which makes it a subvariety of $\P^4$) together with the quartic equation 
\begin{equation}
\left( \sum_{m \in E} x_m^2 \right)^2 - 4 \sum_{m \in E} x_m^4 = 0.
\end{equation}

\begin{rem}
	For the attentive reader, the first linear equation has sign $(+1)$ in $x_{1111}$ whereas it is $(-1)$ in \cite{vdG82}, as there seems to be a typographic mistake there : we have realised it during our computations on Sage in part \ref{subsecplacesabove2} and found the right sign back from Igusa's relations (\cite{Igusa64bis}, Lemma 1 combined with the proof of Theorem 1). 
\end{rem}

 There is a natural definition for a tubular neighbourhood of $Y = \partial A_2(2)$: for a  finite place $v$, as in Theorem \ref{thmtubularRungeproduitCE},  we choose $V_v$ as the set of triples $P = \overline{(A,\lambda,\alpha_2)}$ where $A$ has potentially bad reduction modulo $v$. To complete it with archimedean places, we use the classical fundamental domain for the action of $\Sp_4(\Z)$ on $\Hcal_2$ denoted by $\Fcal_2$ (see \cite{Klingen}, section I.2 for details). Given some parameter $t \geq \sqrt{3}/2$, the neighbourhood $V(t)$ of $\partial A_2(2)_\C^S$ in $A_2(2)^S_\C$ is made up with the points $P$ whose lift $\tau$ in $\Fcal_2$ (for the usual quotient morphism $\Hcal_2 \rightarrow A_2(1)_\C$) satisfies $\Im (\tau_4) \geq t$, where $\tau_4$ is the lower-right coefficient of $\tau$. We choose $V(t)$ as the archimedean component of the tubular neighbourhood for every archimedean place. The reader knowledgeable with the construction of Satake compactification will have already seen such neighbourhoods of the boundary.

 Notice that for a point $P=\overline{(A,\lambda,\alpha_2)} \in A_2(2)(K)$, the abelian surface $A$ is only defined over a finite extension $L$ of $K$, but for prime ideals $\gP_1$ and $\gP_2$ of $\Ocal_L$ above the same prime ideal $\gP$ of $\Ocal_K$, the reductions of $A$ modulo $\gP_1$ and $\gP_2$ are of the same type because $P \in A_2(2) (K)$. This justifies what we mean by ``semistable reduction of $A$ modulo $\gP$'' below.

\begin{thm}
\label{thmproduitCEexplicite}
Let $K$ be a number field and $P=\overline{(A,\lambda,\alpha_2)} \in A_2(2)(K)$ where $A$ has potentially good reduction at every finite place.

Let $s_P$ be the number of prime ideals $\gP$ of $\Ocal_K$ such that the semistable reduction of $A$ modulo $\gP$ is a product of elliptic curves. We denote by $h_\Fcal$ the stable Faltings height of $A$.

$(a)$ If $K=\Q$ or an imaginary quadratic field and 
\[
	|s_P|<4
\]
then
\[
	h(\psi(P)) \leq 10.75, \quad h_\Fcal(A) \leq 1070.
\]

$(b)$ Let $t \geq \sqrt{3}/2$ be a real number. If for any embedding $\sigma : K \rightarrow \C$, the point $P_\sigma \in A_2(2)_\C$ does not belong to $V(t)$, and 
 \[
|s_P| + |M_K^{\infty}| < 10
\]
then
\[
                                                                                                   h(\psi(P)) \leq 4 \pi t + 6.14, \quad h_\Fcal(A) \leq 2 \pi t + 535 \log(2 \pi t + 9)           		
\]
\end{thm}

The Runge condition for $(b)$ is a straightforward application of our tubular Runge theorem. For $(a)$, we did not assume anything on the point $P$ at the (unique) archimedean place, which eliminates six divisors when applying Runge's method here, hence the different Runge condition here (see Remark \ref{remRungetubulaire} $(b)$).

The principle of proof is very simple: we apply Runge's method to bound the height of $\psi(P)$ when $P$ satisfies the conditions of Theorem \ref{thmtubularRungeproduitCE}, and using the link between this height and Faltings height given in (\cite{Pazuki12b}, Corollary 1.3), we know we will obtain a bound of the shape 
\[
	h_\Fcal(P) \leq f(t)
\]
where $f$ is an explicit function of $t$, for every point $P$ satisfying the conditions of Theorem \ref{thmtubularRungeproduitCE}.

At the places of good reduction not dividing 2, the contribution to the height is easy to compute thanks to the theory of algebraic theta functions devised in \cite{Mumford66} and \cite{Mumford67}. The theory will be sketched in part \ref{subsecalgebraicthetafunctions}, resulting in Proposition \ref{propalgthetafoncetreduchorsde2}. 

For the archimedean places, preexisting estimates due to Streng for Fourier expansions on each of the ten theta functions allow to make explicit how only one of them can be too small compared to the others, when we are out of $V(t)$. This is the topic of part \ref{subsecarchimedeanplaces}.

For the places above 2, the theory of algebraic theta functions cannot be applied. To bypass the problem, we use Igusa invariants (which behave in a well-known fashion for reduction in any characteristic) and prove that the theta functions are algebraic and ``almost integral'' on the ring of these Igusa invariants, with explicit coefficients. Combining these two facts in part \ref{subsecplacesabove2}, we will obtain Proposition \ref{propbornesfoncthetaaudessus2}, a less-sharp avatar of Proposition \ref{propalgthetafoncetreduchorsde2}, but explicit nonetheless. 

Finally, we put together these estimates in part \ref{subsecfinalresultRungeCEexplicite} and obtain the stated bounds on $h \circ \psi$ and the Faltings height.

\subsection{Algebraic theta functions and the places of potentially good reduction outside of 2}
\label{subsecalgebraicthetafunctions}

The goal of this part  is the following result.

\begin{prop}
\label{propalgthetafoncetreduchorsde2}
Let $K$ be a number field and $\gP$ a maximal ideal of $\Ocal_K$, of residue field $k(\gP)$ with characteristic different from 2.
	Let $P = \overline{(A, \lambda,\alpha_2)} \in A_2(2)(K)$. Then, $\psi(P) \in \P^9(K)$ and : 
	
	$(a)$ If the semistable reduction of $A$ modulo $\gP$ is a product of elliptic curves, the reduction of $\psi(P)$ modulo $\gP$ has exactly one zero coordinate, in other words every coordinate of $\psi(P)$ has the same $\gP$-adic norm except one which is strictly smaller.
	
	$(b)$ If the semistable reduction of $A$ modulo $\gP$ is a jacobian of hyperelliptic curve, the reduction of $\psi(P)$ modulo $\gP$ has no zero coordinate, in other words every coordinate of $\psi(P)$ has the same $\gP$-adic norm.
\end{prop}

To link $\psi(P)$ with the intrinsic behaviour of $A$, we use the theory of algebraic theta functions, devised in \cite{Mumford66} and \cite{Mumford67} (see also \cite{DavidPhilippon} and \cite{Pazuki12b}). As it is not very useful nor enlightening to go into detail or repeat known results, we only mention them briefly here.
In the following, $A$ is an abelian variety of dimension $g$ over a field $k$ and $L$ an ample symmetric line bundle on $A$ inducing a principal polarisation $\lambda$. We also fix $n \geq 2$ even, assuming that all the points of $2n$-torsion of $A$ are defined over $k$ and $\carac(k)$ does not divide $n$ (in particular, we always assume $\carac(k) \neq 2$). Let us denote formally the Heisenberg group $\Gcal(\underline{n})$ as the set
\[
	\Gcal(\underline{n}) := k^* \times (\Z/n\Z)^g \times (\Z/n\Z)^g
\]
equipped with the group law 
\[
	(\alpha,a,b) \cdot (\alpha',a',b') := (\alpha \alpha' e^{\frac{2 i \pi}{n} a{}^t b'},a+a',b+b')
\]
(contrary to the convention of \cite{Mumford66}, p.294, we identified the dual of $(\Z/n\Z)^g$ with itself). Recall that $A[n]$ is exactly the group of elements of $A(\overline{k})$ such that $T_x^* (L^{\otimes n}) \cong L ^{\otimes n}$ : indeed, it is the kernel of the morphism $\lambda_{L^{\otimes n}} = n \lambda$ from $A$ to $\widehat{A}$ (see proof of Proposition \ref{propambiguitedivthetaAL}).

\begin{proof}

	Given the datum of a \textit{theta structure} on $L^{\otimes n}$, i.e. an isomorphism $\beta: \Gcal(L^{\otimes n}) \cong \Gcal(\underline{n})$ which is the identity on $k^*$ (see \cite{Mumford66}, p. 289 for the definition of $\Gcal(L^{\otimes n})$), one has a natural action of $\Gcal(\underline{n})$ on $\Gamma(A,L^{\otimes n})$ (consequence of Proposition 3 and Theorem 2 of \cite{Mumford66}), hence for $n \geq 4$ the following  projective embedding of $A$ : 
	\begin{equation}
	\label{eqplongementA}
		\fonction{\psi_\beta}{A}{\P^{n^{2g} - 1}_k}{x}{\left( ((1,a,b)\cdot ( s_0^{\otimes n})) (x) \right)_{a,b \in (\Z/n\Z)^g}},
	\end{equation}
where $s_0$ is a nonzero section of $\Gamma(A,L)$, hence unique up to multiplicative scalar (therefore $\psi_\beta$ only depends on $\beta$). This embedding is not exactly the same as the one defined in (\cite{Mumford66}, p. 298) (it has more coordinates), but the principle does not change at all. One calls \textit{Mumford coordinates of $(A,L)$ associated to $\beta$} the projective point $\psi_\beta(0) \in \P^{n^{2g-1}}(k)$.

	Now, one has the following commutative diagram whose rows are canonical exact sequences (\cite{Mumford66}, Corollary of Theorem 1)
	\[
		\xymatrix{
		0 \ar[r] & k^* \ar[d]^{=} \ar[r] & \Gcal(L^{\otimes n}) \ar[r] \ar[d]^{\beta}& A[n] \ar[d]^{\alpha_n} \ar[r] & 0 \\
		0 \ar[r] & k^* \ar[r] & \Gcal(\underline{n}) \ar[r] & (\Z/n\Z)^{2g} \ar[r] & 0,
		}
	\]
where $\alpha_n$ is a symplectic level $n$ structure on $A[n]$ (Definition \ref{defibaseabvar}), called \textit{the symplectic level $n$ structure induced by $\beta$}. Moreover, for every $x \in A(k)$, the coordinates of $\psi_\beta (x)$ are (up to constant values for each coordinate, only depending on $\beta$) the $\vartheta_{A,L} ([n] x +\alpha_{n}^{-1} (a,b))$ (see Definition \ref{defithetadivisorabvar}). In particular, for any $a,b \in (\Z/n \Z)^g$, 
\begin{equation}
\label{eqliencoordonneesMumfordetdivtheta}
\psi_\beta(0)_{a,b} = 0 \Leftrightarrow \alpha_n^{-1} (a,b) \in \Theta_{A,L}.
\end{equation}
Furthermore, for two theta structures $\beta,\beta'$ on $[n]^* L$ inducing $\alpha_n$, one sees that $\beta' \circ \beta^{-1}$ is of the shape $(\alpha,a,b) \mapsto (\alpha \cdot f(a,b),a,b)$, where $f$ has values in $n$-th roots of unity, hence $\psi_\beta$ and $\psi_{\beta'}$ only differ multiplicatively by $n$-th roots of unity. 

      Conversely, given the datum of a symplectic structure $\alpha_{2n}$ on $A[2n]$, there exists an unique \textit{symmetric theta structure} on $[n]^* L$ which is \textit{compatible} with some symmetric theta structure on $[2n]^* L$ inducing $\alpha_{2n}$ (\cite{Mumford66}, p.317 and Remark 3 p.319). We call it the \textit{theta structure on $[n]^* L$ induced by $\alpha_{2n}$}. Thus, we just proved that the datum of a symmetric theta structure on $[n]^*L$ is intermediary between a level $2n$ symplectic structure and a level $n$ symplectic structure (the exact congruence group is easily identified as $\Gamma_g(n,2n)$ with the notations of \cite{Igusa66}).

 Now, for a triple $(A,L,\alpha_{2n})$ (notations of subsection \ref{subsecabvarSiegelmodvar}), when $A$ is a complex abelian variety, there exists $\tau \in \Hcal_g$ such that this triple is isomorphic to $(A_\tau,L_\tau,\alpha_{\tau,2n})$ (Definition-Proposition \ref{defipropuniformcomplexabvar}). By definition of $L_\tau$ as a quotient \eqref{eqdeffibresurAtau}, the sections of $L_\tau ^{\otimes n}$ canonically identify to holomorphic functions $\vartheta$ on $\C^g$ such that 
      \begin{equation}
      \label{eqsectionsLtaupuissancen}
      \forall p,q \in \Z^g, \forall z \in \C^g, \quad \vartheta(z + p \tau + q) = e^{ - i \pi n \tau ^t n - 2 i \pi n ^t z} \vartheta(z),
      \end{equation}
and through this identification one sees (after some tedious computations) that the symmetric theta structure $\beta_\tau$ on $L_\tau^{\otimes n}$ induced by $\alpha_{\tau,2n}$ acts by 
\[
	((\alpha,a,b) \cdot \vartheta) (z) = \alpha \exp\left( \frac{i \pi}{n} \widetilde{a} \tau \widetilde{a} + \frac{2 i \pi}{n} \widetilde{a}{}^t (z+\widetilde{b}) \right)  \vartheta \left( z+\frac{\widetilde{a}}{n} \tau + \frac{\widetilde{b}}{n} \right),
\]
where $\widetilde{a},\widetilde{b}$ are lifts of $a,b$ in $\Z^g$ (the result does not depend on this choice by \eqref{eqsectionsLtaupuissancen}). Therefore, by $\psi_\beta$ and the theta functions with characteristic (formula \eqref{eqthetaabenfonctiontheta}), the Mumford coordinates of $(A,L,\alpha_{2n})$ (with the induced theta structure $\beta$ on $L^{\otimes n})$ are \textit{exactly} the projective coordinates 
\[
	\left( \Theta_{\widetilde{a}/n,\widetilde{b}/n (\tau)}^n(\tau) \right)_{a,b \in \frac{1}{n} \Z^{2g} / \Z^{2g}} \in \P^{n^{2g-1}} (\C),
\]
where the choices of lifts $\widetilde{a}$ and $\widetilde{b}$ for $a$ and $b$ still do not matter.

In particular, for every $\tau \in \Hcal_2$, the point $\psi(\tau)$ can be intrinsically given as the squares of Mumford coordinates for $\beta_\tau$, where the six odd characteristics (whose coordinates vanish everywhere) are taken out. The result only depends on the isomorphism class of $(A_\tau,L_\tau,\alpha_{\tau,2})$, as expected.

Finally, as demonstrated in the paragraph 6 of \cite{Mumford67} (especially the Theorem p. 83), the theory of theta structures (and the associated Mumford coordinates) can be extended to abelian schemes (Definition \ref{defabelianscheme}) (still outside characteristics dividing $2n$), and the Mumford coordinates in this context lead to an embedding of the associated moduli space in a projective space as long as the \textit{type} of the sheaf is a multiple of 8 (which for us amounts to $8|n$). Here, fixing a principally polarised abelian variety $A$ over a number field $K$ and $\gP$ a prime ideal of $\Ocal_K$ not above 2, this theory means thats given a symmetric theta structure on $(A,L)$ for $L^{\otimes n}$ where $8|n$, if $A$ has good reduction modulo $\gP$, this theta structure has a natural reduction to a theta structure on the reduction $(A_{\gP}, L_{\gP})$ for $L_\gP^{\otimes n}$, and this reduction is compatible with the reduction of Mumford coordinates modulo $\gP$. To link this with the reduction of coordinates of $\psi$, one just has to extend the number field $K$ of definition of $A$ so that all 8-torsion points of $A$ are defined over $K$ (in particular, the reduction of $A$ modulo $\gP$ is semistable), and consider a symmetric theta structure on $L^{\otimes 8}$. The associated Mumford coordinates then reduce modulo $\gP$, but their vanishing is linked to the belonging of $8$-th torsion points to $\Theta_{A_\gP,L_\gP}$ by \eqref{eqliencoordonneesMumfordetdivtheta}. The number of vanishing coordinates is then entirely determined in Propositions \ref{propdivthetaproduitCE} and \ref{propnombrepointsdivthetajacobienne}, which proves Proposition \ref{propalgthetafoncetreduchorsde2} (not forgetting the six ever-implicit odd characteristics).
\end{proof}

\subsection{Evaluating the theta functions at archimedean places}
\label{subsecarchimedeanplaces}

We denote by $\Hcal_2$ the Siegel half-space of degree 2, and by $\Fcal_2$ the usual fundamental domain of this half-space for the action of $\Sp_4(\Z)$ (see \cite{Klingen}, section I.2 for details). For $\tau \in \Hcal_2$, we denote by $y_4$ the imaginary part of the lower-right coefficient of $\tau$.

\begin{prop}
\label{proparchimedeanbound}
	For every $\tau \in \Hcal_2$ and a fixed real parameter $t \geq \sqrt{3}/2$, one has : 
	
	$(a)$ Amongst the ten even characteristics $m$ of $E$, at most six of them can satisfy 
	\[
	|\Theta_{m/2} (\tau)| < 0.42 \max_{m' \in E} |\Theta_{m'/2} (\tau)|.
	\]
	
	$(b)$ If the representative of the orbit of $\tau$ in the fundamental domain $\Fcal_2$ satisfies $y_4 \leq t$, at most one of the ten even characteristics $m$ of $E$ can satisfy
	\[
		|\Theta_{m/2} (\tau)| < 1.22 e^{- \pi t} \max_{m' \in E} |\Theta_{m'/2} (\tau)|.
	\]
\end{prop}

\begin{proof}
	First, we can assume that $\tau \in \Fcal_2$ as the inequalities $(a)$ and $(b)$ are invariant by the action of $\Sp_4(\Z)$, given the complete transformation formula of these theta functions (\cite{MumfordTata}, section II.5). Now, using the Fourier expansions of the ten theta constants (mentioned in the proof of Definition-Proposition \ref{defipropdivthetaA2ncomplexes}) and isolating their respective dominant terms (such as in \cite{Klingen}, proof of Proposition IV.2), we obtain explicit estimates. More precisely, Proposition 7.7 of \cite{Strengthesis} states that, for every $\tau  = \begin{pmatrix} \tau_1 & \tau_2 \\ \tau_2 & \tau_4 \end{pmatrix} \in \Bcal_2$ (which is a domain containing $\Fcal_2$), one has 
	\begin{eqnarray*}
		\left| \Theta_{m/2}(\tau) - 1 \right| & < & 0.405, \quad {\scriptstyle m \in \{(0000)(0001),(0010),(0011) \}}. \\
		\left| \frac{ \Theta_{m/2}(\tau)}{2 e^{ i \pi \tau_1/2}} - 1 \right| & < & 0.348, \quad {\scriptstyle m \in \{(0100),(0110) \}}. \\
		\left| \frac{ \Theta_{m/2}(\tau)}{2 e^{ i \pi \tau_4/2}} - 1 \right| & < & 0.348, \quad {\scriptstyle m \in \{(1000),(1001) \}}. \\
		\left| \frac{ \Theta_{m/2}(\tau)}{(\varepsilon_m + e^{2 i \pi \tau_2})e^{ i \pi (\tau_1 + \tau_4 - 2 \tau_3)/2}} - 1 \right| & < & 0.438,\quad {\scriptstyle m \in \{(1100),(1111) \}},
	\end{eqnarray*}
with $\varepsilon_m = 1$ if $m = (1100)$ and $-1$ if $m=(1111)$.

Under the assumption that $y_4 \leq t$ (which induces the same bound for $\Im \tau_1$ and $2 \Im \tau_2$), we obtain 
\[
\begin{array}{rcccl}
	0.595 & < & \left|  \Theta_{m/2}(\tau) \right| & < & 1.405, \quad {\scriptstyle m \in \{(0000)(0001),(0010),(0011) \}}. \\
	1.304 e^{ - \pi t/2} & < & \left|  \Theta_{m/2}(\tau) \right| & < & 0.692, \quad {\scriptstyle m \in \{(0100),(0110),(1000),(1001)\}}. \\
	1.05 e^{ - \pi t} & < & \left| \Theta_{m/2}(\tau) \right| & < &  0.855, \quad {\scriptstyle m = (1100)}. \\
	& & \left|  \Theta_{m/2}(\tau) \right| & < & 0.855, \quad {\scriptstyle m = (1111)}
\end{array}
\]
Thus, we get $(a)$ with $0.595/1.405 > 0.42$, and $(b)$ with $1.05 e^{- \pi t}/ 0.855> 1.22 e^{- \pi t}$.
\end{proof}

\subsection{Computations with Igusa invariants for the case places above 2}
\label{subsecplacesabove2}

In this case, as emphasized before, it is not possible to use Proposition \ref{propalgthetafoncetreduchorsde2}, as the algebraic theory of theta functions does not work.

We have substituted it in the following way. 
\begin{defi}[Auxiliary polynomials]
\hspace*{\fill}

	For every $i \in \{1, \cdots, 10\}$, let $\Sigma_i$ be the $i$-th symmetric polynomial in the ten modular forms $\Theta_{m/2}^8$, $m \in E$ (notation \eqref{eqevenchartheta}).
	This is a modular form of level $4i$ for the whole modular group $\Sp_4(\Z)$.
\end{defi}

Indeed, each $\Theta_{m/2}^8$ is a modular form for the congruence subgroup $\Gamma_2(2)$ of weight 4, and they are permuted by the modular action of  $\Gamma_2(1)$ (\cite{MumfordTata}, section II.5). The important point is that the $\Sigma_i$ are then polynomials in the four Igusa modular forms $\psi_4,\psi_6,\chi_{10}$ and $\chi_{12}$ (\cite{Igusa67bis}, p.848 and 849). We can now explain the principle of this paragraph : these four modular forms are linked explicitly with the Igusa invariants (for a given jacobian of an hyperelliptic curve $C$ over a number field $K$), and the semi-stable reduction of the jacobian at some place $v|2$ is determined by the integrality (or not) of some quotients of these invariants, hence rational fractions of the modular forms. Now, with the explicit expressions of the $\Sigma_i$ in terms of $\psi_4,\psi_6,\chi_{10}$ and $\chi_{12}$, we can bound these $\Sigma_i$ by one of the Igusa invariants, and as every $\Theta_{m/2}^8$ is a root of the polynomial
\[
	P(X) = X^{10} - \Sigma_1 X^9 + \Sigma_2 X^8 - \Sigma_3 X^7 + \Sigma_4 X^6 - \Sigma_5 X^5 + \Sigma_6 X^4 - \Sigma_7 X^4 + \Sigma_8 X^2 - \Sigma_9 X + \Sigma_{10},
\]
we can infer an explicit bound above on the $\Theta_{m/2}^8/\lambda$, with a well-chosen normalising factor $\lambda$ such that these quotients belong to $K$. Actually, we will even give an approximative shape of the Newton polygon of the polynomial $\lambda^{10} P(X/\lambda)$, implying that its slopes (except maybe the first one) are bounded above and below, thus giving us a minoration of each of the $|\Theta_{m/2}|_v/\max_{m' \in E} |\Theta_{m'/2}|_v$, except maybe for one $m$. The explicit result is the following.

\begin{prop}
\label{propbornesfoncthetaaudessus2}
	Let $K$ be a number field, $(A,L)$ a principally polarised jacobian of dimension 2 over $K$ and $\tau \in \Hcal_2$ such that $(A_\tau,L_\tau) \cong (A,L)$. 
	
	Let $\gP$ be a prime ideal of $K$ above $2$ such that $A$ has potentially good reduction at $\gP$, and the reduced (principally polarised abelian surface) is denoted by $(A_\gP,L_\gP)$.  By abuse of notation, we forget the normalising factor ensuring that the coordinates $\Theta_{m/2} (\tau)^8$ belong to $K$.
	
	$(a)$ If $(A_\gP,L_\gP)$ is the jacobian of a smooth hyperlliptic curve, all the $m \in E$ satisfy 
	\[
		 \frac{\left| \Theta_{m/2} (\tau) ^8\right|_\gP}{\max_{m' \in E} \left| \Theta_{m'/2} (\tau)^8 \right|_\gP } \geq |2|_\gP^{12}. 
	\]
	
	$(b)$ If $(A_\gP,L_\gP)$ is a product of elliptic curves, all the $m \in E$ except at most one satisfy
	\[
		 \frac{\left| \Theta_{m/2} (\tau)^8 \right|_\gP}{\max_{m' \in E} \left| \Theta_{m'/2} (\tau)^8\right|_\gP } \geq |2|_\gP^{21}. 
	\]

\end{prop}

\begin{proof}
The most technical part is computing the $\Sigma_i$ as polynomials in the four Igusa modular forms. To do this, we worked with Sage in the formal algebra generated by some sums of $\Theta_{m/2}^4$ with explicit relations (namely, $y_0, \cdots, y_4$ in the notations of \cite{Igusa64bis}, p.396 and 397). Taking away some timeouts probably due to the computer's hibernate mode, the total computation time on a portable PC has been about twelve-hours-long (including verification of the results). The detail of algorithms and construction is available on a Sage worksheet \footnote{This worksheet can be downloaded at \url{http://perso.ens-lyon.fr/samuel.le_fourn/contenu/fichiers_publis/Igusainvariants.ipynb}} (in Jupyter format). An approach based on Fourier expansions might be more efficient, but as there is no clear closed formula for the involved modular forms, we privileged computations in this formal algebra. For easier reading, we slightly modified the Igusa modular forms into $h_4,h_6,h_{10},h_{12}$ defined as 
\begin{equation}
\label{eqdefmodifiedIgusamodularforms}
\left\{
\begin{array}{rcccl}
	h_4 & = & 2 \cdot  \psi_4  & = & {\displaystyle \frac{1}{2} \sum_{m \in E} \Theta_{m/2}^8}\\
	h_6 & = & 2^2 \cdot \psi_6 & = & {\displaystyle \sum_{\scriptscriptstyle \substack{\{m_1,m_2,m_3\} \subset E \\ \textrm{syzygous}}} \pm (\Theta_{m_1/2} \Theta_{m_2/2} \Theta_{m_3/2})^4}  \\
	h_{10}  & = & 2^{15} \cdot \chi_{10} & = & {\displaystyle 2 \prod_{m \in E} \Theta_{m/2}^2} \\
	h_{12} & = &  2^{16} \cdot 3 \cdot  \chi_{12} & = & {\displaystyle \frac{1}{2} \sum_{\scriptscriptstyle \substack{C \subset E\\ C \textrm{ Göpel}}} \prod_{m \in E \backslash C} \Theta_{m/2}^4 }
\end{array}
\right.
\end{equation}
(\cite{Igusa67bis}, p.848 for details on these definitions, notably syzygous triples and Göpel quadruples). The third expression is not explicitly a polynomial in $y_0, \cdots, y_4$, but there is such an expression, given p.397 of \cite{Igusa64bis}. We also used to great benefit (both for understanding and computations) the section I.7.1 of \cite{Strengthesis}.

Now, the computations on Sage gave us the following formulas (the first and last one being trivial given \eqref{eqdefmodifiedIgusamodularforms}, they were not computed by the algorithm)

\begin{align}
	\Sigma_1 & =  2 h_4  \label{eqSigma1} \\
	\Sigma_2 & =  \frac{3}{2} h_4^2 \label{eqSigma2} \\
	\Sigma_3 & =  \frac{29}{2 \cdot 3^3} h_4^3 - \frac{1}{2 \cdot 3^3} h_6^2 + \frac{1}{2 \cdot 3} h_{12} \label{eqSigma3}\\
	\Sigma_4 & =  \frac{43}{2^4 \cdot 3^3} h_4^4 - \frac{1}{2 \cdot 3^3} h_4 h_6^2 + \frac{23}{2 \cdot 3} h_4 h_{12} + \frac{2}{3} h_6 h_{10} \label{eqSigma4} \\
	\Sigma_5 & =  \frac{1}{2^2 \cdot 3^3} h_4^5 - \frac{1}{2^3 \cdot 3^3} h_4^2 h_6^2 + \frac{25}{2^3 \cdot 3} h_4^2 h_{12} - \frac{1}{2 \cdot 3} h_4 h_6 h_{10} + \frac{123}{2^2} h_{10}^2 \label{eqSigma5}\\
	\Sigma_6 & =  \frac{1}{2^2 \cdot 3^6} h_4^6 - \frac{1}{2^2 \cdot 3^6} h_4^3 h_6^2 + \frac{7}{2 \cdot 3^3} h_4^3 h_{12} - \frac{1}{2^2 \cdot 3} h_4^2 h_6 h_{10} \label{eqSigma6} \\
	& +  \frac{47}{2 \cdot 3} h_4 h_{10}^2 + \frac{1}{2^4 \cdot 3^6} h_6^4 - \frac{5}{2^3 \cdot 3^3} h_6^2 h_{12} + \frac{43}{2^4 \cdot 3} h_{12}^2  \nonumber \\
	\vspace{1cm} \\
	\Sigma_7 & =  \frac{1}{2 \cdot 3^4} h_4^2 h_{12} - \frac{1}{2 \cdot 3^4} h_4^3 h_6 h_{10} + \frac{41}{2^3 3^2} h_4^2 h_{10}^2 - \frac{1}{2^2 \cdot 3^4} h_4 h_6^2 h_{12} \label{eqSigma7} \\ & +  \frac{11}{2^2 \cdot 3^2} h_4 h_{12}^2 + \frac{1}{2^2 \cdot 3^4} h_6^3 h_{10} - \frac{19}{2^2 \cdot 3^2} h_6 h_{10} h_{12}  \nonumber\\
	\Sigma_8 & =  \frac{1}{2^2 \cdot 3^3} h_4^3 h_{10}^2 + \frac{1}{2^2 \cdot 3^2} h_4^2 h_{12}^2 - \frac{1}{2 \cdot 3^2} h_4 h_6 h_{10} h_{12} + \frac{5}{2^3 \cdot 3^3} h_6^2 h_{10}^2 - \frac{11}{2^3} h_{10}^2 h_{12} \label{eqSigma8} \\
	\Sigma_9 & =  \frac{-5}{2^2 \cdot 3^2} h_4 h_{10}^2 h_{12} + \frac{7}{2^2 \cdot 3^3} h_6 h_{10}^3 + \frac{1}{3^3} h_{12}^3 \label{eqSigma9} \\
	\Sigma_{10} & =  \frac{1}{2^4} h_{10}^4. \label{eqSigma10}
\end{align}

\begin{rem}
The denominators are always products of powers of 2 and 3. This was predicted by \cite{Ichikawa09}, as all Fourier expansions of $\Theta_{m/2}$ (therefore of the $\Sigma_i$) have integral coefficients. Surprisingly, the result of \cite{Ichikawa09} would actually be false for a $\Z[1/3]$-algebra instead of a $\Z[1/6]$-algebra, as the expression of $\Sigma_3$ (converted as a polynomial in $\psi_4,\psi_6, \chi_{12}$) shows, but this does not provide a counterexample for a $\Z[1/2]$-algebra.
\end{rem}

Now, let $C$ be an hyperelliptic curve of genus 2 on a number field $K$ and $\gP$ a prime ideal of $\Ocal_K$ above 2. We will denote by $|\cdot|$ the norm associated to $\gP$ to lighten the notation. Let $A$ be the jacobian of $C$ and $J_2,J_4,J_6,J_8,J_{10}$ the homogeneous Igusa invariants of the curve $C$, defined as in (\cite{Igusa60}, pp. 621-622) up to a choice of hyperelliptic equation for $C$. We fix $\tau \in \Hcal_2$ such that $A_\tau$ is isomorphic to $A$, which will be implicit in the following (i.e. $h_4$ denotes $h_4(\tau)$ for example). By (\cite{Igusa67bis}, p.848) applied with our normalisation, there is an hyperelliptic equation for $C$ (and we fix it) such that 
\begin{align}
 J_2 & =  \frac{1}{2}  \frac{h_{12}}{h_{10}}  \\
J_4 & =  \frac{1}{2^5 \cdot 3} \left( \frac{h_{12}^2}{h_{10}^2} - 2 h_4 \right) \\
J_6 & =  \frac{1}{2^7 \cdot 3^3} \left( \frac{h_{12}^3}{h_{10}^3} - 6 \frac{h_4 h_{12}}{h_{10}} + 4 h_6 \right) \\
J_8 & =  \frac{1}{2^{12} \cdot 3^3} \left( \frac{h_{12}^4}{h_{10}^4} - 12 \frac{h_4 h_{12}^2}{h_{10}^2} + 16 \frac{h_6 h_{12}}{h_{10}} - 12 h_4^2 \right) \\
J_{10} & =  \frac{1}{2^{13}}  h_{10}.
\end{align}
Let us now figure out the Newton polygons allowing us to bound our theta constants.

$(a)$ If $A$ has potentially good reduction at $\gP$, and this reduction is also a jacobian, by Proposition 3 of \cite{Igusa60}, the quotients $J_2^5/J_{10}, J_4^5 /J_{10}^2, J_6^5 / J_{10}^3$ and $J_{8}^5 / J_{10}^4$ are all integral at $\gP$. Translating it into quotients of modular forms, this gives 
\begin{eqnarray*} 
\left| \frac{J_2^5}{J_{10}} \right|  & = & |2|^8 \left| \frac{h_{12}^5}{h_{10}^6} \right|  \leq 1 \\ 
\left|\frac{J_4^5}{J_{10}^2} \right| & = & |2|^3 \left| \frac{h_{12}^2}{h_{10}^{12/5}} - 2 \frac{h_4}{h_{10}^{2/5}}\right|^{5} \leq 1 \\
\left| \frac{J_6^5}{J_{10}^3} \right| & = & |2|^4 \left| \frac{h_{12}^3} {h_{10}^{18/5}} - 6 \frac{h_4 h_{12}}{h_{10}^{8/5}} + 4 \frac{h_6}{h_{10}^{3/5}} \right|^5 \leq 1 \\
\left| \frac{J_8^5}{J_{10}^4} \right| & = & |2|^{-8} \left| \frac{h_{12}^4}{h_{10}^{24/5}} - 12 \frac{h_4 h_{12}^2}{h_{10}^{14/5}} + 16 \frac{h_6 h_{12}}{h_{10}^{9/5}} - 12 \frac{h_4^2}{h_{10}^{4/5}} \right|^5 \leq 1.
\end{eqnarray*}

By successive bounds on the three first lines, we obtain 
\[
\left| \frac{h_4}{h_{10}^{2/5}} \right|  \leq  |2|^{-21/5},  \quad \left| \frac{h_6}{h_{10}^{3/5}} \right|  \leq  |2|^{-34/5}, \quad 
\left| \frac{h_{12}}{h_{10}^{6/5}} \right|  \leq  |2|^{-8/5}.
\]
Using the expressions of the $\Sigma_i$ (\eqref{eqSigma1} to \eqref{eqSigma10}), we compute that for every $i \in \{1, \cdots, 10\}$, one has $\left| \Sigma_i / h_{10}^{2 i /5} \right|  \leq  |2|^{\lambda_i}$ with the following values of $\lambda_i$ :  
\[
	\begin{array}{c|cccccccccc}
	\hline
	i & 10 & 9 & 8 & 7 & 6 & 5 & 4 & 3 & 2 & 1 \\
	\lambda_i & - \frac{20}{5} & - \frac{44}{5}& - \frac{83}{5}& - \frac{112}{5}& - \frac{156}{5}& - \frac{125}{5}& - \frac{104}{5}& - \frac{73}{5}& - \frac{47}{5}& - \frac{16}{5}  \\
	\hline
	\end{array}
\]
and for $i=10$, it is an equality. Therefore, the highest slope of the Newton polygon is at most $26/5 \cdot v_\gP(2)$, whereas the lowest one is at least $-34/5 \cdot v_\gP(2)$, which gives part $(a)$ of Proposition \ref{propbornesfoncthetaaudessus2} by the theory of Newton polygons.

$(b)$ If $A$ has potentially good reduction at $\gP$ and the semistable reduction is a product of elliptic curves, defining 
\begin{eqnarray}
I_4 & = & J_2^3 - 25 J_4 = \frac{h_4}{2} \label{eqdefI4} \\
I_{12} & = & - 8 J_4^3 + 9 J_2 J_4J_6 - 27 J_6^2 - J_2^2 J_8 = \frac{1}{2^{10} \cdot 3^3} (2 h_4^3 - h_6^2), \label{eqdefI12} \\
P_{48} & = & 2^{12} \cdot 3^3 h_{10}^4 J_8 =  h_{12}^4 - 12 h_4  h_{12}^2 h_{10}^2 + 16 h_6 h_{12} h_{10}^3 - 12 h_4^2 h_{10}^4 \label{eqdefP48}
\end{eqnarray}
(which as modular forms are of respective weights $4,12$ and $48$), by Theorem 1 (parts $(V_*)$ and $(V)$) of \cite{Liu93}, we obtain in the same fashion that
\begin{equation}
\label{eqbornesenfoncP481}
\left| \frac{h_4}{P_{48}^{1/12}} \right| \leq |2|^{-13/3}, \left| \frac{h_6}{P_{48}^{1/8}}\right| \leq |2|^{-3}, \quad \left| \frac{h_{10}}{P_{48}^{5/24}}\right| \leq |2|^{-4/3}.
\end{equation}
Using the Newton polygon for the polynomial of \eqref{eqdefP48} defining $P_{48}$, one deduces quickly that 
\begin{equation}
\label{eqbornesenfoncP482}
\left| \frac{h_{12}}{P_{48}^{1/4}} \right| \leq |2|^{-7/2}.
\end{equation}
As before, with the explicit expression of the $\Sigma_i$, one obtains that the $|\Sigma_i / P_{48}^{i/12}|$ are bounded by $|2|^{\lambda_i}$ with the following values of $\lambda$ : 
\begin{equation}
\label{eqvalSigmairedproduitCE1}
\begin{array}{c|cccccccccc}
\hline
i & 10 & 9 & 8 & 7 & 6 & 5 & 4 & 3 & 2 & 1  \\
\lambda_i  & -\frac{28}{3} & -\frac{71}{6} &  \frac{-53}{3} & \frac{-55}{3} & \frac{-84}{3} & \frac{-71}{3} & \frac{-64}{3} & -14 & \frac{-29}{3} & \frac{-10}{3} \\
\hline
\end{array}
\end{equation}
This implies directly that the highest slope of the Newton polygon is at most $16/3 \cdot v_\gP(2)$. Now, for the lowest slope, there is no immediate bound and it was expected : in this situation, $\Sigma_{10} = 2^{-4} h_{10}^4$ can be relatively very small compared to $P_{48}^{5/6}$. 

As $P_{48}$ is in the ideal generated by $h_{10},h_{12}$ (in other words, is cuspidal) and dominates all modular forms $h_4,h_6,h_{10},h_{12}$, one of $h_{10}$ and $h_{12}$ has to be relatively large enough compared to $P_{48}$ . In practice, we get (with \eqref{eqbornesenfoncP481}, \eqref{eqbornesenfoncP482} and \eqref{eqdefP48})
\[
\left| \frac{h_{12}}{P_{48}^{1/4}} \right| \geq 1 \quad \textrm{or} \quad \left| \frac{h_{10}}{P_{48}^{5/24}} \right| \geq |2|^{13/6}.
\]
Now, if $h_{10}$ is relatively very small (for example, $\left|h_{10}/P_{48}^{5/24} \right| \leq |2|^{19/6} \left|h_{12}/P_{48}^{1/4} \right|$), we immediately get $\left| h_{12}/P_{48}^{1/4} \right| = 1$ and $\left| \Sigma_9 / P_{48}^{3/4} \right| = 1$. Computing again with these estimates for $h_{10}$ and $h_{12}$, we obtain that the $\left|\Sigma_i / P_{48}^{i/12} \right|$ are bounded by $|2|^{\lambda_i}$ with the following slightly improved values of $\lambda$, 
\[
\begin{array}{c|ccccccccc}
\hline
i & 9 & 8 & 7 & 6 & 5 & 4 & 3 & 2 & 1  \\
\lambda_i  & 0 & -\frac{32}{3} & -\frac{51}{3} &  \frac{-84}{3} & \frac{-71}{3} & \frac{-64}{3} & -14 & \frac{-29}{3} & \frac{-10}{3}  \\
\hline
\end{array}	
\]
The value at $i=9$ is exact, hence the second lowest slope is then at least $-\frac{32}{3} \cdot v_\gP(2)$.

If it is not so small, we have a bound on $v_\gP(\Sigma_{10}/P_{48}^{6/5})$, hence the Newton polygon itself is bounded (and looks like in the first situation). In practice, one finds that the lowest slope is at least $-47/3 \cdot v_\gP(2)$, hence all others slopes are at least this value, and this concludes the proof of Proposition \ref{propbornesfoncthetaaudessus2} $(b)$. 
\end{proof}

\begin{rem}
In characteristics $\neq 2,3$, Theorem 1 of \cite{Liu93} and its precise computations pp. 4 and 5 give the following exact shapes of Newton polygons (notice the different normalisation factors).
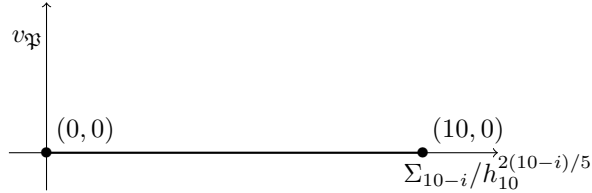
\begin{figure}[H]
	\centering

\begin{tikzpicture}[scale=0.5]
	\draw[->, thin] (-1,0) -- (12,0);
	\draw (-1/2,3) node {$v_\gP$};
 	\draw (12,-1/2) node {$\Sigma_{10-i}/h_{10}^{2(10-i)/5}$};
 	\draw[->, thin] (0,-1) -- (0,4); 
 	\draw (0,0) node {$\bullet$};
 	\draw (0,0) node[above right] {$(0,0)$};
 	\draw (10,0) node {$\bullet$};
 	\draw (10,0) node[above right] {$(10,0)$};
 	\draw[thick] (0,0) -- (10,0);
\end{tikzpicture}
\caption{When the reduction of $A$ is a jacobian}
\end{figure}

\begin{figure}[H]
\centering
\begin{tikzpicture}[scale=0.5]
	\draw[->, thin] (-1,0) -- (11,0);
	\draw (11,-1/2) node {$\Sigma_{10-i}/h_{12}^{(10-i)/3}$};
	\draw[->, thin] (0,-1) -- (0,5); 
	\draw (-1/2,4) node {$v_\gP$};
	\draw (0,3) node {$\bullet$};
	\draw (1,0) node {$\bullet$};
	\draw (10,0) node {$\bullet$};
	\draw[thick] (0,3) -- (1,0);
	\draw[thick] (1,0) -- (10,0);
\end{tikzpicture}
\caption{When the reduction of $A$ is a product of elliptic curves}
\end{figure}
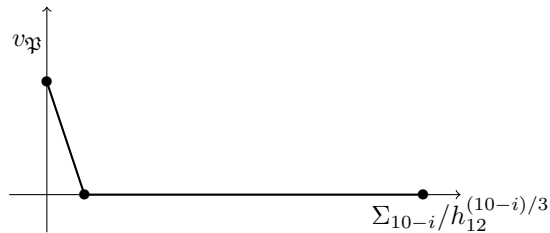
In particular, when $A$ reduces to a jacobian, the theta coordinates all have the same $\gP$-adic norm and when $A$ reduces to a product of elliptic curves, exactly one of them has smaller norm : in other words, we reproved Proposition \ref{propalgthetafoncetreduchorsde2}, and the Newton polygons have a very characteristic shape.

The idea behind the computations above is that in cases $(a)$ and $(b)$ (with other normalisation factors), the Newton polygons have a shape close to these ones, therefore estimates can be made. It would be interesting to see what the exact shape of the Newton polygons is, to maybe obtain sharper results.
\end{rem}

\subsection{Wrapping up the estimates and end of the proof}
\label{subsecfinalresultRungeCEexplicite}
We can now prove the explicit refined version of Theorem \ref{thmtubularRungeproduitCE}, namely Theorem \ref{thmproduitCEexplicite}.

\begin{proof}[Proof of Theorem \ref{thmproduitCEexplicite}]
	
In case $(a)$, one can avoid the tubular assumption for the (unique) archimedean place of $K$: indeed, amongst the ten theta coordinates, there remain 4 which are large enough with no further assumption. As $|s_P|<4$, there remains one theta coordinates which is never too small (at any place). In practice, normalising the projective point $\psi(P)$ by this coordinate, one obtains with Propositions \ref{proparchimedeanbound} $(a)$ (archimedean place), \ref{propalgthetafoncetreduchorsde2} (finite places not above 2) and  \ref{propbornesfoncthetaaudessus2} (finite places above 2)
\[
	h(\psi(P)) \leq - 4 \log(0.42) + \frac{1}{[K:\Q]} \sum_{v|2} n_v |2|^{21/2} \leq 10.75 
\]
after approximation.

In case $(b)$, one has to use the tubular neighbourhood implicitly given by the parameter $t$, namely Proposition \ref{proparchimedeanbound} $(b)$ for archimedean places, again with Propositions \ref{propalgthetafoncetreduchorsde2} and \ref{propbornesfoncthetaaudessus2} for the finite places, hence we get 
\[
	h(\psi(P)) \leq 4 \log(e^{\pi t}/1.33) + \frac{1}{[K:\Q]} \sum_{v|2} n_v |2|^{21/2} \leq 4 \pi t + 6.14
\]
after approximation.

Finally, we deduce from there the bounds on the stable Faltings height by Corollary 2.2 of \cite{Pazuki12b} (with its notations, $h_\Theta(A,L) = h(\psi(P))/4$).
\end{proof}

It would be interesting to give an analogous result for Theorem \ref{thmtubularRungegeneral}, and the estimates for archimedean and finite places not above 2 should not give any particular problem. For finite places above 2, the method outlined above can only be applied if, taking the symmetric polynomials $\Sigma_1, \cdots, \Sigma_{f(n)}$ in well-chosen powers $\Theta_{\widetilde{a}/n,\widetilde{b/n}} (\tau)$ for $\widetilde{a},\widetilde{b} \in \Z^g$, we can figure out by other arguments the largest rank $k_0$ for which $\Sigma_{k_0}$ is cuspidal but not in the ideal generated by $h_{10}$. Doing so, we could roughly get back the pictured shape of the Newton polygon when $h_{10}$ is relatively very small (because then $\Sigma_k$ is relatively very small for $k>k_0$ by construction). Notice that for this process, one needs some way to theoretically bound the denominators appearing in the expressions of the $\Sigma_i$ in $h_4,h_6,h_{10},h_{12}$, but if this works, the method can again be applied.

\bibliographystyle{alphaSLF}
\bibliography{bibliotdn}

\end{document}